\documentclass{amsart}	
\usepackage{amsmath,amssymb,amsthm,amscd}
\usepackage{graphicx,subfigure}
\usepackage{color}
\usepackage{pb-diagram} 
\usepackage{tabularx}
\usepackage{multirow}
\usepackage{longtable}
\usepackage{hyperref}
\usepackage[foot]{amsaddr}
\usepackage{enumitem}
\usepackage{comment}
\usepackage{phonetic}

\usepackage{geometry}
\newcommand{\blue}[1]{\bgroup\color{blue}{#1}\egroup}
\newcommand{\red}[1]{\bgroup\color{red}{#1}\egroup}
\newcommand{\magenta}[1]{\bgroup\color{magenta}{#1}\egroup}

\newtheorem{theorem}{Theorem}[section]
\newtheorem{proposition}[theorem]{Proposition}
\newtheorem{corollary}[theorem]{Corollary}

\newtheorem{lemma}[theorem]{Lemma}

\theoremstyle{remark} 
\newtheorem{remark}[theorem]{Remark}

\numberwithin{equation}{section}

\newcommand{\ep}{\varepsilon}

\renewcommand\epsilon{\varepsilon}

\newcommand{\cA}{{\mathcal A}}

\newcommand{\cD}{{\mathcal D}}

\newcommand{\cK}{{\mathcal{K}}}
\newcommand{\cL}{{\mathcal L}}

\newcommand{\cN}{{\mathcal N}}

\newcommand{\cX}{{\mathcal X}}
\newcommand{\cV}{{\mathcal V}}

\newcommand{\RR}{{\mathbb R}}
\newcommand{\CC}{{\mathbb C}}

\newcommand{\TT}{{\mathbb T}}
\newcommand{\ZZ}{{\mathbb Z}}

\newcommand{\abs}[1]{|{#1}|}
\newcommand{\Abs}[1]{\left|{#1}\right|}
\newcommand{\norm}[1]{\left\|{#1}\right\|}
\newcommand{\Norm}[1]{\left\|{#1}\right\|}
\newcommand{\aver}[1]{\langle{#1}\rangle}

\newcommand\omegap{{\omega_p}}

\newcommand\wnorm[3]{{\Abs{\norm{#1^L}_{#2}, \frac{1}{\gamma #3^\tau}  \norm{#1^N}_{#2}}}}
\newcommand\wnormeta{{\wnorm{\eta}{\rho}{\delta}}}

\newcommand\qwnormeta{{{\wnormeta}^2}}

\newcommand\Xh{{X_h}}

\newcommand\Xp{{X_p}}

\newcommand{\Xoper}[1]{\cX_{#1}}

\newcommand{\Dioph}[3]{\cD_{#1,#2}^{#3}}

\newcommand{\CR}{c_{\mathfrak{R}}}
\newcommand{\CRb}{{\hat c}_{\mathfrak{R}}}
\newcommand{\CRd}{c_{\mathfrak{R}}^1}
\newcommand{\CRdb}{{\hat c}_{\mathfrak{R}}^1}

\newcommand\Loper{{\mathfrak{L}_{\omega}}}
\newcommand\Roper{{\mathfrak{R}_{\omega}}}
 
\newcommand{\comp}{{\!\:\circ\!\,}}

\newcommand\K{{K}}
\newcommand\Ki{{\bar K}}
\newcommand\Kn{{\bar{\bar K}}}

\newcommand\E{{E}}
\newcommand\Ei{{\bar E}}
\newcommand\En{{\bar{\bar E}}}

\newcommand\ttop{{\!\top}}

\newcommand\DK{{\Dif\K}}
\newcommand\DKi{{\Dif\Ki}}
\newcommand\DKn{{\Dif\Kn}}
\newcommand\DKT{({\Dif\K})^{\ttop}}
\newcommand\DKiT{({\Dif\Ki})^{\ttop}}
\newcommand\DKnT{({\Dif\Kn})^{\ttop}}

\renewcommand\L{{L}}
\newcommand\LT{{\L}^\ttop}
\newcommand\Li{{\bar{L}}}
\newcommand\LiT{{\Li}\phantom)\!\!^\ttop}
\newcommand\Ln{{\bar {\bar{\L}}}}
\newcommand\LnT{{\Ln}\phantom)\!\!^\ttop}

\newcommand\N{{N}}
\newcommand\NT{{\N}^\ttop}

\newcommand\Nn{{\bar {\bar{\N}}}}
\newcommand\NnT{{\bar{\bar\N}}\phantom)\!\!^{\ttop}}

\newcommand\OK{{\Omega\comp\K}}

\newcommand\GL{{G_{\L}}}

\newcommand\GLn{{G_{\Ln}}}
\newcommand\invGL{{G_{\L}^{\text{-}1}}}

\newcommand\B{{B}}

\newcommand\Bn{{\bar{\bar B}}}

\newcommand\T{{T}}

\newcommand\Tn{{\bar{\bar T}}}

\newcommand\OmegaDK{{\Omega_{\text{\tiny ${\rm D}K$}}}}
\newcommand\OmegaL{{\Omega_{L}}}
\newcommand\OmegaN{{\Omega_{N}}}

\newcommand\etaL{{\eta^L}}
\newcommand\etaN{{\eta^N}}
\newcommand\etaNDK{{\eta^N_{\text{\tiny $\DK$}}}}
\newcommand\etaNXp{{\eta^N_{\text{\tiny $X_p$}}}}

\newcommand\averxiL{{\hat\xi^L_0}}
\newcommand\averxiN{{\hat\xi^N_0}}
\newcommand\xiL{\xi^L}
\newcommand\xiN{\xi^N}
\newcommand\xiLDK{{\xiL_{\text{\tiny $\DK$}}}}
\newcommand\xiLXp{{\xiL_{\text{\tiny $X_p$}}}}

\newcommand\etai{{\bar\eta}}

\newcommand\etaiN{{\etai^N}}

\newcommand\etan{{\bar {\bar \eta}}}
\newcommand\etanL{{{\etan}^L}}
\newcommand\etanN{{\etan^N}}

\newcommand\etanLu{{{\etan}^L_1}}
\newcommand\etanLd{{{\etan}^L_2}}
\newcommand\etanLt{{{\etan}^L_3}}
\newcommand\etanLq{{{\etan}^L_4}}
\newcommand\etanLc{{{\etan}^L_5}}

\newcommand\etaLs[1]{{\eta}_{#1}^L}
\newcommand\etaNs[1]{{\eta}_{#1}^N}
\newcommand\etanLs[1]{{\bar {\bar \eta}}_{#1}^L}
\newcommand\etanNs[1]{{\bar {\bar \eta}}_{#1}^N}

\renewcommand\to{\rightarrow}

\newcommand\ii{\mathrm{i}}
\newcommand\ee{\mathrm{e}}
\newcommand\re{\mathrm{Re}}
\newcommand\im{\mathrm{Im}}
\newcommand\id{\mathrm{id}}
\newcommand\dist{\mathrm{dist}}

\newcommand\dif{ {\mbox{\rm d}} }
\newcommand\Dif{ {\mbox{\rm D}} }
\newcommand\Tan{ {\mbox{\rm T}} }
\newcommand\pd{ \partial}

\newcommand\sform{\mbox{\boldmath $\omega$}}
\newcommand\aform{\mbox{\boldmath $\alpha$}}
\newcommand\gform{\mbox{\boldmath $g$}}

\newcommand\J{\mbox{\boldmath $J$}}

\newcommand\mani{{U}}
\newcommand\cmani{{\mathcal{U}}}
\newcommand\domPhi{{D}}
\newcommand\cdomPhi{{\mathcal{D}}}
\newcommand\localtimedomPhi{{B}}

\newcommand\Anal{{\cA}}

\newcommand\U{{\mathcal{U}}}
\newcommand\torus{{\mathcal{K}}}

\newcommand\LieO{\mathfrak{L}}     
\renewcommand\H{h}                   
\renewcommand\P{p}                   

\newcommand\fake{\nu}

\newcommand\fakeE{\kappa}
\newcommand\fakeetaN{\varkappa} 

\newcommand\Elag{\Omega_L} 
\newcommand\Esym{E_{\mbox{\tiny \rm sym}}} 

\newcommand\Ered{E_{\mbox{\tiny \rm red}}} 
\newcommand\EredLL{E_{\mbox{\tiny \rm red}}^{\mbox{\tiny $LL$}}}
\newcommand\EredLN{E_{\mbox{\tiny \rm red}}^{\mbox{\tiny $LN$}}}
\newcommand\EredNL{E_{\mbox{\tiny \rm red}}^{\mbox{\tiny $NL$}}}
\newcommand\EredNN{E_{\mbox{\tiny \rm red}}^{\mbox{\tiny $NN$}}}

\newcommand\DeltaK{\Delta K}

\newcommand\cteOmega{c_{\mbox{\tiny $\Omega$}}}
\newcommand\cteDOmega{c_{\mbox{\tiny ${\rm D}\Omega$}}}

\newcommand\cteJ{c_{\text{\tiny $J$}}}
\newcommand\cteDJ{c_{\text{\tiny $\Dif J$}}}

\newcommand\cteJT{c_{J^{\top}}}
\newcommand\cteDJT{c_{\text{\tiny $\Dif J^\ttop$}}}
\newcommand\cteG{c_{\mbox{\tiny $G$}}}
\newcommand\cteDG{c_{\mbox{\tiny ${\rm D}G$}}}

\newcommand\cteXH{c_{\text{\tiny $\Xh$}}}
\newcommand\cteDXH{c_{\text{\tiny ${\rm D} \Xh$}}}
\newcommand\cteDDXH{c_{\text{\tiny ${\rm D}^2 \Xh$}}}
\newcommand\cteDXHT{c_{\text{\tiny $({\rm D} \Xh)^\ttop$}}}

\newcommand\cteXp{c_{X_p}}
\newcommand\cteXpT{c_{X_p^\ttop}}
\newcommand\cteDXp{c_{\mbox{\tiny ${\rm D} X_p$}}}
\newcommand\cteDXpT{c_{\mbox{\tiny ${\rm D} X_p^\ttop$}}}

\newcommand\cteDPhi{c_{\text{\tiny $\Dif \Phi$}}}
\newcommand\cteTH{c_{T_\H}}

\newcommand\cteTh{c_{\text{\tiny $\T_\H$}}}
\newcommand\cteDTh{c_{\text{\tiny$\Dif \T_\H$}}}

\newcommand\sigmaDK{\sigma_{\mbox{{\tiny ${\rm D} K$}}}}
\newcommand\sigmaDKT{\sigma_{\mbox{{\tiny $({\rm D} K)\!^\ttop$}}}}

\newcommand\sigmaL{\sigma_{\mbox{{\tiny $\L$}}}}
\newcommand\sigmaLT{\sigma_{\mbox{{\tiny $\LT$}}}}
\newcommand\sigmaN{\sigma_{\mbox{{\tiny $\N$}}}}
\newcommand\sigmaNT{\sigma_{\mbox{{\tiny $\NT$}}}}

\newcommand\sigmaB{\sigma_B}
\newcommand\sigmaTinv{\sigma_{\mbox{{\tiny {$\aver{T}^{\mbox{-}1}$}}}}}

\newcommand\CE{C_{E}}
\newcommand\CET{C_{E^\ttop}}
\newcommand\CDE{C_{{\rm D}E}}
\newcommand\CDET{C_{({\rm D} E)^\ttop}}

\newcommand\Csym{C_{\mathrm{sym}}}

\newcommand\CL{\sigmaL}
\newcommand\CLT{\sigmaLT}
\newcommand\CN{\sigmaN}
\newcommand\CNT{\sigmaNT}

\newcommand\ttiny[1]{{\text{\tiny $#1$}}}
\newcommand\supN{{\ttiny{$N$}}}
\newcommand\supL{{\text{\tiny $L$}}}

\newcommand\CNOmegaL{C_{\Omega_{\text{\tiny $L$}}}^\supN}
\newcommand\CNOmegaN{C_{\Omega_\supN}^\supN}

\newcommand\CLE{C_{E}^{\text{\tiny $L$}}}
\newcommand\CLET{C_{E^\ttop}^{\text{\tiny $L$}}}
\newcommand\CLDE{C_{{\rm D}E}^{\text{\tiny $L$}}}
\newcommand\CLDET{C_{({\rm D}E)^\ttop}^{\text{\tiny $L$}}}

\newcommand\CNE{C_{E}^\supN}
\newcommand\CNET{C_{E^\ttop}^\supN}
\newcommand\CNDE{C_{{\rm D}E}^\supN}
\newcommand\CNDET{C_{({\rm D}E)^\ttop}^\supN}

\newcommand\CEredLN{C_{E_{\ttiny{\rm red}}^{\ttiny{$LN$}}}}

\newcommand\CEredNN{C_{E_{\ttiny{\rm red}}^{\ttiny{$NN$}}}}

\newcommand\CLEredLL{C^L_{E_{\ttiny{\rm red}}^{\ttiny{$LL$}}}}
\newcommand\CLEredLN{C^L_{E_{\ttiny{\rm red}}^{\ttiny{$LN$}}}}
\newcommand\CLEredNL{C^L_{E_{\ttiny{\rm red}}^{\ttiny{$NL$}}}}
\newcommand\CLEredNN{C^L_{E_{\ttiny{\rm red}}^{\ttiny{$NN$}}}}

\newcommand\CNEredLL{C^N_{E_{\ttiny{\rm red}}^{\ttiny{$LL$}}}}
\newcommand\CNEredLN{C^N_{E_{\ttiny{\rm red}}^{\ttiny{$LN$}}}}
\newcommand\CNEredNL{C^N_{E_{\ttiny{\rm red}}^{\ttiny{$NL$}}}}
\newcommand\CNEredNN{C^N_{E_{\ttiny{\rm red}}^{\ttiny{$NN$}}}}

\newcommand\CaverxiN{C_{\ttiny{\averxiN}}}
\newcommand\CLaverxiN{C_{\ttiny{\averxiN}}^{\ttiny{L}}}
\newcommand\CNaverxiN{C_{\ttiny{\averxiN}}^{\ttiny{N}}}

\newcommand\CxiN{C_{\xi^{\text{\tiny $N$}}}}
\newcommand\CLxiN{C_{\xi^{\text{\tiny $N$}}}^{\text{\tiny $L$}}}
\newcommand\CNxiN{C_{\xi^{\text{\tiny $N$}}}^{\text{\tiny $N$}}}

\newcommand\CDeltaKi{C_{\Delta \Ki}}
\newcommand\CLDeltaKi{C_{\Delta \Ki}^L}
\newcommand\CNDeltaKi{C_{\Delta \Ki}^N}

\newcommand\CDeltaDKi{C_{\ttiny{\Delta \DKi}}}
\newcommand\CLDeltaDKi{C_{\ttiny{\Delta \DKi}}^\supL}
\newcommand\CNDeltaDKi{C_{\ttiny{\Delta \DKi}}^\N}
\newcommand\CDeltaDKiT{C_{\ttiny{\Delta\DKiT}}}
\newcommand\CLDeltaDKiT{C_{\ttiny{\Delta\DKiT}}^\L}
\newcommand\CNDeltaDKiT{C_{\ttiny{\Delta\DKiT}}^\N}

\newcommand\CDeltaLi{C_{\Delta \Li}}
\newcommand\CLDeltaLi{C_{\Delta \Li}^\L}
\newcommand\CNDeltaLi{C_{\Delta \Li}^\N}
\newcommand\CDeltaLiT{C_{\Delta \LiT}}
\newcommand\CLDeltaLiT{C_{\Delta \LiT}^\L}
\newcommand\CNDeltaLiT{C_{\Delta \LiT}^\N}

\newcommand\CEi{C_{\Ei}}
\newcommand\CLEi{C_{\Ei}^\L}
\newcommand\CNEi{C_{\Ei}^\N}

\newcommand\QCetaiN{Q_{\etaiN}}

\newcommand\CxiL{C_{\xi^{\text{\tiny $L$}}}}
\newcommand\CLxiL{C_{\xi^{\text{\tiny $L$}}}^{\text{\tiny $L$}}}
\newcommand\CNxiL{C_{\xi^{\text{\tiny $L$}}}^{\text{\tiny $N$}}}

\newcommand\CDeltaKn{C_{\Delta \Kn}}
\newcommand\CLDeltaKn{C_{\Delta \Kn}^L}
\newcommand\CNDeltaKn{C_{\Delta \Kn}^N}

\newcommand\CDeltaDKn{C_{\Delta {\rm D}\Kn}}
\newcommand\CLDeltaDKn{C_{\Delta {\rm D}\Kn}^L}
\newcommand\CNDeltaDKn{C_{\Delta {\rm D}\Kn}^N}
\newcommand\CDeltaDKnT{C_{\Delta({\rm D}\Kn)^\ttop}}
\newcommand\CLDeltaDKnT{C_{\Delta({\rm D}\Kn)^\ttop}^L}
\newcommand\CNDeltaDKnT{C_{\Delta({\rm D}\Kn)^\ttop}^N}

\newcommand\CDeltaLn{C_{\Delta \Ln}}
\newcommand\CLDeltaLn{C_{\Delta \Ln}^L}
\newcommand\CNDeltaLn{C_{\Delta \Ln}^N}
\newcommand\CDeltaLnT{C_{\Delta \LnT}}
\newcommand\CLDeltaLnT{C_{\Delta \LnT}^L}
\newcommand\CNDeltaLnT{C_{\Delta \LnT}^N}

\newcommand\CDeltaGLn{C_{\Delta G_{\Ln}}}
\newcommand\CLDeltaGLn{C_{\Delta G_{\Ln}}^\L}
\newcommand\CNDeltaGLn{C_{\Delta G_{\Ln}}^\N}

\newcommand\CDeltaBn{C_{\Delta \Bn}}
\newcommand\CLDeltaBn{C_{\Delta \Bn}^\L}
\newcommand\CNDeltaBn{C_{\Delta \Bn}^\N}

\newcommand\CDeltaNn{C_{\Delta \Nn}}
\newcommand\CLDeltaNn{C_{\Delta \Nn}^\L}
\newcommand\CNDeltaNn{C_{\Delta \Nn}^\N}

\newcommand\CDeltaNnT{C_{\Delta \NnT}}
\newcommand\CLDeltaNnT{C_{\Delta \NnT}^\L}
\newcommand\CNDeltaNnT{C_{\Delta \NnT}^\N}

\newcommand\CDeltaTn{C_{\Delta \Tn}}
\newcommand\CLDeltaTn{C_{\Delta \Tn}^L}
\newcommand\CNDeltaTn{C_{\Delta \Tn}^N}

\newcommand\CDeltaiTn{C_{\Delta \aver{\Tn}^{\text{-}1}}}
\newcommand\CLDeltaiTn{C_{\Delta \aver{\Tn}^{\text{-}1}}^L}
\newcommand\CNDeltaiTn{C_{\Delta \aver{\Tn} ^{\text{-}1}}^N}

\newcommand\CEn{C_{\En}}
\newcommand\CLEn{C_{\En}^\L}
\newcommand\CNEn{C_{\En}^\N}

\newcommand\CLLiexiL{C_{\LieO\xi^{\L}}^L}
\newcommand\CNLiexiL{C_{\LieO\xi^{\L}}^N}
\newcommand\CLiexiL{C_{\LieO\xi^{\L}}}

\newcommand\QCetanN{Q_{\etanN}}

\newcommand\QCetanL{Q_{\etanL}}

\newcommand\QCetanLu{Q_{\etanLu}}
\newcommand\QCetanLd{Q_{\etanLd}}
\newcommand\QCetanLt{Q_{\etanLt}}
\newcommand\QCetanLq{Q_{\etanLq}}
\newcommand\QCetanLc{Q_{\etanLc}}

\newcommand\QCetan{Q_{\etan}}

\newcommand\CT{C_T}

\newcommand\CtheoE{\mathfrak{C}}
\newcommand\CtheoDeltaK{\mathfrak{C}_{\DeltaK}}
\newcommand\CtheoDeltaDK{\mathfrak{C}_{\Delta{\rm D} \K}}
\newcommand\CtheoDeltaDKT{\mathfrak{C}_{\Delta({\rm D} \K)^\ttop}}

\newcommand\CtheoDeltaL{\mathfrak{C}_{\mbox{\tiny{$\Delta\L$}}}}
\newcommand\CtheoDeltaLT{\mathfrak{C}_{\mbox{\tiny{$\Delta\LT$}}}}
\newcommand\CtheoDeltaN{\mathfrak{C}_{\mbox{\tiny{$\Delta\N$}}}}
\newcommand\CtheoDeltaNT{\mathfrak{C}_{\mbox{\tiny{$\Delta\NT$}}}}
\newcommand\CtheoDeltaB{\mathfrak{C}_{\Delta\B}}
\newcommand\CtheoDeltaiT{\mathfrak{C}_{\Delta\aver{T}^{\text{-}1}}}

\newcommand\CLieK{C_{\LieO K}}
\newcommand\CLieL{C_{\LieO L}}
\newcommand\CLieLT{C_{\LieO L^\ttop}}
\newcommand\CLieB{C_{\LieO B}}

\newcommand\CLieN{C_{\LieO N}}

\newcommand\CLieGL{C_{\LieO G_L}}

\newcommand\hCDelta{\hat{C}_\Delta}

\allowdisplaybreaks

\begin{document}
\title[A modified parameterization method]{
A modified parameterization method \\ for invariant Lagrangian tori \\ for
partially integrable Hamiltonian systems
}
\date{\today}

\author{Jordi Llu\'is Figueras$^{\mbox{1}}$}
\address[1]{Department of Mathematics, Uppsala University, 
Box 480, 751 06 Uppsala, Sweden}
\email{figueras@math.uu.se}

\author{Alex Haro$^{\mbox{2,3}}$}
\address[2]{Departament de Matem\`atiques i Inform\`atica, Universitat de Barcelona,
Gran Via 585, 08007 Barcelona, Spain.}
\address[3]{Centre de Recerca Matem\`atica, Edifici C, Campus Bellaterra, 08193 Bellaterra, Spain}
\email{alex@maia.ub.es}

\begin{abstract}
In this paper we present an a-posteriori KAM theorem for the existence of an
$(n-d)$-parameters family of $d$-dimensional isotropic invariant tori with
Diophantine frequency vector $\omega\in \RR^d$, of type $(\gamma,\tau)$,  for
$n$ degrees of freedom Hamiltonian systems with $(n-d)$ independent first
integrals in involution. If the first integrals induce a Hamiltonian action of
the $(n-d)$-dimensional torus, then we can produce  $n$-dimensional Lagrangian
tori with frequency vector of the form $(\omega,\omegap)$, with
$\omegap\in\RR^{n-d}$.  In the light of the parameterization method, we design
a (modified) quasi-Newton method for the invariance equation of the
parameterization of the torus, whose proof of convergence from an initial
approximation, and under appropriate non-degeneracy conditions, is the object
of this paper. We present the results in the analytic category, so the initial
torus is real-analytic in a certain complex strip of size $\rho$, and the
corresponding error in the functional equation is $\varepsilon$.  We heavily
use geometric properties and the so called automatic reducibility to deal
directly with the functional equation and get convergence if $\gamma^{-2}
\rho^{-2\tau-1}\varepsilon$ is small enough, in contrast with most of KAM
results based on the parameterization method, that get convergence if
$\gamma^{-4} \rho^{-4\tau}\varepsilon$ is small enough.  The approach is
suitable to perform computer assisted proofs.
\end{abstract}

\maketitle

\tableofcontents

\section{Introduction}

Persistence under perturbations of regular (quasi-periodic) motion is one of
the most important problems in Mechanics and Mathematical Physics, and has deep
implications in Celestial and Statistical Mechanics.  The seminal works  of
Kolmogorov \cite{Kolmogorov54}, Arnold \cite{Arnold63a} and Moser
\cite{Moser62} put the name to the KAM theory, that has become a full body of
knowledge that connects fundamental mathematical ideas in many different
contexts around the so-called small divisors.  See e.g. the popular book
\cite{Dumas14}, and the surveys \cite{BroerHS96,Llave01}.

Although KAM theory held for general dynamical systems under very mild
technical assumptions, its application to concrete systems became a challenging
problem.  With the advent of computers and new methodologies, the distance
between theory and practice was shortened (see e.g.
\cite{CellettiC88,LlaveR90,CellettiC07}).  One direction that have experienced
a lot of progress is the a-posteriori approach based on the parameterization
method \cite{Llave01,GonzalezJLV05,HaroCFLM16,Sevryuk14,FiguerasHL20}, that in
this context was originally known as KAM theory without angle-action
coordinates. This approach lead to the design  of a general methodology to
perform computer assisted proofs of existence of Lagrangian invariant tori
\cite{FiguerasHL17}, enlarging the threshold of validity to practically the one
predicted by numerical observations \cite{Greene79,MacKay93}, in academic
examples such as the Chirikov stardard map.  But new impulses have to be made
in order to make KAM theory fully applicable to realistic physical systems,
which often have extra first integrals, or  degeneracies.

The goal of this paper is to present a KAM theorem in a-posteriori format for
the existence of invariant Lagrangian cylinders for Hamiltonian systems with
first integrals in involution, see Theorem~\ref{thm:KAM}.  If the first
integrals induce a Hamiltonian action of a torus, then the theorem produces
invariant Lagrangian tori.  The presence of first integrals in involution is
usually treated with symplectic reduction techniques
\cite{Cannas01,MarsdenW74}, and then applying KAM theorems to the reduced
systems. For the sake of versatility, we do not pursue such changes of
variables, and we get $n$-dimensional Lagrangian invariant cylinders from a
single $d$-dimensional isotropic torus through a Reduction lemma, see
Lemma~\ref{lem:reduction lemma}, that avoids the use of symplectic reduction
techniques.  Of course, our results  work for systems without additional first
integrals in involution, that could be obtained after reduction techniques.

Although the setting of the paper is close to that in \cite{HaroL19}, we
incorporate some constructs in \cite{Villanueva17} to improve crucial
estimates. 
We present a (modified) quasi-Newton method for these systems. The method
solves exactly the same equations as in the quasi-Newton method in
\cite{HaroL19}, but differ in the last step with \cite{HaroL19} in the way the
updates of the parameterization of the torus are made. In particular, normal
corrections to the torus are made to improve the invariance of the torus, while
tangent corrections to the torus are made to conjugate the internal dynamics of
the torus to a linear flow. As a result, while in \cite{HaroL19} we got
convergence of the quasi-Newton method if $\gamma^{-4}
\rho^{-4\tau}\varepsilon$ is small enough, as in most of KAM papers based on
the parameterization method, in this paper we get convergence if $\gamma^{-2}
\rho^{-2\tau-1}\varepsilon$ is small enough, as in \cite{Villanueva17},
\cite{Villanueva18} (see also the prequel \cite{Villanueva08}), and in the KAM
Theorem \cite{Arnold63a,Neishtadt81,Poschel82}. But, while  the inspiring paper
\cite{Villanueva17} substitutes the invariance equation of the parameterization
method by three different conditions which are altogether equivalent to
invariance, here we suitably project the error of invariance in tangent and
normal components, as in \cite{HaroL19}. Also, our proof of the result heavily
relies on geometrical and reducibility properties, complementing the approach
in  \cite{Villanueva17}, that does not use (at least explicitly) these
properties. 

In order to emphasize the geometric properties, and for the sake of
versatility, we have considered symplectic structures other than the canonical.
This was one main point in the seminal paper \cite{GonzalezJLV05}. Here,
however, we have only considered symplectic structures that have  a compatible
almost-complex structure, i.e. that is inducing a Riemannian metrics, as in
\cite{GonzalezHL13}.  It is known that any symplectic manifold admits a
compatible almost-complex structure (see e.g. \cite{Cannas01}). 

This paper complements \cite{HaroL19} also in the sense that, from the algorithm derived in that paper (whose convergence is proved), 
one can produce approximations of the solutions of the invariance equations, and then apply the a posteriori KAM theorem derived in this paper, in combination with techniques introduced in \cite{FiguerasHL17}. There are situations in which the a posteriori KAM theorem in \cite{HaroL19} could fail to be applied in practical situations, in which computer resources time and memory are finite. For instance, when the tori are high dimensional and/or about to break, thus needing high order Fourier approximations and/or being the size of the analyticity strip $\rho$ very small. The improved estimates here could mitigate such problems, thus pushing further the domain of existence of KAM tori. 

As usual in KAM theory, and in particular in the parameterization method, the proof of existence of invariant tori is pursued by means of a quasi-Newton method in a scale of Banach spaces, here analytic functions. At each step the analyticity strip in which the objects are defined is reduced, and one has to control the whole sequence to get convergence to an object defined in a final analyticity strip. The way the bites are produced on the analyticity strip seems to influence the results in practical situations. We have included a digression on the fact that the best choice seems to be close to consider geometric series of bites with ratio $\frac12$.

The paper is organized as follows.
Section~\ref{sec:setting and KAM theorem} introduces the background and the geometric and analytical constructions, and the main result, Theorem~\ref{thm:KAM}, is presented at the end. The proof of Theorem~\ref{thm:KAM} is detailed in Section~\ref{sec:proof KAM}. An auxiliary lemma to control inverses of matrices is included in Appendix~\ref{app:invertibility}. 
In order to collect the long list of expressions leading to the explicit estimates and conditions of the theorems, we include
separate tables in the Appendix~\ref{app:constants}.  
We pay special attention in providing explicit and rather optimal bounds,
with an eye in the application of the theorems and in computer assisted proofs.

\section{The setting and the KAM theorem}\label{sec:setting and KAM theorem}

In this section we set the geometrical and analytical background of this paper,
and present the main result.  In Subsection~\ref{ssec:basic notation} we
establish the basic notation.  In Subsection~\ref{ssec:symplectic setting} we
introduce the geometrical setting of this paper, i.e. symplectic structures  on
open sets of $\RR^{2n}$ that admit compatible almost-complex structures.  In
Subsection~\ref{ssec:hamiltonian systems} we review standard definitions for
Hamiltonian systems and first integrals, with an eye in Hamiltonian actions of
tori.  In Subsection~\ref{ssec:invariant tori} we set the main equations of this
paper, to find invariant tori carrying quasi-periodic motion, and some
implications of the presence of first integrals in involution, that lead to a
reduction of the dimensionality of the problem of finding Lagrangian invariant
cylinders and tori, without the need of reducing the Hamiltonian itself, see
Lemma~\ref{lem:reduction lemma}.  In Subsection~\ref{ssec:reducibility} we
review some geometrical constructions and reducibility properties of invariant
tori, and present some implications of the presence of compatible triples, see
Proposition~\ref{prop:reducibility}.  In Subsection~\ref{ssec:quasi Newton} we
review the quasi-Newton method introduced in \cite{HaroL19}, and present a new
modified quasi-Newton method to solve the invariance equations.  In
Subsection~\ref{ssec:analytic setting} we introduce the spaces in which the
invariance equations are considered, i.e. spaces of analytic functions, and
review some key results regarding small divisors equations, see
Lemmas~\ref{lem:Russmann} and \ref{lem:Russmann-Cauchy}. Finally, in
Subsection~\ref{ssec:KAM theorem} we present the main result of this paper,
Theorem~\ref{thm:KAM}, an a posteriori  theorem on the existence of isotropic
invariant tori for Hamiltonian systems with first integrals in involution. The
proof of this theorems constitutes the bulk of this paper and is found
in Section~\ref{sec:proof KAM}.

\subsection{Basic notation}\label{ssec:basic notation}

We denote by $\RR^m$ and $\CC^m$ the vector spaces of $m$-dimensional vectors
with components in $\RR$ and $\CC$, respectively, endowed with the norm
\[
|v|= \max_{i=1,\dots,m} |v_i|.
\] 
Given $a,b\in\CC$, we also often use the notation $\Abs{a,b}=
\max\{|a|,|b|\}$.
We consider the real and imaginary projections 
$\re,\im: \CC^m\to \RR^m$, and identify 
$\RR^m \simeq \im^{\text{-}1}\{0\} \subset \CC^m$.
Given $U\subset \RR^m$ and $\rho>0$, the complex strip of size $\rho$ is 
$U_\rho= \{ \theta  \in \CC^m \,:\, \re\, \theta \in U \, , \, |\im\, \theta|<\rho \}$.
Given two sets $X,Y \subset \CC^{m}$, $\dist(X,Y)$ is defined as 
$\inf\{|x-y| \,:\, x\in X \, , \, y\in Y\}$.

We denote $\RR^{n_1 \times n_2}$ and $\CC^{n_1\times n_2}$ the spaces of $n_1
\times n_2$ matrices with components in $\RR$ and $\CC$, respectively,
identifying $\RR^m\simeq \RR^{m\times 1}$ and $\CC^m\simeq \CC^{m\times 1}$.  We
denote $I_n$ and $O_n$ the $n\times n$ identity and zero matrices, respectively.
The $n_1\times n_2$ zero matrix is represented by $O_{n_1\times n_2}$.  Finally,
we use the notation $0_n$ to represent the column vector $O_{n \times 1}$,
although we mostly write $0$ when the dimension is known from the context.
Matrix norms in both  $\RR^{n_1 \times n_2}$ and $\CC^{n_1\times n_2}$ are the
ones induced from the corresponding vector norms.  That is to say, for an $n_1
\times n_2$ matrix $M$, we have
\[
	|M| = \max_{i= 1,\dots,n_1} \sum_{j= 1,\dots, n_2} |M_{i,j}|.
\]
In particular, if $v\in \CC^{n_2}$, 
$|M v|\leq |M| |v|$. Moreover, $M^\ttop$ denotes the transpose of the matrix $M$, so that 
\[
 	|M^\ttop|= \max_{j= 1,\dots, n_2} \sum_{i= 1,\dots,n_1} |M_{i,j}|.
\]
Notice that $|M^\ttop|\leq n_1 |M|$ and
$|(Mv)^\ttop|\leq |M^\ttop|\, |v^\top|$, but also $|(Mv)^\ttop|\leq n_1 |M|\, |v|$ and 
$|(Mv)^\ttop|\leq n_2 |M^\ttop|\, |v|$.

Given an analytic function $f: \U\subset \CC^m \to \CC$, defined on an open set
$\U$, the action of the $r$-order derivative of $f$ at a point $x\in \U$ on  a
collection of vectors $v_1,\dots, v_r\in \CC^m$, with $v_k= (v_{1k}, \dots,
v_{mk})$, is
\[
	\Dif^r f(x) [v_1,\dots,v_r] = \sum_{\ell_1,\dots,\ell_r} \frac{\partial^r f}{\partial x_{\ell_1}\dots\partial x_{\ell_r}}(x)\ v_{\ell_1 1} \cdots v_{\ell_r r}, 
\]
where the indices $\ell_1,\dots,\ell_r$ run from $1$ to $m$.
This construction is extended
to vector and matrix-valued maps as 
follows: given
a matrix-valued map $M:\U\subset\CC^m\to \CC^{n_1\times n_2}$
(whose components $M_{i,j}$ are analytic functions), a point $x\in \U$,
and a collection of (column) vectors $v_1,\dots, v_r\in \CC^m$, 
we obtain an $n_1\times n_2$  matrix  $\Dif^r M(x) [v_1,\dots,v_r]$ 
such that 
\[
\left(\Dif^r M(x) [v_1,\dots,v_r]\right)_{i,j} = \Dif^r M_{i,j}(x) [v_1,\dots,v_r].
\] 
Notice that, if we split $M$ in its columns $M_{\cdot, j}$ for $j= 1,\dots, n_2$, so that 
$(M_{\cdot, j})_i= M_{i,j}$ for $i= 1,\dots, n_1$, we have 
\[
\Dif^r M(x) [v_1,\dots,v_r]= 
\begin{pmatrix}
\Dif^r M_{\cdot, 1}(x) [v_1,\dots,v_r] & \dots & \Dif^r M_{\cdot,n_2}(x) [v_1,\dots,v_r]
\end{pmatrix}.
\] 
For $r=1$, we will often write $\Dif M(x) [v]= \Dif^1 M(x) [v]$ for $v\in \CC^m$.

Given a function $f: \U \subset \CC^m \to \CC^n\simeq \CC^{n\times 1}$,
we can think of  $\Dif f$ as a matrix function 
$\Dif f: \U \to \CC^{n\times m}$, hence, $\Dif^1 f(x) [v]= \Dif f(x) v$ for $v\in\CC^m$.
Therefore,
we can apply the transpose to obtain a matrix function $(\Dif f)^\ttop$, 
which acts on $n$-dimensional vectors, while $\Dif f^\ttop=\Dif (f^\ttop)$ acts on $m$-dimensional vectors.
Hence, according to the above notation, the operators $\Dif$ and $(\cdot)^\ttop$ do not
commute. Therefore, in order to avoid confusion, we must pay attention to the use of parenthesis. 

A function $u : \RR^d \to \RR$ is 1-periodic if $u(\theta+e)=u(\theta)$ for all $\theta \in \RR^d$ and $e \in \ZZ^d$.
Abusing notation, we write
$u:\TT^d \to \RR$, where $\TT^d = \RR^d/\ZZ^d$ is the $d$-dimensional standard torus. 
Analogously, for $\rho>0$, a function $u:\RR^d_\rho\to \CC$
is 1-periodic if $u(\theta+e)=u(\theta)$ for all $\theta \in \RR^d_\rho$ and $e \in \ZZ^d$.
We also abuse notation and write  $u:\TT^d_\rho \to \CC$, where 
$\TT^{d}_{\rho}=  \{\theta\in \CC^d / \ZZ^d : |\im\ \theta| < \rho\}$ is the complex strip of  $\TT^d$ of width
$\rho>0$. 
We write the Fourier expansion of a periodic function as
\[
u(\theta)=\sum_{k \in \ZZ^d} \hat u_k \ee^{2\pi \ii k \cdot \theta}, \qquad \hat u_k =
\int_{\TT^d} u(\theta) \ee^{-2 \pi \ii k \cdot \theta} \dif \theta.
\]
and introduce the notation $\aver{u}:=\hat u_0$ for the average. Given $\omega\in\RR^d$, 
we define the operator $\Loper$ acting on $u$ as the Lie derivative of $u$ in the direction of the constant vector field $\dot\theta= -\omega$ on the torus:
\begin{equation}\label{def:Loper}
\Loper  u = -\Dif u\: \omega = -\sum_{i=1}^d \omega_i \frac{\pd u}{\pd \theta_i}.
\end{equation}
The corresponding Fourier series of $v= \Loper  u$ is 
\[
v(\theta)=\sum_{k \in \ZZ^d} \hat v_k \ee^{2\pi \ii k \cdot \theta}, \qquad \hat v_k = -2\pi \ii (k\cdot\omega) \hat u_k.
\]
The notation in this paragraph is extended to $n_1 \times n_2$
matrix-valued periodic maps $M: \TT^{d}_{\rho}\to\CC^{n_1\times n_2}$, for which
$\hat M_k \in \CC^{n_1 \times n_2}$ denotes the Fourier coefficient of index
$k\in \ZZ^d$.

\subsection{Symplectic setting} \label{ssec:symplectic setting}

In this paper we consider the phase space is an open set $\mani$ of $\RR^{2n}$,
whose points are denoted by $z= (z_1,\dots, z_{2n})$,  endowed with an {\em
exact symplectic form} $\sform= \dif \aform$, where the 1-form $\aform$ is
called \emph{action form}.  The setting and the results of this paper can be
easily adapted to other settings such as $\mani\subset \TT^k \times \RR^{2n-k}$
with $k\leq n$. See e.g. \cite{HaroCFLM16} for a discussion.

In order to simplify some of the geometrical constructs of this paper, we also
assume that $\mani$  is endowed with a Riemannian metric $\gform$ and an
anti-involutive linear isomorphism $\J:  \Tan \mani \to \Tan\mani$, i.e. $\J^2=
-I$,  such that $\forall z\in \mani, \forall u,v \in T_z \mani,$ $\sform_z(\J_z
u,v)=\gform_z(u,v)$.  It is said that $(\sform,\gform,\J)$ is a compatible
triple and that $\J$ endows $\mani$ with an almost-complex structure. The
anti-involution preserves both 2-forms $\sform$ and $\gform$.

We rather use the matrix representations of the previous objects, given by the
matrix-valued maps $a: \mani \longrightarrow \RR^{2n}$, representing the 1-form
$\aform$, and $\Omega,G, J : \mani \longrightarrow  \RR^{2n\times 2n}$,
representing the 2-forms $\sform$ and $\gform$, and the
anti-involution $\J$, respectively.  The fact that $\sform$ is closed reads
\[
 \frac{\partial \Omega_{r,s}}{\partial z_t} 
+\frac{\partial \Omega_{s,t}}{\partial z_r} 
+\frac{\partial \Omega_{t,r}}{\partial z_s} = 0,
\]
for any triplet $(r,s,t)$, and the fact that $\sform$ is exact,  with
$\sform= \dif \aform$ reads
\[
 \Omega =(\Dif a)^\ttop-\Dif a.
\]
Moreover,  $\Omega^\ttop = -\Omega$, and $\Omega$ is pointwise invertible.
Moreover, the metric condition of $\gform$ reads  $G^\ttop= G$ and it is
positive definite, and the compatibility conditions read
\[
J^\ttop \Omega= -\Omega J = G, \qquad J^2= -I_{2n}.
\]
Notice also that $\Omega= G J$.
These properties also imply the relations
\[
\Omega= J^\ttop \Omega J, \qquad G=  J^\ttop G J.
\end{equation*}

\begin{remark}
The prototype example of compatible triple is $(\sform_0,\gform_0,\J_0)$ in
$\RR^{2n}$, where $\sform_0$ is the \emph{standard symplectic structure},
$\sform_0= \sum_{i= 1}^n  {\rm d} z_{n+i}\wedge {\rm d}z_i$, $\gform_0$ is the
Euclidian metric, and $\J_0$ is the linear complex structure in $\RR^{2n}$ (as
a real vector space), coming from the complex structure in $\CC^n$. The matrix
representations of these objects are
\[
\Omega_0= \begin{pmatrix} O_n & -I_n \\ I_n & O_n \end{pmatrix},\quad 
G_0= \begin{pmatrix} I_n & O_n \\ O_n & I_n \end{pmatrix},\qquad
J_0= \begin{pmatrix} O_n & -I_n \\ I_n & O_n \end{pmatrix}.
\]
Moreover, an action form for $\sform_0$ is ${\aform_0}=\tfrac{1}{2}\sum_{i=
1}^n (z_{n+i}\ {\rm d} z_i-z_i \ \dif z_{n+i})$, which is represented as
\[
a_0(z)= \frac{1}{2} \begin{pmatrix} O_n & I_n \\ -I_n & O_n\end{pmatrix} z\,. 
\]
\end{remark}

\begin{remark} 
Even though usually KAM theory is presented in the stardard case, a notable
counterexample is the seminal paper \cite{GonzalezJLV05} that was followed by
other papers such as \cite{FontichLS09, GonzalezHL13,FiguerasHL17}.  Even more
general constructs were presented in chapter 4 of \cite{HaroCFLM16}.
\end{remark}

\subsection{Hamiltonian systems  and first integrals in involution}\label{ssec:hamiltonian systems}

Given a function ${\H : \mani\to \RR}$, the corresponding Hamiltonian vector
field $\Xh : \mani\to \RR^{2n}$ is the one such that $i_{\Xh} \sform =-\dif
\H$.  In coordinates, the Hamiltonian vector field $\Xh$ satisfies 
\[
\Xh(z)^\ttop \Omega(z) = -\Dif \H(z)\,,
\qquad
\mbox{i.e.,}
\qquad
\Xh(z)= \Omega(z)^{\text{-}1} (\Dif \H(z))^\ttop\,.
\]
Using Cartan's magic formula, the Lie derivative of $\sform$ in the direction
of $X_h$ vanishes, 
\begin{equation}\label{eq:prop2killSK}
\Dif \Omega [ \Xh ]
+ (\Dif \Xh)^\ttop \Omega
+ \Omega\: \Dif \Xh=
O_{2n}.
\end{equation}

The Poisson bracket of two functions $f$, $g$ is given by
$\{f,g\} = 
-\sform(X_f,X_g)$.
It is related to the Lie bracket through the well-known formula 
$
[X_f, X_g]= - X_{\{f,g\}},
$
that in coordinates is written as 
\[
\{f,g\} = -(X_f)^\ttop \Omega \: X_g= \Dif f\: X_g.
\]
In particular, $f$  is a first integral or preserved quantity of $X_g$ if and
only $\{f,g\}=0$, and it is said that $f,g$ are in involution. As a
consequence, 
\[
	\Dif X_f \: X_g  = \Dif X_g \: X_f
\]
and the corresponding flows commute.

We assume there is a {\em moment map} $\P:\mani\to \RR^{n-d}$, meaning that its
components $\P_1,\dots, \P_{n-d}$, jointly with $\H$, are pairwise in
involution functionally independent functions. We encode the corresponding
Hamiltonian vector fields as the columns of 
${\Xp: \mani \to \RR^{2n\times (n-d)}}$: 
\[
\Xp= \Omega^{\text{-}1} (\Dif p)^\ttop.
\]
The properties mentioned above are summarized as follows: $\Dif \P\: \Xp= 0$,
$\Dif \P\: \Xh= 0$ (and, then, $\Dif \Xp [\Xh] = \Dif \Xh \ \Xp$), and the
matrix $\begin{pmatrix} \Xh(z) & \Xp(z) \end{pmatrix}$ has (maximal) rank
$n-d+1$, for any $z\in \mani$.  

For $j= 1,\dots, n-d$, the vector fields $X_{\P_j}$ generate (local) flows
$\varphi^{j}: \domPhi_j\subset \RR\times\mani\to\mani$, for which we write
$\varphi^j_{s_j}(z)= \varphi^j(s_j,z)$ for $(s_j,z)\in \domPhi_j$.  The flows
commute (and also commute with the flow of $X_\H$). We then define
$\Phi:\domPhi \subset \RR^{n-d}\times\mani\to\mani$ as 
\begin{equation*}
\Phi(s,z) = \varphi_{s_1}^{1} \comp \cdots \comp \varphi_{s_{n-d}}^{n-d}(z), 
\end{equation*}
where 
\[
\domPhi=  \{ (s,z)\in \RR^{n-d}\times \mani \ | \ (s_{n-d},z)\in \domPhi_{n-d}, \dots,  (s_1, \varphi_{s_{2}}\comp \dots \comp \varphi_{s_{n-d}}^{n-d}(z))\in \domPhi_1\}.
\]
Since 
\[
\Dif_s\Phi = \Xp\comp \Phi,
\] 
we say that $\Phi$ is the (local) {\em moment flow} associated to $\Xp$, and
write $\Phi_s(z)= \Phi(s,z)$. Notice also that 
\[
\Dif_z \Phi \ X_p = X_p\comp \Phi,
\]
and 
\[
\Dif_z \Phi \ X_h = X_h\comp\Phi.
\]

The case $d= 1$ corresponds to the integrable case, and from now on we will assume $d>1$.
The case $d= n$ corresponds to not assuming the existence of first integrals in involution, and all the results of this paper hold for such a case.

\begin{remark}\label{rem:Hamiltonian action}
An important special case is when the moment map $p$ induces a Hamiltonian
torus action, that is $\domPhi= \RR^{n-d}\times \mani$ and the generated flow
$\Phi: \RR^{n-d}\times \mani\to \mani$ is periodic in all components of $s$. By
scaling {\em times} we can get all periods equal to one, and then, with a
slight abuse of notation, consider $\Phi:\TT^{n-d}\times\mani\to \mani$.
\end{remark}

\subsection{Invariant tori}\label{ssec:invariant tori}

In this paper, we refer to an embedding $\K : \TT^d \rightarrow \mani$  as a
parameterization of the torus $\cK= \K(\TT^d)$. 

Given 
$\omega \in \RR^d$ with $1\leq d\leq n$ , we say that $\K$ is invariant for
$\Xh$ with frequency vector $\omega$ if \begin{equation}\label{eq:inv:fv}
\Xh \comp K  + \Loper K= 0.
\end{equation}
This means that the $d$-dimensional torus $\torus=K(\TT^d)$ is invariant and
the internal dynamics is given by the constant vector field $\dot\theta=
\omega$.  Equation \eqref{eq:inv:fv} is called \emph{invariance equation for
$K$ and frequency $\omega$}.  Case $d= 1$ corresponds to $\torus$ being a
periodic orbit. We then assume $d\geq 2$ and $\omega$ to be {\em ergodic}: for
$k\in \ZZ^d\backslash\{0\}, k \cdot \omega\neq 0$. Hence, the flow on the torus
is quasi-periodic and non-resonant.

\begin{remark}
Other topologies can be considered for both the ambient manifold $\mani$ and the torus $\cK$. 
See e.g. \cite{HaroCFLM16} for a discussion.
\end{remark}

Ergodicity of $\omega$ implies additional geometric and dynamical properties of
the torus  $\cK$.  In particular, it is \emph{isotropic}, i.e. the pullback
$K^*\sform$ is zero  (see \cite{Herman86, Moser66c}), and it is contained in an
energy level of the Hamiltonian and of the additional first integrals (if any):
\[
\H\comp K = \aver{\H \comp K}, \quad \P\comp K= \aver{\P\comp K},
\]
since $\Loper (\H\comp K)= 0$ and $\Loper (\P\comp K)= 0$.

Another  consequence of the presence of extra first integrals in involution is
that an invariant torus $\torus$, with frequency vector $\omega$, induces a
family of invariant tori  $\torus_s = \Phi_s(\torus)$, with the same frequency
vector (with $s$ defined, a priori, in an open neighborhood $\localtimedomPhi$
of $0\in \RR^{n-d}$). If $K_s= \Phi_s\comp K$ is the parameterization of
$\torus_s$:
\[
\Xh \comp K_s  + \Loper K_s = 
(\Dif \Phi_s)\comp K\:  \Xh\comp K + (\Dif \Phi_s)\comp K \: \Loper K = 0.
\] 
We can think the family foliating an $n$-dimensional invariant object $\hat\torus$ 
parameterized by $\hat\K: \TT^d \times \localtimedomPhi \to \mani$ defined as 
\[
	\hat\K(\theta,s) = \Phi_s(K(\theta)).
\]
Notice that we can rephrase the previous argumentation by writing
\[
\Xh \comp \hat K  + \LieO_{(\omega,0)} \hat K= 0.
\]
This argument of getting $n$-dimensional invariant objects (including
$n$-dimensional invariant tori) from $d$-dimensional invariant tori is extended
in the following lemma.

\begin{lemma}[Reduction lemma]\label{lem:reduction lemma}
Let $\H:\mani\to \RR$ be a Hamiltonian for which there exists a moment map
$\P:\mani\to \RR^{n-d}$, being
$\Phi:\domPhi\subset \RR^{n-d}\times \mani \to \mani$ the corresponding moment
flow.  Let $f:p(\mani) \subset \RR^{n-d}\to \RR$ be a function defined on the
set of possible momenta, and let $\hat\H:\mani\to\RR$ be the {\em discounted
Hamiltonian}, defined as $\hat\H= \H - f\comp p$.  Let $K:\TT^d\to \mani$ be a
parameterization of a torus $\cK=
\K(\TT^d)$, invariant for $X_{\hat\H}$ with ergodic frequency $\omega\in\RR^d$,
i.e. 
\begin{equation}\label{eq:KinvarianceXhat}
	X_{\hat\H} \comp K + \LieO_{\omega} K= 0.
\end{equation}
Let $\localtimedomPhi$ be an open neighborhood of $0\in \RR^{n-d}$ such that 
$\localtimedomPhi\times \cK \subset \domPhi$.
Define $\hat\K: \TT^d \times \localtimedomPhi\to \mani$ as 
\[
	\hat\K(\theta, s)= \Phi(s, \K(\theta)),
\]
and  $p_0= \aver{p\comp\K}$, $\omega_{p_0}= (\Dif f (p_0))^\ttop$.
Then, 
\begin{equation}\label{eq:KhatinvarianceXhat}
	\Xh\comp \hat\K  + \LieO_{(\omega,\omega_{p_0})} \hat\K= 0.
\end{equation}
That is $\hat\cK= \hat\K(\TT^d \times \localtimedomPhi)$ is an invariant object for $\Xh$.

Moreover, if for a certain $\theta_0\in \TT^d$ there exists $S\in
\localtimedomPhi$  such that $\hat K(\theta_0,S)= K(\theta_0)$, then for all
$\theta\in\TT^d$ we have $\hat K(\theta,S)= K(\theta)$, we can take
$\localtimedomPhi= \RR^{n-d}$ and $\hat K$ is $S$-periodic in the $s$ variables
(i.e., $\hat K(\theta, s+S) =\hat K(\theta)$). That is $\hat\cK= \hat\K(\TT^d
\times \RR^{n-d})$ is an $n$-dimensional invariant torus for $\Xh$.
\end{lemma}

\begin{proof} 
The hypotheses \eqref{eq:KinvarianceXhat} of invariance of $K$ for the vector
field $X_{\hat \H}$ reads 
\[
0= X_h\comp K - X_p\comp K\: (\Dif f \comp p\comp K)^\ttop+ \LieO_{\omega} K.
\]
Then, since
\[
\Loper(p\comp K)= \Dif p \comp K\: \Loper K= - \Dif p \comp K\: (X_h\comp K - X_p\comp K\: (\Dif f \comp p\comp K)^\ttop)= 0, 
\]
and $\omega$ is ergodic, then $p\comp K= \aver{p\comp K}= p_0$ is constant. 

Then, since $\Dif_z\Phi_s(z) X_h(z)= X_h(\Phi_s(z))$ and 
$\Dif_z\Phi_s(z) X_p(z)= X_p(\Phi_s(z)) = \Dif_s\Phi_s(z)$, we get
\[
\begin{split}
X_h \comp \hat K + \LieO_{(\omega,\omega_{p_0})} \hat K
& = 
X_h\comp \Phi_s \comp K + \Dif_z\Phi_s \comp K\: \LieO_\omega K - X_p\comp \Phi_s\comp K\: \omega_{p_0}
\\
& 
=  \Dif_z\Phi_s\comp K \: (X_h\comp K - X_p\comp K\: (\Dif f \comp p\comp K)^\ttop + \LieO_\omega K) = 0,
\end{split}
\]
proving \eqref{eq:KhatinvarianceXhat}.

Let us assume now that $K(\theta_0)= \hat K(\theta_0,S)= \Phi_S(K(\theta_0))$.
Denote $\hat\varphi_t$ the (local) flow of $X_{\hat \H}$. Then, for all $t\in
\RR$,
\[
K(\theta_0 + \omega t) = \hat\varphi_t(K(\theta_0)) = \hat\varphi_t(\Phi_S(K(\theta_0))) = \Phi_S(\hat\varphi_t(K(\theta_0))) = 
\Phi_S(K(\theta_0 + \omega t)), 
\]
and the result follows from ergodicity of $\omega$.
\end{proof}

\begin{remark}\label{rem:uniqueness}
We can think of  $\K$ as the generator of the
$n$-dimensional object $\hat\cK$. The generator is defined up to a change of
phase, say $\alpha\in\RR^d$, and flying times, say $\beta\in \localtimedomPhi$,
since the parameterization $K_{\alpha,\beta}$ defined as 
\[
K_{\alpha,\beta}(\theta)= \Phi_{\beta}(K(\theta+\alpha)),
\] 
is also a generator of $\hat\cK$. This indeterminacy of the generator can be
fixed by imposing extra conditions apart from the invariance equation. Some
strategies are described in \cite{HaroL19}.
\end{remark}

\begin{remark}
In case of existence of $\theta_0\in \TT^d$  such that $\hat K(\theta_0,S)=
K(\theta_0)$ for a certain $S\in\RR^{n-d}$, Lemma~\ref{lem:reduction lemma}
establishes the existence of an $n$-dimensional invariant torus $\hat\cK$, with
frequency vector $\tilde \omega= (\omega, \omega_{p_0}/S)$ (we understand here
that the division of the two vectors is made componentwise). In other words, by scaling
periods we get a parameterization $\tilde K:\TT^{d}\times \TT^{n-d}\to \mani$ 
defined as 
\[
	\tilde K(\theta,\vartheta)= \hat K(\theta,S \vartheta)
\] 
(again considering the product $S\vartheta$ componenwise), 
that satisfies 
\[
\Xh\comp \tilde \K  + \LieO_{\tilde\omega} \tilde\K= 0.
\]
\end{remark}

\begin{remark}\label{rem:Hamiltonian torus action}
In case the moment map $p$ induces a Hamiltonian torus action, see
Remark~\ref{rem:Hamiltonian action}, we can consider the moment flow as a map
$\Phi: \TT^{n-d}\times \mani\to \mani$ (this is in fact the Hamiltonian action
of $\TT^{n-d}$ on $\mani$). In particular, for a fixed complementary frequency
$\omegap\in \RR^{n-d}$, to which we will refer to as a {\em moment frequency},
we consider the discounted Hamiltonian $\hat\H_{\omegap} = \H - \P^\ttop
{\omegap}$. Then, a parameterization $K:\TT^d\to\mani$ invariant for
$X_{\hat\H_\omegap}$ with frequency $\omega$ induces a parameterization $\tilde
K:\TT^n\to\mani$ invariant for $X_\H$ with frequency $\tilde\omega=
(\omega,\omegap)$.  We can also think of the moment frequency $\omegap$ as a
parameter, and the discounted Hamiltonian $\hat\H_\omegap$ as a family of
Hamiltonians.  From the results of this paper, one can obtain $n-d$ parametric
families of $n$-dimensional tori with fixed {\em  internal frequency} $\omega$,
parameterized by $\omegap$. 
\end{remark}

\subsection{Linearized dynamics and reducibility}\label{ssec:reducibility}

In this section we describe the geometric construction of a suitable symplectic
frame attached to a torus $\torus$ with respect to a Hamiltonian system $\Xh$,
possibly with first integrals $\P$, and an ergodic frequency $\omega\in\RR^d$.

In the following, for a matrix-valued map $V : \TT^d \rightarrow \RR^{2n \times
m}$, with $1\leq m \leq 2n$, we introduce the matrix-valued map $\Xoper{V}:\TT^d
\rightarrow \RR^{2n \times m}$ defined as
\begin{equation}\label{def:Xoper}
\Xoper{V} := \Dif \Xh\comp K\: V + \Loper V.
\end{equation}
If we think of $V$ as a parameterization of a frame of an $m$-dimensional vector
bundle $\cV$, then $\Xoper{V}$ corresponds to its infinitesimal displacement by
the flow of $\Xh$ around $\cK$.  We say that $V$ is invariant under (the
linearized equations of) $\Xh$ if $\Xoper{V}=O_{2n \times m}$.  We also define
the pullback of $\Omega$ and $G$ of $V$ on $K$ to be the matrix-valued maps
$\Omega_V, G_V: \TT^d \to \RR^{m\times m}$ defined as
\begin{equation*}\label{def:OmegaV}
	\Omega_V = V^\ttop \Omega\comp K\: V, 
\end{equation*}
and 
\begin{equation*}\label{def:GV}
G_V = V^\ttop G\comp K\: V,
\end{equation*}
respectively.  We say  that $V$ is isotropic if $\Omega_V= O_m$ and Lagrangian
if, moreover, $m= n$. We also say that $V$ is orthonormal if $G_V= I_m$.

We consider the matrix-valued map $L:\TT^d \rightarrow \RR^{2 n\times n}$ given
by
\begin{equation}\label{def:L}
L =
\begin{pmatrix} \DK & X_p\comp K \end{pmatrix}\,,
\end{equation}
and we assume that $\mathrm{rank}\, L(\theta)=n$ for every $\theta \in \TT^d$.
Notice that, while $\DK$ parameterizes the tangent bundle of $\cK$, $T\cK$ , $L$
parameterizes a subbundle $\cL$ of rank $n$ of the bundle $\Tan_{\torus} \mani$.
We refer to $L$ as the tangent frame (attached to $K$).

One can use the geometric structure of the problem to complement the above
frame.  From the several choices one could do (see e.g. \cite{HaroCFLM16} for a
discussion), we consider here the one  that is  specially tailored for a
compatible triple. In particular, we define $N:\TT^d \rightarrow \RR^{2n \times
n}$  as
\begin{equation*}\label{def:N}
N = J\comp K\: L\: B
\end{equation*}
where
\begin{equation*}\label{def:B}
B= G_L^{\text{-}1}.
\end{equation*}
Notice that $\N$ parameterizes another subbundle $\cN$ of rank $n$ of $\Tan_{\torus} \mani$.
We refer to $N$ as the normal frame (attached to $K$).

Finally, we define  the matrix-valued map $P:\TT^d \rightarrow \RR^{2n\times
2n}$ as the juxtaposition of matrix-valued maps $L,N:\TT^d \rightarrow
\RR^{2n\times n}$:
\begin{equation}\label{def:P}
P =
\begin{pmatrix} L  & N \end{pmatrix}.
\end{equation} 
We refer to $P$ as an (adapted) frame (attached to $\K$).

We  define then the torsion of the parameterization $\K$ (with respect to the
Hamiltonian $\H$) to be the matrix-valued map $T:\TT^d\to \RR^{n\times n}$
defined as 
\begin{equation}\label{def:T2}
T = N^\ttop\ T_h\comp K\ N,
\end{equation}
where $T_h:\mani\to \RR^{2n\times 2n}$  is the torsion of the Hamiltonian $\H$, defined as
\begin{equation*}\label{def:Th}
\begin{split}
T_h 
&= \Omega \left(  \Dif X_h + \Dif J [X_h] \ J + J \Dif X_h  J  \right) \\
& =\Omega  \ \Dif X_h - J^\ttop \Omega  \ \Dif X_h J + J^\ttop \Omega \ \Dif J [X_h],
\end{split}
\end{equation*}
where we use that $J^2= -I_{2n}$, so $\Dif J [X_h] J = - J \Dif J [X_h]$, and 
$G= -\Omega J= J^\ttop \Omega$. Notice that $T_h$ is symmetric:
\[
\begin{split}
T_h - T_h^\ttop 
& =
\Omega  \ \Dif X_h - J^\ttop \Omega  \ \Dif X_h J + J^\ttop \Omega \ \Dif J [X_h]
\\
& \phantom{=}
+  (\Dif X_h)^\ttop \Omega  - J^\ttop  (\Dif X_h)^\ttop \Omega J +  \Dif J^\ttop [X_h]\ \Omega J\\
&
= -\Dif \Omega[X_h]  + J^\ttop  \Dif \Omega[X_h] J +  J^\ttop \Omega \ \Dif J [X_h] + \Dif J^\ttop [X_h]\ \Omega J \\
& 
= -\Dif \Omega[X_h]  + \Dif(J^\ttop \Omega J)[X_h]
= O_{2n}.
\end{split}      
\]
As a result, the torsion $T$ of the parameterization $K$ is symmetric.

The use of the frame $P$ has several advantages. Among them, it produces a
natural and geometrically meaningful non-degeneracy condition (twist condition,
that is the invertibility of the average of the torsion $T$) in the KAM theorem,
and, most importantly, when the torus is invariant, it reduces the linearized
dynamics to a block-triangular form. 

\begin{proposition}\label{prop:reducibility}
If $K$ is invariant for $\Xh$ with ergodic frequency $\omega$, then:
\begin{enumerate}
\item $P$ is symplectic:
\begin{equation*}\label{eq:Psym}
P^\ttop \OK\: P = \Omega_0,
\end{equation*}
and, in particular, $L$ and $N$  parameterize complementary Lagrangian bundles
($\Tan_{\torus} \mani= \cL \oplus \cN$); 
\item $P$ reduces $\Dif \Xh\comp K$ to the block-triangular form $\Lambda:\TT^d
\rightarrow \RR^{2n\times 2n}$ given by
\begin{equation} \label{def:Lambda}
\Lambda
= \begin{pmatrix}
O_n &  T \\ 
O_n  & O_n
\end{pmatrix},
\end{equation}
where $T$ is defined in \eqref{def:T2}, so 
\begin{equation}\label{eq:reducibility}
\Dif \Xh\comp K\: P + \Loper P
= P\: \Lambda.
\end{equation}
\end{enumerate}
\end{proposition}

\begin{proof} The facts that $L$ is invariant, i.e. $\Xoper{L} = O_{2n \times
n}$, and Lagrangian, i.e. $\Omega_L= O_{n}$, follow directly from the invariance
of $K$ and the fact that the frequency $\omega$ is ergodic. The construction
leads to the Lagrangianity of $N$, i.e. $\Omega_N= O_{n}$, and the fact that $P$
is symplectic. The construction also leads to the reducibility property
\eqref{eq:reducibility} with 
\begin{equation*}
\label{def:T1}
T = N^\ttop \OK\: \Xoper{N}.
\end{equation*}
See e.g. \cite{HaroL19}. Finally,  $T= N^\ttop T_\H N$ follows from the
computation: 
\[
\begin{split}
\Loper   N 
& = (\Dif J)\comp K [\Loper   K] \ L B + J\comp K \ \Loper   L \ B +
J\comp K\  L \ \Loper   B \\
& = (\Dif J)\comp K [-X_h\comp K]\ L B - J\comp K\ \Dif X_h\comp K L \ B +
J\comp K \ L \ \Loper   B \\ 
&= (\Dif J)\comp K [X_h\comp K] \ J\comp K \ N + J\comp K\ \Dif X_h\comp K \ J\comp K N +
J\comp K \ L \ \Loper   B.
\end{split}
\]
\end{proof}

\begin{remark}
The torsion measures the symplectic area determined by the normal bundle and its
infinitesimal displacement.  Notice that, in the present paper, the torsion
involves geometrical and dynamical properties of both the torus and the first
integrals.
\end{remark}

\subsection{A (modified) quasi-Newton method}\label{ssec:quasi Newton}

In this section we outline a (modified) quasi-Newton method to obtain a solution
of the invariance equation \eqref{eq:inv:fv} from an initial
approximation. Sufficient conditions of the convergence of the method are
provided in Theorem \ref{thm:KAM}, whose proof is detailed in
Section~\ref{sec:proof KAM}. We focus here on the geometry of the method. 

Assume then we are given a parameterization $K:\TT^d\to\mani$. The  error of the
invariance is $E:\TT^d \to \RR^{2n}$, defined as
\[
E =  \Xh\comp K + \Loper  K.
\]

We may obtain a new parameterization 
\[
\bar K= K +\DeltaK
\]
by considering the linearized equation
\begin{equation}\label{eq:newton}
\Dif \Xh\comp K\: \DeltaK+ \Loper  \DeltaK
= - E.
\end{equation}
If the error $E$ is sufficiently small and we obtain a good enough
approximation of the solution $\DeltaK$ of \eqref{eq:newton}, then $\bar K$
provides a new parameterization with a new error $\bar E$ which is
quadratically small in terms of $E$. To do so, following the nowadays standard
practice \cite{GonzalezJLV05,HaroCFLM16} we may resort to a frame $P$ (here the
one defined in \ref{def:P}), so by writing 
\begin{equation}\label{eq:DeltaK}
\DeltaK = P\:  \xi= L \xiL + N \xiN,
\end{equation}
$\xi= (\xiL, \xiN):\TT^d\to\RR^{n}\times\RR^{n}$ is the new unknown. We think
of the components of $\xi$ as the tangent and normal components
of the correction.  The new approximation is 
\begin{equation}\label{eq:new-aproximation-old}
\bar K= K + L \xiL + N \xiN.
\end{equation}
Following e.g. \cite{GonzalezHL13,HaroL19}, using \eqref{eq:DeltaK},
multiplying both sides of \eqref{eq:newton} by $\Omega_0^{\text{-}1} P^\ttop
\OK\:$ (an approximation of $P^{\text{-}1})$, and skipping second order small
terms (in the error $E$ and correction $\xi$) one reaches  the block-triangular
system
\begin{equation}\label{eq:Lambda-eq}
\Lambda \: \xi +  \Loper   \xi \\ 
= {\eta},
\end{equation}
(see \eqref{def:Lambda} and \eqref{def:T2}), where 
\begin{equation}\label{def:eta}
	\eta= -\Omega_0^{\text{-}1} P^\ttop \OK\: E.
\end{equation}
We think of the components of $\eta= (\etaL,
\etaN):\TT^d\to\RR^{n}\times\RR^{n}$ as the tangent and normal components  (of
the negative) of the error $E$. 

Inspired by \cite{Villanueva17}, we may also consider  the new approximation as
\begin{equation}\label{eq:new-aproximation-new}
\Kn= \Phi_\xiLXp\comp (K  + N\xiN) \comp (\id + \xiLDK)
\end{equation}
where the unknowns are $\xiL= (\xiLDK,\xiLXp):\TT^d\to\RR^{d}\times\RR^{n-d}$
and $\xiN: \TT^d\to \RR^{n}$ and $\id$ is the identity map on $\TT^d$.  By
Taylor expanding the previous expression, we get
\[
\begin{split}
\Kn 
& = K + N \xiN + \DK\: \xiLDK + \Dif_s\Phi_0\comp K\: \xiLXp + \mbox{\em hot} \\
& = K + N \xiN + \DK\: \xiLDK + \Xp\comp K\: \xiLXp + \mbox{\em hot} \\
& = \bar K + \mbox{\em hot},
\end{split}
\]
where, as usual, $\mbox{\em hot}$ stands for {\em higher order  terms}. Hence,
the approximations \eqref{eq:new-aproximation-new} and
\eqref{eq:new-aproximation-old} differ in quadratically small terms, and the
correction terms $\xi= (\xiL,\xiN)$ may be computed by solving the triangular
system \eqref{eq:Lambda-eq}.

It turns out that the  triangular system \eqref{eq:Lambda-eq}, requires to
solve two cohomological equations consecutively. More specifically, the system
is 
\begin{eqnarray}
 \label{eq:stepL} T \xiN + \Loper   \xiL & = & \eta^L, \\
 \label{eq:stepN} \Loper   \xiN & = & \eta^N,
\end{eqnarray}
where 
\begin{eqnarray*}
\label{eq:defetaL} \etaL & = & - N^\ttop \OK\: E, \\
\label{eq:defetaN} \etaN & = & \phantom{-} L^\ttop \OK\: E,
\end{eqnarray*}
and the torsion $T$ is given 
by \eqref{def:T2}. 

In the solution of the triangular system it is crucial the fact that the
average of the normal component of the error, $\etaN$, is zero. Notice that 
\[
\etaN=
\begin{pmatrix}
(\DK)^\ttop \OK\: E \\
(X_p\comp K)^\ttop \OK\: E
\end{pmatrix} 
= 
\begin{pmatrix}
(\Dif(\H\comp K))^\ttop - \OmegaDK\omega \\
\Dif (p\comp K) \omega 
\end{pmatrix}.
\]
The fact that  $\aver{\OmegaDK} = O_d$ follows directly from the exact
symplectic structure, since $K^*\sform={\rm d}( K^{*}\aform)$,  from where 
$\aver{\etaN}= 0$  follows immediately. See \cite{HaroL19}.

Quantitative estimates for the solutions of such equations (under Diophantine
conditions of $\omega$) are obtained by applying R\"ussmann estimates, see
Lemma~\ref{lem:Russmann}.  Here we just mention that if we denote by $\Roper v$
the only zero-average solution $u$ of equation $\Loper u= v - \aver{v}$,  then,
since $\aver{\etaN}= 0$ (see the {\em compatibily condition} above) and 
$\aver{T}$ is invertible (the so-called {\em twist condition}),  $\xi=
(\xiL,\xiN)$ is given by
\begin{eqnarray*}
\label{eq:defxiN} \xiN &= &\aver{T}^{\text{-}1} \aver{\etaL-T\Roper\etaN}\; + \Roper \etaN, \\
\label{eq:defxiL} \xiL & = &\Roper(\etaL - T \xiN),
\end{eqnarray*}
is the solution of the system \eqref{eq:stepL},\eqref{eq:stepN} with
$\aver{\xiL}= 0_n$. While the average of the normal correction is selected  to
solve the equation for $\xiL$, and it is $\averxiN= \aver{\xiN}=
\aver{T}^{\text{-}1} \aver{\etaL-T\Roper\etaN}$, 
we have the freedom of choosing any
value for the average of the tangent correction, $\aver{\xiL}= \averxiL \in
\RR^n$. This is related with the freedom to select a particular generator of
the $n$-dimensional torus,  which is a $d$-dimensional torus, and a particular
phase, see Remark~\ref{rem:unicity}.
For the sake of simplicity, we select the solution with  $\averxiL= \aver{\xiL}=0_n$.

Recapitulating,  for future reference we describe one step of the (two) quasi-Newton
methods described above as follows: 
\begin{enumerate}
\item Compute the error of invariance: 
\begin{equation*}
\label{eq:error1}
E =  \Xh\comp K + \Loper  K.
\end{equation*}
\item Compute the tangent and normal frames to the torus:
\begin{eqnarray*}
L &=& \begin{pmatrix} \DK & X_p\comp\K\end{pmatrix},\\
N &=& J\comp K\: L\: B,
\end{eqnarray*}
where 
\begin{equation*}
B= (L^\ttop G\comp K\: L)^{\text{-}1}.
\end{equation*}
\item Compute the tangent and normal components of the error of invariance:
\begin{eqnarray*}
\label{eq:defetaL1} \etaL & = & - N^\ttop \OK\: E, \\
\label{eq:defetaN1} \etaN & = & \phantom{-} L^\ttop \OK\: E.
\end{eqnarray*}
\item Compute the torsion: 
\begin{equation*}\label{eq:T2}
T= N^\ttop\ T_h\comp K\ N,
\end{equation*}
where  
\begin{equation*}
T_h = \Omega \left(  \Dif X_h + \Dif J [X_h] \ J + J \Dif X_h  J  \right).
\end{equation*}
\item Compute the tangent and normal components of the correction, provided that 
$\aver{T}$ is invertible:
\begin{eqnarray*}
\label{eq:defxiN0} \averxiN &=& \aver{T}^{\text{-}1} \aver{\etaL-T\Roper\etaN}, \\
\label{eq:defxiN1} \xiN &= &\averxiN\; + \Roper \etaN, \\
\label{eq:defxiL1} \xiL & = &\Roper(\etaL - T \xiN),
\end{eqnarray*}
where $\Roper$ is the R\"ussmann operator (see Lemma~\ref{lem:Russmann}).
\item Compute the new approximation as:
\begin{itemize}
\item (quasi-Newton method)
\begin{equation*}
\bar K=  K  + L \xiL + N \xiN;
\end{equation*}
\item (modified quasi-Newton method)
\begin{equation*}
\Kn= \Phi_\xiLXp\comp (K  + N\xiN) \comp (\id + \xiLDK).
\end{equation*}
\end{itemize}
\end{enumerate}

As it is usual in the a-posteriori approach to KAM theory, the argument
consists in refining an initial approximation $K$ by means of the iterative
method and proving the convergence to a solution of the invariance equation.
The proof of the corresponding theorem to the quasi-Newton method is in
\cite{HaroL19}. The goal of this paper is proving such a result
for the modified version. 

\begin{remark}
We have seen that the difference of the two approaches
\eqref{eq:new-aproximation-old} and \eqref{eq:new-aproximation-new} are
quadratically small. As we will see, the way the correction $\xiL$ is included
in \eqref{eq:new-aproximation-new} results in a better behavior of the
analyticity properties. This is the main idea in \cite{Villanueva17}, in which
the invariance equation is replaced by three conditions which are altogether
equivalent to invariance.  Instead, we will keep the invariance equation
formulation. 
\end{remark}

\begin{remark}
An important feature of the quasi-Newton method is that it can be implemented in a computer. 
See e.g. \cite{HaroCFLM16} for some details of implementations in similar contexts,
using FFT. In this respect, it seems that the approach
\eqref{eq:new-aproximation-old} is better than \eqref{eq:new-aproximation-new},
since the second involves compositions of periodic functions that, in general,
are approximated by truncated Fourier series.  Composition of periodic
functions is much harder computationally than multiplications of periodic
functions, that can be done using Fast Fourier Transform. 
\end{remark}

\begin{remark}
Approaches as \eqref{eq:new-aproximation-old}  have been implemented as
computer assisted proofs \cite{FiguerasHL17}, for invariant KAM tori for exact
symplectic maps. We think that the new approach
\eqref{eq:new-aproximation-new} could result in better posed analytical bounds,
that could improve the efficiency of the computer assisted proofs.
\end{remark}

\subsection{Analytic setting }\label{ssec:analytic setting}

The proof of the convergence of the algorithm is presented in the analytic
category.  Hence, we work with real analytic functions defined in complex
neighborhoods of real domains.  We consider the sup-norms of
(matrix-valued) analytic maps and their derivatives (see the notation in
Section \ref{ssec:basic notation}).  That is, for $f: \U\subset \CC^m \to \CC$,
we consider 
\[
\norm{f}_\U= \sup_{x\in \U} |f(x)|,
\]
and
\[
\norm{\Dif^r f}_\U= \sum_{\ell_1,\dots,\ell_r} \Norm{\frac{\partial^r
f}{\partial x_{\ell_1}\dots\partial x_{\ell_r}}}_\U, \]
that could be infinite. 
For $M:\U \subset \CC^m\to \CC^{n_1\times n_2}$, we consider the norms  
\[
\norm{M}_\U= \max_{i= 1,\dots,n_1} \sum_{j= 1,\dots,n_2} \norm{M_{i,j}}_\U \,,
\]
\[
\norm{\Dif^r M}_\U= \max_{i= 1,\dots,n_1} \sum_{j= 1,\dots,n_2} \Norm{\Dif^r M_{i,j}}_\U\,,
\]
and notice, of course, that the norms $\norm{M^\ttop}_\U$ and $\norm{\Dif^r
M^\ttop}_\U$ are obtained simply by interchanging the role of the indices $i$
and $j$.

The above norms present Banach algebra-like properties.  For example, given $r$
analytic functions $v_1,\dots, v_r: \U\to \CC^m\simeq \CC^{m\times 1}$,
the function 
$\Dif^r M [v_1,\dots, v_r]: \U\subset \CC^m\to \CC^{n_1\times n_2}$ defined as 
\[
\Dif^r M [v_1,\dots, v_r](x)= \Dif^r M(x) [v_1(x),\dots, v_r(x)]
\]
is also analytic and satisfies
\begin{align*}
\norm{\Dif^r M [v_1,\dots, v_r]}_\U 
\leq \Norm{\Dif^r M}_\U \ \norm{v_1}_\U\cdots \norm{v_r}_\U\,.
\end{align*}
There is also a similar bound for the action of the transpose:
\begin{align*}
\norm{(\Dif^r M [v_1,\dots, v_r])^\ttop}_\U 
& \leq \Norm{\Dif^r M^\ttop}_\U \ \norm{v_1}_\U\cdots \norm{v_r}_\U\,.
\end{align*}
In addition, given $M_1: \U \subset \CC^m\to \CC^{n_1\times n_3}$ and $M_2: \U
\subset \CC^m\to \CC^{n_3\times n_2}$, we have 
\[
\norm{M_1 M_2}_\U \leq 
\norm{M_1}_\U \norm{M_2}_\U \, ,
\]
and 
\[
\norm{\Dif(M_1 M_2)}_\U \leq 
\norm{\Dif M_1}_\U \norm{M_2}_\U + \norm{M_1}_\U \norm{\Dif M_2}_\U  \, .
\]

In particular, if $f: \U\subset \CC^m \to \CC^{n_1}$, we may take $n_2= 1$ and
consider the matrix constructions just made. One recover bounds such as \[
	\norm{\Dif f\:V}_\U \leq \norm{\Dif f}_\U \norm{V}_\U,
\]
or 
\[
 \norm{(\Dif f\: V)^\ttop}_\U \leq \norm{(\Dif f)^\ttop}_\U \norm{V^\ttop}_\U, 
	\quad \norm{(\Dif f\: V)^\ttop}_\U \leq \norm{\Dif f^\ttop}_\U \norm{V}_{\U}, 
\]
where $V:\U\subset \CC^m \to\CC^{n_1\times n_3}$. Notice we obtain two possible upper bounds for $ \norm{(\Dif f\: V)^\ttop}_\U$.

\subsubsection*{Spaces of periodic real-analytic functions}
The particular case of real-analytic periodic functions deserves some
additional definitions and comments.  We denote by $\Anal(\TT^d_\rho)$ the
Banach space of holomorphic functions $u:\TT^d_\rho \to \CC$, that can be
continuously extended to $\bar\TT^d_\rho$, and such that
$u(\TT^d) \subset \RR$ (real-analytic), endowed with the norm
\[
\norm{u}_\rho = \norm{u}_{\TT^d_\rho}= \max_{|\im\theta|\leq\rho} |u(\theta)|\,.
\]
We also denote by $\Anal_{C^r}(\TT^d_\rho)$ the Banach space of holomorphic
functions $u:\TT^d_\rho \to \CC$ whose partial derivatives up to order $r$ can
be continuously extended to $\bar\TT^d_\rho$, and such that $u(\TT^d) \subset
\RR$,  endowed with the norm
\[
\norm{u}_{\rho,C^r} = \max_{k= 0,\dots, r}  \norm{\Dif^k u}_{\TT^d_\rho}.
\]

As usual in the analytic setting, we use {\em Cauchy estimates} to
control the derivatives of a function. 
Given $u \in \Anal(\TT^d_\rho)$, with $\rho>0$, then for any $0<\delta<\rho$
the partial derivative $\pd u/\pd {\theta_\ell}$ belongs to $\Anal(\TT^d_{\rho-\delta})$
and we have the estimates
\[
\Norm{
\frac{\pd u}{\pd \theta_{\ell}}}_{\rho-\delta} \leq \frac{1}{\delta}\norm{u}_\rho, 
\qquad
\Norm{\Dif u}_{\rho-\delta} \leq \frac{d}{\delta}\norm{u}_\rho, 
\qquad
\Norm{(\Dif u)^\ttop}_{\rho-\delta} \leq \frac{1}{\delta}\norm{u}_\rho.
\]
The above definitions and estimates extend naturally to matrix-valued maps, that is,
given  $M: \TT^d_\rho \to \CC^{n_1\times n_2}$, with components in $ \Anal(\TT^d_\rho)$, we have
\[
\norm{\Dif M}_{\rho-\delta} 
=   \max_{i=1,\ldots,n_1} \sum_{j= 1,\dots, n_2}  \Norm{\Dif M_{i,j}}_{\rho-\delta} 
\leq \frac{d}{\delta} \norm{M}_{\rho}.
\]
A direct consequence is that $\norm{\Dif M^\ttop}_{\rho-\delta} \leq
\frac{d}{\delta} \norm{M^\ttop}_{\rho} \leq  \frac{d\: n_1}{\delta}
\norm{M}_{\rho}$.

As it was mentioned in Section \ref{ssec:basic notation}, the operators $\Dif$
and $(\cdot)^\ttop$ do not commute. In particular,
given a real analytic vector function $w: \TT^d_\rho \to \CC^n\simeq \CC^{n\times 1}$, we have:
\[
	\norm{\Dif w}_{\rho-\delta} \leq \frac{d}{\delta} \norm{w}_\rho,\quad 
	\norm{\Dif w^\ttop}_{\rho-\delta} \leq \frac{d}{\delta} \norm{w^\ttop}_\rho \leq \frac{n d}{\delta} \norm{w}_\rho,\quad
	\norm{(\Dif w)^\ttop}_{\rho-\delta} \leq \frac{n}{\delta}\norm{w}_\rho.
\]

\subsubsection*{Small divisors equations and Diophantine vectors}
Another ingredient in KAM theory are the estimates of the solutions of the {\em
small divisors equations}. Given $\omega \in \RR^d$, ergodic, and a
real-analytic periodic function $v\in \Anal(\TT^d_\rho)$, we consider the
equation \begin{equation}\label{eq:calL}
\Loper   u = v- \aver{v},
\end{equation}
where $\Loper$ is defined in \eqref{def:Loper}.  Expanding in Fourier series,
the only zero-average solution $u= \Roper   v$ of \eqref{eq:calL} is be
\begin{equation}\label{eq:small:formal}
\Roper  v(\theta) = \sum_{k \in \ZZ^d \backslash \{0\} } \hat u_k \ee^{2\pi
\ii k \cdot  \theta}, \qquad \hat u_k = \frac{-\hat
v_k}{2\pi \ii \: k \cdot \omega},
\end{equation}
and all the other solutions of 
\eqref{eq:calL} are of the form $u= \hat u_0 + \Roper  v$ with $\hat u_0\in\RR$.

The convergence of the expansion \eqref{eq:small:formal} is implied by a
Diophantine condition on $\omega$.  Specifically, for given $\gamma >0$ and
$\tau \geq d-1$, we denote the set of Diophantine vectors 
\begin{equation*}\label{eq:def:Dioph}
\Dioph{\gamma}{\tau}{d} =
\left\{
\omega \in \RR^d \, : \,
\abs{k \cdot \omega} \geq  \frac{\gamma}{|k|_1^{\tau}}
\,, 
\forall k\in\ZZ^d\backslash\{0\}
\right\},
\end{equation*}
where $|k|_1 = \sum_{i= 1}^d |k_i|$, and $\omega\in\Dioph{\gamma}{\tau}{d}$.
Sharp estimates are provided by the following lemma \cite{FiguerasHL17},
in which we also include sharp estimates of partial derivatives of the
solution, in combination with Cauchy estimates.

\begin{lemma}[R\"ussmann estimates]\label{lem:Russmann}
Let $\omega \in \Dioph{\gamma}{\tau}{d}$ for some $\gamma>0$ and $\tau \geq
d-1$.  Then, for any $v \in \Anal(\TT^d_\rho)$, with $\rho>0$, there exists a
unique zero-average solution of $\Loper   u = v -\aver{v}$, denoted by
$u=\Roper  v$, such that $u \in \Anal(\TT^d_{\rho-\delta})$ for any $\delta\in
]0,\rho]$.  Moreover,
\begin{equation*}\label{eq:Russmann}
\norm{\Roper   v}_{\rho-\delta} \leq \frac{\CR(\delta)}{\gamma \delta^\tau}
\norm{v}_\rho\,,
\end{equation*}
where $\CR(\delta)$ depends on $\omega$ and $\delta$, and it is bounded from
above by a constant $\CRb$ than depends only on $d$ and $\tau$.  More
concretely, for any $m>0$, 
\begin{equation*}\label{def:CR}
\begin{split}
\CR(\delta) & := \sqrt{\gamma^2\delta^{2\tau}2^d
\sum_{k\in\ZZ^d\backslash\{0\}}\frac{e^{-4\pi|k|_1\delta}}{|k\cdot\omega|^2}}
\\
& \leq
\sqrt{\gamma^2\delta^{2\tau}2^d
\sum_{0<|k|_1\leq m}\frac{e^{-4\pi|k|_1\delta}}{|k\cdot\omega|^2}+
\frac{
2^{d+1-2\tau}\zeta(2, 2^{\tau})}{\pi^{2\tau}}\int_{4\pi\delta(m+1)}^\infty
u^{2\tau}e^{-u}du} \quad =: \CR(\delta,m)
\\
& \leq \sqrt{
2^{d+1-2\tau}\zeta(2, 2^{\tau})\pi^{-2\tau-2}\Gamma(2\tau+1)
}
\quad =:\CRb,
\end{split}
\end{equation*}
where $\zeta(a,b)= \sum_{j\geq 0} (b+j)^{-a}$ is the Hurtwitz zeta function. 
\end{lemma}

\begin{proof}
There results follows from the classical results in \cite{Russmann75,Russmann76a}, 
where a uniform bound (independent of $\delta$) is obtained. We refer
to \cite{FiguerasHL17} for sharp non-uniform computer-assisted
estimates (in the discrete case) of the form $\CR=\CR(\delta)$,
which represent a substantial
advantage in order to apply the result to particular problems. Adapting these
estimates to the continuous case is straightforward. Also, we refer to \cite{FiguerasHL18}
for a numerical quantification of these estimates and for an analysis of
the different sources of overestimation.
\end{proof}

\begin{remark}
In applications, for a given $\delta\in ]0,\rho]$ one selects $m$ big enough so
that the integral term in $\CR(\delta,m)$ or $\CRd(\delta,m)$ is small compared
with the preceeding sum of terms up to order $m$.
\end{remark}

Along the proof, we  encounter situations in which we have to combine Cauchy
and R\"ussmann estimates. The following lemma gives sharp bites to perform such
combined bounds. 

\begin{corollary}[R\"ussmann-Cauchy estimates]\label{lem:Russmann-Cauchy}
Let $\omega \in \Dioph{\gamma}{\tau}{d}$
for some $\gamma>0$ and $\tau \geq d-1$.
Then, for any $v, w \in \Anal(\TT^d_\rho)$, with $\rho>0$, for any $\ell= 1,\dots, d$, 
$\delta\in]0,\rho]$ and $m>0$:
\begin{equation*}\label{eq:CRderivada}
\norm{\frac{\pd}{\pd\theta_\ell} (v\: \Roper w)}_{\rho-\delta} 
\leq \frac{\CR^1(\delta)}{\gamma \delta^{\tau+1}} \norm{v}_\rho \norm{w}_\rho,
\end{equation*}
where
\begin{equation}\label{def:CRderivada}
\begin{split}
\CRd(\delta)  & := \frac{(\tau+1)^{\tau+1}}{\tau^\tau}\ \CR\!\left(\frac{\tau}{\tau+1}\delta\right) 
\\ 
& \leq \frac{(\tau+1)^{\tau+1}}{\tau^\tau}\ \CR\!\left(\frac{\tau}{\tau+1}\delta, m \right) \quad =: 	\CRd(\delta, m)
\\
& \leq  \frac{(\tau+1)^{\tau+1}}{\tau^\tau}\ \CRb \quad =: \CRdb.
\end{split}
\end{equation}
\end{corollary}

\begin{proof}
The estimates for the derivatives of the solution $u$ follow from applying
Cauchy and R\"ussmann estimates with bites  $\delta-\hat\delta$ and
$\hat\delta$, respectively, and  choosing $\hat\delta$ to maximize
$(\delta-\hat\delta) {\hat\delta}^\tau$.  This happens for $\hat\delta=
\frac{\tau}{1+\tau}\delta$.
\end{proof}

\begin{remark}
If one applies R\"ussmann and Cauchy estimates with bites  $\delta/2$ to get
the upper bound \eqref{def:CRderivada}, then one obtains $\CR^1(\delta)=
2^{\tau+1} \CRb$. Notice the factor in \eqref{def:CRderivada} is 
\[
   \frac{(\tau+1)^{\tau+1}}{\tau^\tau} \leq 2^{\tau+1}.
\]
\end{remark}

Again, the above definitions, constructs and estimates extend naturally to
matrix-valued maps.

\subsection{The KAM theorem}\label{ssec:KAM theorem}

In this subsection, we present an a-posteriori KAM theorem for $d$-dimensional
quasi-periodic invariant tori in Hamiltonian systems with $n$
degrees-of-freedom that have $n-d$ additional first integrals in involution.
The hypotheses in Theorem \ref{thm:KAM} are tailored to be verified with a
finite amount of computations.

\def\smax{r}

\begin{theorem}[KAM theorem with first integrals]\label{thm:KAM}
Let  $(\sform,\gform,\J)$ be  a compatible triple on the open set
$\mani\subset\RR^{2n}$, where the symplectic form is exact: $\sform=\dif
\aform$.  Let $\H:\mani\to \RR$ be a Hamiltonian for which there exists a
moment map $\P:\mani\to \RR^{n-d}$ whose components and $\H$ are pairwise in
involution functionally independent functions, being $\Phi:\domPhi\subset
\RR^{n-d}\times \mani \to \mani$ the corresponding moment flow.  Let
$\omega\in \Dioph{\gamma}{\tau}{d}$ be a Diophantine vector, for some constants
$\gamma >0$ and $\tau \geq d-1$.  Let $K:\TT^d\to \mani$ be a parameterization.
We assume that the following hypotheses hold.
\begin{itemize}[leftmargin=*]
\item [$H_1$]
The geometric objects $\sform,\gform,\J,\aform$, the Hamiltonian $\H$, and the
moment map $\P$ can be analytically extended to an open complex set
$\cmani\subset \CC^{2n}$ covering $\mani$, and the moment flow  to an open
complex set $\cdomPhi\subset \CC^{n-d}\times \cmani$ covering $\domPhi$.
Moreover, there exist constants $\cteOmega$, $\cteG$, $\cteJ$, $\cteJT$,
$\cteDOmega$, $\cteDG$, $\cteDJ$, $\cteDJT$, $\cteXH$, $\cteDXH$, $\cteDXHT$,
$\cteDDXH$, $\cteTh$, $\cteDTh$, $\cteXp$, $\cteDXp$, $\cteXpT$, $\cteDXpT$,
and $\cteDPhi$ such that:
\begin{itemize}
\item the matrix representations $\Omega,G,J:  \cmani \rightarrow \CC^{2n \times 2n}$ of $\sform$,$\gform$,$\J$ satisfy: 
\begin{align*}
\norm{\Omega}_{\cmani} \leq \cteOmega, 
&&\norm{G}_{\cmani} \leq \cteG, 
&& \norm{J}_{\cmani} \leq \cteJ, 
&& \norm{J^{\top}}_{\cmani} \leq \cteJT,
\\
\norm{\Dif \Omega}_{\cmani} \leq \cteDOmega,  
&& \norm{\Dif G}_{\cmani} \leq \cteDG, 
&& \norm{\Dif J}_{\cmani} \leq \cteDJ,
&& \norm{\Dif J^\ttop}_{\cmani} \leq \cteDJT;
\end{align*}
\item the Hamiltonian vector field 
$\Xh: \cmani \rightarrow \CC^{2n}$ and its torsion $T_\H:\cmani \rightarrow \CC^{2n\times 2n}$, satisfy:
\begin{align*}
& \norm{\Xh}_{\cmani} \leq \cteXH, 
&& \norm{\Dif \Xh}_{\cmani} \leq \cteDXH,
&& \norm{(\Dif \Xh)^\ttop}_{\cmani} \leq \cteDXHT
& \norm{\Dif^2 \Xh}_{\cmani} \leq \cteDDXH \\
& \norm{T_\H}_{\cmani} \leq \cteTh,
&& \norm{\Dif T_\H}_{\cmani} \leq \cteDTh;
\end{align*}
\item the moment vector fields $\Xp: \cmani \rightarrow \CC^{2n \times (n-d)}$ 
and the moment flow $\Phi: \cdomPhi\to \cmani$, satisfy:
\begin{align*}
& \norm{X_p}_{\cmani} \leq \cteXp, && \norm{\Dif X_p}_{\cmani} \leq \cteDXp, && 
\\
& \norm{X_p^\ttop}_{\cmani} \leq \cteXpT, && \norm{\Dif X_p^\ttop}_{\cmani} \leq \cteDXpT &&\norm{\Dif_z\Phi}_{\cdomPhi} \leq \cteDPhi.
\end{align*}
\end{itemize}

\item [$H_2$] 
There are $\smax>0$, an  open subset $\cmani_0 \subset \cmani$, and {\em condition numbers}  
$\sigmaDK$, $\sigmaDKT$, 
 $\sigmaB$, $\sigmaN$, $\sigmaNT$, and $\sigmaTinv$
such that:
\begin{itemize}
\item $K\in (\Anal_{C^1}(\TT^d_\rho))^{2n}$, with $0<\rho<\smax$,
is an embedding with $K(\bar\TT^d_{\rho})\subset \cmani_0$, whose averaged
torsion $\aver{T}$ is invertible and, moreover: \[
\norm{\DK}_{\rho} < \sigmaDK, \quad
\norm{\DKT}_{\rho} < \sigmaDKT,\quad
 \norm{\B}_\rho < \sigmaB,
 \]
 \[
\norm{\N}_{\rho} < \sigmaN, \quad
 \norm{\NT}_{\rho} < \sigmaNT, \quad
 \abs{\aver{T}^{\text{-}1}} < \sigmaTinv;
\]
We define $\sigmaL= \sigmaDK + \cteXp$ and $\sigmaLT= \max\{ \sigmaDKT,\cteXpT\}$, 
so that 
\[
\norm{\L}_{\rho} < \sigmaL, \quad
 \norm{\LT}_{\rho} < \sigmaLT,
\]
\item $\cdomPhi_0 := \{ s\in \CC^{n-d}  \ | \ |s|< \smax\} \times \cmani_0 \subset \cdomPhi$.
\end{itemize}

\item [$H_3$] We are given positive {\em control constants}  $\fake<1$, $\fakeetaN$ and $\fakeE<1$.
\end{itemize}
Under the above hypotheses, for each $\rho_\infty\in ]0,\rho[$  and $\delta\in ]0,(\rho-\rho_\infty)/3[$,
there exists a constant $\mathfrak{C}$ depending on $\rho, \rho_\infty, \delta$ and the constants introduced above such that, if the error of invariance
\begin{equation}\label{eq:invE}
E=\Xh\comp K + \Loper\K, 
\end{equation}
satisfies
\begin{equation}\label{eq:KAM:HYP}
\frac{\CtheoE}{\gamma\delta^{\tau+1}} \wnormeta < 1,
\end{equation}
where 
\[
	\eta^L=  - N^\ttop \OK\: E,\quad \etaN  =  \phantom{-} L^\ttop \OK\: E,
\]
then there exists an invariant torus
$\torus_\infty = K_\infty(\TT^d)$ with frequency $\omega$, satisfying
$K_\infty \in \Anal(\TT^{d}_{\rho_\infty})^{2n}$ and
\begin{align}
& \dist(\K_\infty(\TT^d_{\rho_\infty}),\partial\cmani_0) >0, \label{eq:distKinf} \\
& \norm{\DK_\infty}_{\rho_\infty} < \sigmaDK, \label{eq:DKinf} \\
& \norm{\DKT_{\!\,\infty}}_{\rho_\infty} < \sigmaDKT, \label{eq:DKTinf} \\
& \norm{\B_\infty}_{\rho_\infty} < \sigmaB, \label{eq:Binf} \\
& \norm{\N_\infty}_{\rho_\infty} < \sigmaN, \label{eq:Ninf} \\
& \norm{\NT_{\!\,\infty}}_{\rho_\infty} < \sigmaNT, \label{eq:NTinf} \\
& \abs{\aver{T_\infty}^{\text{-}1}} < \sigmaTinv. \label{eq:iTinf}
\end{align}
Furthermore, the objects are close to the original ones: there exist constants
$\CtheoDeltaK$, $\CtheoDeltaL$, $\CtheoDeltaLT$, $\CtheoDeltaB$,
$\CtheoDeltaN$, $\CtheoDeltaNT$ and $\CtheoDeltaiT$ (like $\CtheoE$,
given explicitly throughout the proof and summarized in
Appendix~\ref{app:constants}) such that
\begin{align}
& \norm{\K_\infty-\K}_{\rho_\infty} \leq 
\frac{\CtheoDeltaK}{\gamma\delta^\tau} \wnormeta,  \label{eq:DeltaKinf} \\
& \norm{\DK_\infty-\DK}_{\rho_\infty} \leq 
\frac{\CtheoDeltaDK}{\gamma\delta^{\tau+1}} \wnormeta,  \label{eq:DeltaDKinf} \\
& \norm{\DKT_{\,\!\infty}-\DKT}_{\rho_\infty} \leq 
\frac{\CtheoDeltaDKT}{\gamma\delta^{\tau+1}} \wnormeta,  \label{eq:DeltaDKTinf} \\
& \norm{\B_\infty-\B}_{\rho_\infty} \leq  
\frac{\CtheoDeltaB}{\gamma\delta^{\tau+1}} \wnormeta, \label{eq:DeltaBinf} \\
& \norm{\N_\infty-\N}_{\rho_\infty} \leq 
\frac{\CtheoDeltaN}{\gamma\delta^{\tau+1}} \wnormeta,  \label{eq:DeltaNinf} \\
& \norm{\NT_{\!\,\infty}-\NT}_{\rho_\infty} \leq 
\frac{\CtheoDeltaNT}{\gamma\delta^{\tau+1}} \wnormeta,  \label{eq:DeltaNTinf} \\
& \abs{\aver{T_\infty}^{\text{-}1}-\aver{\T}^{\text{-}1} } \leq  
\frac{\CtheoDeltaiT}{\gamma\delta^{\tau+1}} \wnormeta. \label{eq:DeltaiTinf}
\end{align}
\end{theorem}

\begin{remark}
Notice that if $\norm{\L}_{\rho} < \sigmaL, \norm{\LT}_{\rho} < \sigmaLT$ and
$\norm{B}_\rho < \sigmaB$, then $\norm{\N}_{\rho} < \cteJ \sigmaL \sigmaB$ and
$\norm{\NT}_{\rho} <   \sigmaB\sigmaLT\cteJT$ and, hence, if one takes
$\sigmaN\geq \cteJ \sigmaL \sigmaB$ and $\sigmaNT\geq \sigmaB\sigmaLT\cteJT$
the conditions for $\norm{\N}_{\rho}$ and $\norm{\NT}_{\rho}$ follow
inmediatelly. Our point is to provide maximum flexibility of the results to be
applied to specific problems. Similar controls could be also do for other
objects, such as $X_p\comp K$, $(X_p\comp K)^\ttop$, $\GL$ or $T$, leading to
similar formulae. 
\end{remark}

\begin{remark}
If $d=n$ then there are no additional first integrals
and we recover the classical KAM theorem for Lagrangian tori.
The corresponding estimates follow by taking zero the constants
$\cteXp=0$,
$\cteXpT=0$,
$\cteDXp=0$,
$\cteDXpT=0$,
$\cteDPhi= 1$.
\end{remark}

\begin{remark}
In the \emph{canonical case} we have $\Omega=\Omega_0$, $G = I_{2n}$,
and $J=\Omega_0$ and, hence, 
$\cteOmega=1$,
$\cteDOmega=0$, 
$\cteG=1$,
$\cteDG=0$, 
$\cteJ=1$,
$\cteDJ=0$,
$\cteJT=1$, and
$\cteDJT=0$.
\end{remark}

\begin{remark}
Theorem~\ref{thm:KAM} produces a $d$-dimensional isotropic invariant torus with
frequency $\omega\in \Dioph{\gamma}{\tau}{d}$, that generates an
$(n-d)$-parameters family of $d$-dimensional isotropic invariant tori with such
a frequency, foliating an $n$-dimensional invariant cylinder.  With the aid of
discounted Hamiltonians, one can produce also $n$-dimensional invariant
cylinders, see Lemma~\ref{lem:reduction lemma} or, if the moment map $p$
induces a Hamiltonian torus action, one can produce $n$-dimensional invariant
tori, see Remark~\ref{rem:Hamiltonian torus action}. These tori have
frequencies $(\omega,\omegap)\in \Dioph{\gamma}{\tau}{d}\times  \RR^{n-d}$,
thus one obtains analytic families of Lagrangian invariant tori.  
\end{remark}

\begin{remark}\label{rem:unicity}
The invariant $d$-dimensional tori are locally unique, meaning that if there is
another $d$-dimensional invariant torus with the same frequency nearby, then
both generate the same invariant cylinder. More specifically, the corresponding
parameterizations $K$ and $K'$, say, are related by 
\[
	K'(\theta)= \Phi_\beta(K(\theta+\alpha)),
\] 
for suitable $\alpha\in\RR^d$, $\beta\in\RR^{n-d}$ small.  As mentioned in
Remark~\ref{rem:uniqueness}, both indeterminacies (the phase $\alpha$ and the
displacement $\beta$) could be fixed by adding $n$ extra scalar equations to
the invariance equation.
\end{remark}

\begin{remark}
Theorem~\ref{thm:KAM} gives the convergence to a parameterization of an
invariant torus defined in a complex strip of size $\rho_\infty$ from  a
parameterization of an approximately invariant torus defined in a complex strip
of size $\rho$, through a sequence of approximations (given by a Newton-like
method) whose complex strips sizes are determined by the initial bite $3\delta$
(in the proof, the bites are given by a geometric sequence). In practical
situations, these are parameters that can be adjusted appropriately.
Heuristically, see Remark~\ref{rem:optimal strips}, a good choice is $\delta=
(\rho-\rho_\infty)/6$. Also, if one is not interested in controlling the domain
of analyticity of the invariant torus, can take $\rho_\infty= 0$. Since the
conditions on the initial parameterization are given by strict inequalities,
and the final constants depend continuously on all constants in the hypothesis
(including the sizes), then it follows that for a small enough final strip size
the conditions hold.
\end{remark}

\section{Proof of the KAM theorem}\label{sec:proof KAM}

In this section we present a fully detailed proof of Theorem \ref{thm:KAM}.
Hence, from now on we assume the setting and hyphoteses of Theorem
\ref{thm:KAM}. The proof consists in demonstrating the convergence of the
(modified) quasi-Newton method outlined in Subsection~\ref{ssec:quasi Newton}.
In Subsection~\ref{ssec:some lemmas} we present some estimates regarding the
control of some geometric and dynamical properties for an approximately
invariant torus.  In Subsection~\ref{ssec:iterative lemma} we produce
quantitative estimates for the objects obtained when performing one iteration
of the procedure. Finally, in Subsection~\ref{ssec:convergence} we discuss the
convergence of the (modified) quasi-Newton method.

\subsection{Some lemmas to control approximate geometric properties}
\label{ssec:some lemmas}

Here we present some estimates regarding the control of some geometric and
dynamical properties for an approximately invariant torus, including
approximate symplecticity of the corresponding frame, the control of the total
error of invariance by its tangent and normal projections, and the approximate
reducibility of the linearized dynamics.  We collect all constants appearing in
the bounds in Appendix~\ref{app:constants}, Table~\ref{tab:some lemmas}.

\subsubsection{Approximate symplecticity of the adapted frame}

We prove here that  the adapted frame $P:\bar\TT^d_\rho\to \CC^{2n\times2n}$
attached to the torus $\cK$ parameterized by $\K:\bar\TT^d_\rho\to \cmani_0$,
defined in \eqref{def:P}, induces an approximately symplectic vector bundle
isomorphism and, in particular, that the bundle $\cL$ framed by
$L:\bar\TT^d_\rho\to \CC^{2n\times n}$ given in \eqref{def:L} is approximately
Lagrangian.  See e.g. \cite{GonzalezJLV05,GonzalezHL13} for similar
considerations. An extra ingredient is that, following \cite{Villanueva17}, the
errors in the symplecticity of $P$ and Lagragianity of $L$ are controlled by
the normal component of the invariance error, $\etaN$.

The symplectic form on the bundle  $\cL$, is represented by the  the
anti-symmetric matrix-valued map $\Omega_L:\bar\TT^d_\rho \to \CC^{n \times
n}$, which is 
\begin{equation}
\label{def:OmegaL}
\Elag = 
\begin{pmatrix}
\OmegaDK & 
 (\Dif(p\comp K))^\ttop
\\
- \Dif(p\comp K) & O_{n-d}
\end{pmatrix},
\end{equation}
where we use the pairwise involution of the first integrals,  
\[
(X_p\comp K)^\ttop
\OK\; X_p\comp K	= O_{n-d},
\]
and the corresponding Hamiltonian vector fields to get
\[
(X_p\comp K)^\ttop
\OK\: \DK = -(\Dif p)\comp K \: \DK = 
- \Dif(p\comp K).
\]

\begin{lemma}\label{lem:approximate lagrangianity}
Let $\OmegaL:\bar\TT^d_\rho \to \CC^{n \times n}$ be the matrix-valued map
given by \eqref{def:OmegaL}.  

Then, $\aver{\OmegaL} = O_{n}$ and, in $\TT^d_\rho$, 
\begin{equation*}\label{eq:OmegaL}
\OmegaL = 
\begin{pmatrix} 
\Roper(\Dif\etaNDK - (\Dif\etaNDK)^\ttop)
&
-\Roper (\Dif \etaNXp)^\ttop
\\
\Roper \Dif \etaNXp
& 
O_{n-d}
\end{pmatrix}.
\end{equation*}
Moreover, for any $\delta\in ]0,\rho]$: 
\begin{equation}\label{CNOmegaL}
\norm{\OmegaL}_{\rho-\delta} \leq \frac{\CNOmegaL}{\gamma \delta^{\tau+1}}\norm{\etaN}_\rho
\end{equation}
and
\begin{equation}\label{CNOmegaN}
\norm{\OmegaN}_{\rho-\delta} \leq  \frac{\CNOmegaN}{\gamma\delta^{\tau+1}} \norm{\etaN}_\rho.
\end{equation}
\end{lemma}

\begin{proof}
The fact that  $\aver{\OmegaDK} = O_d$ follows directly from the exact
symplectic structure, since $K^*\sform={\rm d}( K^{*}\aform)$. The fact that
$\aver{(\Dif(p\comp K))}= O_{(n-d)\times d}$ is straightforward. 
Hence $\aver{\OmegaL} = O_{n}$  follows.

We claim that 
\begin{equation}\label{eq:LieOmegaL}
\Loper   \OmegaL = 
\begin{pmatrix}
\Dif\etaNDK - (\Dif\etaNDK)^\ttop & -(\Dif \etaNXp)^\ttop 
\\
\Dif \etaNXp & O_{n-d}
\end{pmatrix}.
\end{equation}
To do so, we first compute the action of $\Loper  $ on $\OmegaDK$ and, using
\eqref{eq:invE} and \eqref{eq:prop2killSK} we get
\begin{equation}
\label{eq:LieOmegaDK}
\Loper   \OmegaDK= 
(\Dif E)^\ttop \OK\, \DK
+
(\DK)^\ttop (\Dif \Omega) \comp K [E]\: \DK \\
 + (\DK)^\ttop \OK\: \Dif E.
\end{equation}
See \cite{HaroL19}. Inspired by \cite{Villanueva17}, we obtain
formula  
\begin{equation}
\label{eq:LieOmegaDK:etaNDK} 
\Loper \OmegaDK= \Dif\etaNDK - (\Dif\etaNDK)^\ttop
\end{equation}
from differentiating the projected error 
\[
	\etaNDK= (\DK)^\ttop \OK \: E, 
\]
and  symplecticity of $\sform$.  To do so using the matrix components, first
notice that \eqref{eq:LieOmegaDK} reads \[
\left(\Loper   \OmegaDK\right)_{i,j} 
   = \sum_{r,s} 
	\left(     \frac{\partial E_r}{\partial \theta_i} \Omega_{r,s}\comp K \frac{\partial K_s}{\partial\theta_j}
	         +\frac{\partial K_r}{\partial \theta_i} \Omega_{r,s}\comp K\frac{\partial E_s}{\partial\theta_j}
	\right) 
		+  
	\sum_{r,s,t} \frac{\partial K_r}{\partial \theta_i} \frac{\partial \Omega_{r,s}}{\partial z_t} \comp K E_t 
	\frac{\partial K_s}{\partial\theta_j}
\]
and also 
\[
\left(\Dif\etaNDK\right)_{i,j} 
=  \sum_{r,s} 
	\left(     \frac{\partial^2 K_r}{\partial\theta_i \partial\theta_j} \Omega_{r,s}\comp K E_s 
	        + \frac{\partial K_r}{\partial \theta_i} \Omega_{r,s}\comp K \frac{\partial E_s}{\partial\theta_j}
	\right)
	   +
	\sum_{r,s,t} \frac{\partial K_r}{\partial \theta_i} \frac{\partial \Omega_{r,s}}{\partial z_t} \comp K 
	\frac{\partial K_t}{\partial\theta_j}  
	E_s,
\]
for $i, j= 1,\dots, d$, where the indices $r,s,t$ run in $1,\dots, 2n$. Hence, 
\[
\begin{split}
\left(\Dif\etaNDK\right)_{i,j} - \left(\Dif\etaNDK\right)_{j,i}
&  = \phantom{+} \sum_{r,s} 
	\left(  \frac{\partial K_r}{\partial \theta_i} \Omega_{r,s}\comp K \frac{\partial E_s}{\partial\theta_j} -
	         \frac{\partial K_r}{\partial \theta_j} \Omega_{r,s}\comp K \frac{\partial E_s}{\partial\theta_i}
        \right)
\\
& \phantom{=} 
        + 
        \sum_{r,s,t} 
        \left( \frac{\partial K_r}{\partial \theta_i} \frac{\partial \Omega_{r,s}}{\partial z_t} \comp K 
	\frac{\partial K_t}{\partial\theta_j}  E_s -
	\frac{\partial K_r}{\partial \theta_j} \frac{\partial \Omega_{r,s}}{\partial z_t} \comp K 
	\frac{\partial K_t}{\partial\theta_i}  
	E_s
        \right)
\\
& = \phantom{+} \sum_{r,s} 
	\left(  \frac{\partial K_r}{\partial \theta_i} \Omega_{r,s}\comp K \frac{\partial E_s}{\partial\theta_j} +
	         \frac{\partial K_s}{\partial \theta_j} \Omega_{r,s}\comp K \frac{\partial E_r}{\partial\theta_i}
        \right)
\\
& \phantom{=} 
        + 
        \sum_{r,s,t} 
         \left( - \frac{\partial K_t}{\partial \theta_i} \frac{\partial \Omega_{s,t}}{\partial z_r} \comp K 
	\frac{\partial K_r}{\partial\theta_j}  E_s -
	\frac{\partial K_r}{\partial \theta_j} \frac{\partial \Omega_{r,s}}{\partial z_t} \comp K 
	\frac{\partial K_t}{\partial\theta_i}  
	E_s
        \right)
\\
& = \phantom{+} \sum_{r,s} 
	\left(  \frac{\partial K_r}{\partial \theta_i} \Omega_{r,s}\comp K \frac{\partial E_s}{\partial\theta_j} +
	         \frac{\partial K_s}{\partial \theta_j} \Omega_{r,s}\comp K \frac{\partial E_r}{\partial\theta_i}
        \right)
\\
& \phantom{=} 
        + 
        \sum_{r,s,t} 
         \frac{\partial K_t}{\partial \theta_i} \frac{\partial \Omega_{t,r}}{\partial z_s} \comp K 
	\frac{\partial K_r}{\partial\theta_j}  E_s 
\\
& = \left(\Loper   \OmegaDK\right)_{i,j}, 
\end{split}
\]
from where we obtain \eqref{eq:LieOmegaDK:etaNDK}.

Notice also that, since 
\[
\etaNXp= (X_p\comp K)^\ttop
\OK\: E = -(\Dif p)\comp K\: (X_h + \Loper K)=  -\Loper (p\comp K),
\] 
then 
\[
\Loper ((X_p\comp K)^\ttop\: \Omega\comp K\: \DK) = - \Loper \Dif(p\comp K)= \Dif \etaNXp,
\]
thus completing the proof of formula \eqref{eq:LieOmegaL}.

Finally, the quantitive estimate \eqref{CNOmegaL} follows from R\"ussmann
and Cauchy estimates from Corollary~\ref{lem:Russmann} applied
to the components of \eqref{eq:LieOmegaL}. In particular, \[
\norm{\OmegaDK}_{\rho-\delta} \leq \frac{2(d-1) \CRd(\delta)}{\gamma\delta^{\tau+1}} \norm{\etaN}_\rho.
\]
\end{proof}

With the previous lemma we control the approximate symplecticity of the frame $P$.

\begin{lemma}\label{lem:approximate symplecticity}
The matrix-valued map $P: \bar\TT^d_\rho \to \CC^{2n \times 2n}$, defined in \eqref{def:P}, is approximately symplectic, i.e.,
the simplecticity error map
\begin{equation*}\label{eq:Esym}
\Esym := P^\ttop \OK\: P-\Omega_0\,,
\qquad
\Omega_0 = 
\begin{pmatrix}
O_n & -I_n \\
I_n & O_n
\end{pmatrix}\,,
\end{equation*}
is small in the sense that, for any $\delta\in ]0,\rho]$:
\begin{equation}\label{Csym}
\norm{\Esym}_{\rho-\delta} \leq \frac{\Csym}{\gamma \delta^{\tau+1}}
\norm{\etaN}_\rho.
\end{equation}
\end{lemma}

\begin{proof}
To characterize the error in the symplectic character of the frame, we
compute
\begin{equation*}\label{eq:newEsym}
\Esym
= 
\begin{pmatrix}
L^\ttop 
\OK\: L & 
L^\ttop 
\OK\: N + I_n
\\
N^\ttop
\OK\: L - I_n & 
N^\ttop
\OK\: N
\end{pmatrix} 
= 
\begin{pmatrix}
\OmegaL  & O_n \\
O_n &  \OmegaN 
\end{pmatrix},
\end{equation*}
from which the result follows immediately. 
\end{proof}

\subsubsection{Relations between the invariance error and their tangent and normal components}

From the definitions of $\etaL$ and $\etaN$, we obtain easily their bounds controlled by $E$:
\begin{equation*}\label{bound:etaL-E}
\norm{\etaL}_\rho \leq \CNT \cteOmega \norm{E}_\rho,
\end{equation*}
\begin{equation*}\label{bound:etaN-E}
\norm{\etaN}_\rho \leq \CLT \cteOmega \norm{E}_\rho.
\end{equation*}
In other to control $E$ in terms of $\etaL$ and $\etaN$ we have to assume the invertibility of the frame $P$, which is a consequence of the approximate symplecticity, that is controlled in a narrower strip. We obtain the following lemma.

\begin{lemma}\label{lem:errors}
Assume that
\begin{equation}\label{eq:fake cond}
	\frac{\Csym}{\gamma \delta^{\tau+1}} \norm{\etaN}_\rho  < \fake <1.
\end{equation}
Then, for any $\delta\in ]0,\rho]$: 
\begin{equation}\label{bound:E-eta}
       \norm{E}_{\rho-\delta} \leq \CLE \norm{\etaL}_{\rho} + \CNE \norm{\etaN}_{\rho} \:\leq \CE  \wnormeta, 
       \end{equation}
\begin{equation}\label{bound:ET-eta}
       \norm{E^\ttop}_{\rho-\delta} \leq \CLET \norm{\etaL}_{\rho} + \CNET \norm{\etaN}_{\rho} \:\leq \CET  \wnormeta, 
       \end{equation}
and 
\begin{equation}\label{bound:DE-eta}
     \norm{\Dif E}_{\rho-\delta} \leq   \frac{\CLDE}{\delta} \norm{\etaL}_{\rho} + \frac{\CNDE}{\delta} \norm{\etaN}_{\rho} 
     \:\leq \frac{\CDE}{\delta}\wnormeta,
   \end{equation}
\begin{equation}\label{bound:DET-eta}
     \norm{(\Dif E)^\ttop}_{\rho-\delta} \leq   \frac{\CLDET}{\delta} \norm{\etaL}_{\rho} + \frac{\CNDET}{\delta} \norm{\etaN}_{\rho} 
     \:\leq  \frac{\CDET}{\delta}\wnormeta.
   \end{equation}
\end{lemma}

\begin{proof}
The hipothesis implies that $\norm{\OmegaL}_{\rho-\delta}< \fake <1$ and $\norm{\OmegaN}_{\rho-\delta}< \fake <1$, 
so the matrices $(I_n+\OmegaN \OmegaL)$ and $(I_n+\OmegaL \OmegaN)$
are invertible and
\[
\norm{(I_n+\OmegaN\OmegaL)^{\text{-}1} }_{\rho-\delta} < \frac{1}{1-\fake^2}, \quad
\norm{(I_n+\OmegaL\OmegaN)^{\text{-}1} }_{\rho-\delta} < \frac{1}{1-\fake^2}.
\]
Then, 
\[
I_{2n} + \Omega_0^{\text{-}1} \Esym = \begin{pmatrix} I_n & \OmegaN \\ -\OmegaL & I_n \end{pmatrix}
\]
is invertible, and 
\[
(I_{2n} + \Omega_0^{\text{-}1} \Esym)^{\text{-}1} = 
\begin{pmatrix} 
I_n & -\OmegaN \\ \OmegaL & I_n
\end{pmatrix}
\:
\begin{pmatrix} 
(I_n+\OmegaN\OmegaL)^{\text{-}1} & O_n \\ O_n & (I_n+\OmegaL\OmegaN)^{\text{-}1}
\end{pmatrix}.
\]
Notice that, since $I_{2n} + \Omega_0^{\text{-}1} \Esym= \Omega_0^{\text{-}1} P^\ttop \OK\: P$, then both $P$ and $P^\ttop$ are invertible in $\bar\TT_{\rho-\delta}$.

From the definition \eqref{def:eta} of $\eta$
we obtain
\[
	E= - P\: (I_{2n} + \Omega_0^{\text{-}1} \Esym)^{\text{-}1} \: \eta,\quad  E^\ttop= - \eta^\ttop \: (I_{2n} +  \Esym \Omega_0^{\text{-}1} )^{\text{-}1} \: P^\ttop,
\]
from where we could obtain easily  bounds for $\norm{E}_{\rho-\delta}$ and 
$\norm{E^\ttop}_{\rho-\delta}$, but it is  better to keep track the dependences with respect to  $\norm{\etaL}_{\rho}$ and $\norm{\etaN}_{\rho}$ separately. 
Since 
\[
E= -(L+N\OmegaL) (I_n+\OmegaN\OmegaL)^{\text{-}1} \etaL 
-(N-L\OmegaN) (I_n+\OmegaL\OmegaN)^{\text{-}1} \etaN
\]
and 
\[
E^\ttop = -(\etaL)^\ttop (I_n+\OmegaL\OmegaN)^{\text{-}1} (L^\ttop-\OmegaL N^\ttop) 
- (\etaN)^\ttop (I_n+\OmegaN\OmegaL)^{\text{-}1} (N^\ttop+\OmegaN L^\ttop),
\]
the bounds \eqref{bound:E-eta}  and \eqref{bound:ET-eta} follow.

We want to avoid using extra Cauchy estimates for $\Dif E$ and $(\Dif E)^\ttop$, and then loose more analyticity strip. To do so,  first, since 
\[
\Dif \eta= -\Dif(\Omega_0^{\text{-}1} \: P^\ttop \: \OK\:) E - \Omega_0^{\text{-}1} \: P^\ttop \: \OK\: \Dif E
\]
then
\[
\begin{split}
	\Dif E
         & = - P \: (I_{2n} + \Omega_0^{\text{-}1} \Esym)^{\text{-}1} \left(\Dif \eta + \Dif(\Omega_0^{\text{-}1} \: P^\ttop \: \OK\:) E \right) 
\end{split}
\]
and 
\[
\begin{split}
(\Dif E)^\ttop 
& = 
- \left((\Dif \eta)^\ttop + (\Dif(\Omega_0^{\text{-}1} \: P^\ttop \: \OK\:) E)^\ttop \right) (I_{2n} +  \Esym \Omega_0^{\text{-}1})^{\text{-}1}  \: P^\ttop
\end{split},
\]
from where the bounds \eqref{bound:DE-eta} and \eqref{bound:DET-eta} follow.
\end{proof}

\subsubsection{Control of the action of the Lie operator}

Here we control the action of the operator $\Loper$ on $\K$, $\L$, $\LT$, $\GL$,  $\B$, and $\N$, avoiding the dependence of the estimates on $\omega$.

\begin{lemma}\label{lem:lie control}
Assume the condition 
\begin{equation}\label{eq:new fake cond}
	\frac{\max\{1,\Csym\}}{\delta} \wnormeta < \fake <1,
\end{equation}
that includes the condition \eqref{eq:fake cond} in Lemma~\ref{lem:errors}.
Then, for any $\delta\in ]0,\rho]$:
\begin{align}
&\norm{\Loper  K}_{\rho-\delta} \leq  \CLieK, \label{CLieK}  
\\
& \norm{\Loper  L}_{\rho-\delta} \leq  \CLieL, \label{CLieL}
\\
& \norm{\Loper  L^\ttop}_{\rho-\delta} \leq  \CLieLT, \label{CLieLT}
\\
& \norm{\Loper   G_L}_{\rho-\delta} \leq \CLieGL,\label{CLieGL}
\\
&\norm{\Loper B }_{\rho-\delta} \leq \CLieB, \label{CLieB}
\\
&\norm{\Loper  N}_{\rho-\delta} \leq  \CLieN. \label{CLieN}
\end{align}
\end{lemma}

\begin{proof}
Estimate \eqref{CLieK} follows from the identity 
$\Loper  K=E  - X_{\H}\comp K$, 
the bound \eqref{bound:E-eta} and hypothesis \eqref{eq:new fake cond}.

Now, we consider the objects
$\Loper  L$ and $\Loper L^\ttop$, given by
\begin{equation*}
\begin{split}
\Loper   L & = 
\begin{pmatrix}
\Loper \DK & \Loper(X_p\comp K)
\end{pmatrix} 
= 
\begin{pmatrix}
 \Dif E -
(\Dif \Xh)\comp K \: \DK 
& 
(\Dif X_p)\comp K \: [\Loper  K]
\end{pmatrix}
\\
&= 
\begin{pmatrix}
 \Dif E -
(\Dif \Xh)\comp K \: \DK 
& 
(\Dif X_p)\comp K \: [E-\Xh\comp K]
\end{pmatrix}
\\
& = 
\begin{pmatrix}
\Dif E & (\Dif X_p)\comp K \: [E]
\end{pmatrix}
- 
(\Dif \Xh)\comp K \: \L
\end{split}
\end{equation*}
and 
\begin{equation*}
\Loper   L^\ttop = 
\begin{pmatrix}
 (\Dif E)^\ttop \\
((\Dif \Xp)\comp K\: [E])^\ttop
\end{pmatrix} -
((\Dif \Xh)\comp K\: \L)^\ttop
\end{equation*}
from where \eqref{CLieL} and \eqref{CLieLT} follow. 

Bound \eqref{CLieGL} follows from 
\begin{equation*}\label{eq:LieGL}
\Loper  G_L =  
\Loper  L^\ttop G\comp K\: L + 
L^\ttop (\Dif G)\comp K [\Loper K] \: L 
+ L^\ttop G\comp K\: \Loper  L.
\end{equation*}
Then, from the identity $G_L\: B= I_n$, we obtain
\[
\Loper  \B=  -\B\: \Loper G_L\: \B,
\]
from where we get the estimate \eqref{CLieB}.
Notice that $G_L$ and $B$ are symmetric, so we obtain the same bounds for their transposes. 

Finally, we consider the object $\Loper N$, 
given by
\begin{equation*}\label{eq:LieN:expanded}
\Loper  N= 
(\Dif J)\comp K [\Loper K] \: L\: \B 
+
J\comp K\: \Loper L \: \B + 
J\comp K\: L \: \Loper \B, 
\end{equation*}
and then we get \eqref{CLieN}.
\end{proof}

\subsubsection{Approximate reducibility}

A crucial step in the proof Theorem~\ref{thm:KAM}  is solving
the linearized equation arising from the application of the Newton
method. This is based on the (approximate) reduction of
such linear system into a block triangular form. This is the content of the following lemma.

\begin{lemma}\label{lem:approximate reducibility}
Assume the condition \eqref{eq:new fake cond} in Lemma~\ref{lem:lie control}.
Then, the linearized dynamics $\Dif \Xh \comp K:\bar \TT^d_\rho \to \CC^{2n \times 2n}$ is approximately reducible via the frame $P:\bar \TT^d_\rho \to \CC^{2n \times 2n}$ defined in \eqref{def:P}, to the block-triangular matrix-valued map 
$\Lambda: \bar \TT^d_\rho \to \CC^{2n \times 2n}$  defined in \eqref{def:Lambda}, 
where $T: \bar \TT^d_\rho \to \CC^{n \times n}$, defined in \eqref{def:T2}, is the torsion, for which
\begin{equation}\label{CT}
\norm{T}_{\rho} \leq \CT.
\end{equation}
Specifically, the {\em reducibility error map} $\Ered:\bar\TT^d_\rho\to \CC^{2n\times 2n}$ 
defined as
\begin{equation}\label{eq:Ered}
\Ered  :=
\Omega_0^{\text{-}1} P^\ttop \OK\: \Xoper{P}
 - \Lambda = \begin{pmatrix} 
\EredLL & \EredLN \\
\EredNL & \EredNN
\end{pmatrix}
\end{equation}
satifies, for any $\delta\in ]0,\rho]$:
\begin{equation}\label{CEredLL}
\norm{\EredLL}_{\rho-\delta} 
\leq \frac{\CLEredLL}{\delta} \norm{\etaL}_\rho +
\frac{\CNEredLL}{\delta} \norm{\etaN}_\rho, \phantom{\leq \frac{\CEredLN}{\delta}  \, \wnormeta}
\end{equation}
\begin{equation}\label{CEredNL}
\norm{\EredNL}_{\rho-\delta} \leq \frac{\CLEredNL}{\delta} \norm{\etaL}_\rho
+ \frac{\CNEredNL}{\delta} \norm{\etaN}_\rho, \phantom{\leq \frac{\CEredLN}{\delta}  \, \wnormeta}
\end{equation}
\begin{equation}\label{CEredNN}
\begin{split}
\norm{\EredNN}_{\rho-\delta} 
& \leq  \frac{\CLEredNN}{\delta} \norm{\etaL}_\rho +
\frac{\CNEredNN}{\delta} \norm{\etaN}_\rho \leq  \frac{\CEredNN}{\delta} \, \wnormeta,
\end{split}
\end{equation}
and
\begin{equation}\label{CEredLN}
\begin{split}
\norm{\EredLN}_{\rho-\delta}  
&\leq \frac{\CLEredLN}{\delta} \norm{\etaL}_{\rho} +  \frac{\CNEredLN}{\gamma \delta^{\tau+1} }\norm{\etaN}_{\rho}
\leq \frac{\CEredLN}{\delta}  \, \wnormeta.
\end{split}
\end{equation}
\end{lemma}

\begin{proof}
Using the notation in \eqref{def:Xoper}
we write the block components of \eqref{eq:Ered}:
\[
\Ered= 
\begin{pmatrix} 
\EredLL & \EredLN \\
\EredNL & \EredNN
\end{pmatrix}
= 
\begin{pmatrix}
N^\ttop\: \OK\: \Xoper{L} & N^\ttop\: \OK\:\Xoper{N} - T \\
-L^\ttop\: \OK\: \Xoper{L} & -L^\ttop\: \OK\: \Xoper{N}
\end{pmatrix}
\]

First, since 
\[
\EredLL = \begin{pmatrix}  N^\ttop \OK\: \Dif E & N^\ttop \OK\: \Dif X_p\comp K [E]\end{pmatrix} 
\]
and 
\[
	N^\ttop \OK\: \Dif E =   -\Dif(N^\ttop \OK) E - \Dif \etaL,
\]
then  we obtain \eqref{CEredLL}.
Analogously, 
\[
\EredNL = \begin{pmatrix}  -L^\ttop \OK\: \Dif E & -L^\ttop \OK\: \Dif X_p\comp K [E]\end{pmatrix} 
\]
and
\[
-L^\ttop \OK\: \Dif E =   \Dif(L^\ttop \OK) E - \Dif \etaN,
\]
from which we obtain \eqref{CEredNL}.

In order to bound $\EredNN$, we apply the left operator $\Loper$ to the identity $L^\ttop \OK\: N= -I_{n}$, 
and obtain 
\[
O_n=  \Loper  L^\ttop\: \OK\: N+ L^\ttop (\Dif\Omega)\comp K\: [E-X_h\comp K] +
L^\ttop \OK\: \Loper  N.
\]
Then, using  the geometric property \eqref{eq:prop2killSK}, we obtain
\begin{align*}
\EredNN= {} & 
L^\ttop (\Dif \OK\: [E]) N
+\Xoper{L}^\ttop \OK\:N = L^\ttop (\Dif \OK\: [E]) N - (\EredLL)^\ttop,
\label{eq:Ered22}
\end{align*}
where 
\[
(\EredLL)^\ttop = 
\begin{pmatrix}  
( \Dif( \OK\: N) E)^\ttop - (\Dif \etaL)^\ttop 
\\  -(\Dif X_p\comp K [E])^\ttop \OK\: N
\end{pmatrix},
\]
from which we obtain \eqref{CEredNN}.

Finally, 
\[
\begin{split}
\EredLN 
&= B\: L^\ttop\: G\comp K\: (\Dif J)\comp K [E]\: L \: B + B\: L^\ttop\: \OK\: \Xoper{L}\: B + B\: \Omega_L\: \Loper B \\
&= B\: L^\ttop\: G\comp K\: (\Dif J)\comp K [E]\: L \: B - B\: \EredNL\: B + B\: \Omega_L\: \Loper B.
\end{split}
\]
from which we obtain \eqref{CEredLN}.
\end{proof}

\subsection{One step of the iterative procedure}
\label{ssec:iterative lemma}

Here we apply one correction of the (modified) quasi-Newton method
described in Subsection~\ref{ssec:quasi Newton} and we obtain  quantitative
estimates for the new approximately invariant torus and related objects. We
set sufficient conditions to preserve the control of the previous
estimates. The constants that appear along the proof are collected in Appendix~\ref{app:constants},
Tables~\ref{tab:constants:all:2} and \ref{tab:constants:all:3}.

\begin{lemma}[The Iterative Lemma] \label{lem:iterative lemma}
For any $\delta\in ]0,\rho/3]$, there exist constants
$\Csym$, 
$\CxiL$, $\CDeltaKn$, $\CDeltaLn$, $\CDeltaLnT$,
$\CDeltaBn$, $\CDeltaNn$, $\CDeltaNnT$, $\CDeltaiTn$ and
$\QCetanL$, $\QCetanN$, 
such that if 
\begin{equation}\label{eq:iterative condition}
\frac{\hCDelta}{\gamma \delta^{\tau+1}}  \wnormeta< 1,
\end{equation}
where 
\begin{equation}\label{def:iterative constant}
\begin{split}
\hCDelta := \max \bigg\{ &
 \gamma\delta^\tau {\frac{\max\{1,\Csym\}}{\fake}},
   \CxiL, \frac{\delta\: \CDeltaKn}{\dist (\K(\TT^d_\rho),\partial \cmani_0)}, 
  \frac{\CDeltaDKn }{\sigmaDK- \norm{\DK}_\rho}, 
  \frac{\CDeltaDKnT }{\sigmaDKT - \norm{\DKT}_\rho}, 
\\
& 
\frac{\CDeltaBn}{\sigmaB - \norm{\B}_\rho}, \frac{\CDeltaNn }{\sigmaN - \norm{\N}_\rho}, \frac{\CDeltaNnT }{\sigmaNT - \norm{\NT}_\rho},
\frac{\CDeltaiTn}{\sigmaTinv - \abs{\aver{\T}^{\text{-}1}}},\frac{1}{\fakeetaN}  
\bigg\},
\end{split}
\end{equation}
then we have a new parameterization $\Kn\in (\Anal(\TT^d_{\rho-2\delta}))^{2n}$, 
that defines new objects $\Ln$, $\Bn$, $\Nn$ and $\Tn$ (obtained replacing $\K$ by $\Kn$ in the corresponding definitions) satisfying
\begin{align}
& \dist(\Kn(\TT^d_{\rho-2\delta}),\partial\cmani_0) >0, \label{eq:distB:iter1} \\
& \norm{\DKn}_{\rho-3\delta} < \sigmaDK, \label{eq:DK:iter1} \\
& \norm{\DKnT}_{\rho-3\delta} < \sigmaDKT, \label{eq:DKT:iter1} \\
& \norm{\Bn}_{\rho-3\delta} < \sigmaB, \label{eq:B:iter1} \\
& \norm{\Nn}_{\rho-3\delta} < \sigmaN, \label{eq:N:iter1} \\
& \norm{\NnT}_{\rho-3\delta} < \sigmaNT, \label{eq:NT:iter1} \\
& \abs{\aver{\Tn}^{\text{-}1}} < \sigmaTinv, \label{eq:iT:iter1} 
\end{align}
and
\begin{align}
& \norm{\Kn-\K}_{\rho-2 \delta} \leq 
\frac{\CDeltaKn}{\gamma\delta^\tau} \wnormeta,  \label{eq:est:DeltaK} \\
& \norm{\DKn-\DK}_{\rho-3 \delta} \leq 
\frac{\CDeltaDKn}{\gamma\delta^{\tau+1}} \wnormeta,  \label{eq:est:DeltaDK} \\
& \norm{\DKnT-\DKT}_{\rho-3 \delta} \leq 
\frac{\CDeltaDKnT}{\gamma\delta^{\tau+1}} \wnormeta,  \label{eq:est:DeltaDKT} \\
& \norm{\Bn-\B}_{\rho-3\delta} \leq  
\frac{\CDeltaBn}{\gamma\delta^{\tau+1}} \wnormeta, \label{eq:est:DeltaB} \\
& \norm{\Nn-\N}_{\rho-3 \delta} \leq 
\frac{\CDeltaNn}{\gamma\delta^{\tau+1}} \wnormeta,  \label{eq:est:DeltaN} \\
& \norm{\NnT-\NT}_{\rho-3 \delta} \leq 
\frac{\CDeltaNnT}{\gamma\delta^{\tau+1}} \wnormeta,  \label{eq:est:DeltaNT} \\
& \abs{\aver{\Tn}^{\text{-}1}-\aver{\T}^{\text{-}1} } \leq  
\frac{\CDeltaiTn}{\gamma\delta^{\tau+1}} \wnormeta. \label{eq:est:DeltaiT}
\end{align}
Moreover, the tangent and normal components of the new error of invariance 
\[
\En =  \Xh\comp \Kn + \Loper \Kn,
\]
satisfy
\begin{equation}\label{lemma:CxinN}
\norm{\etanN}_{\rho-3\delta} \leq \frac{\QCetanN}{\delta} \qwnormeta.
\end{equation}
and
\begin{equation}\label{lemma:CxinL}
\begin{split}
\norm{\etanL}_{\rho-3\delta} 
& \leq \frac{\QCetanL}{\gamma\delta^{\tau+1}} \qwnormeta.
\end{split}
\end{equation}
\end{lemma}

\begin{proof} 
We divide the proof into several steps. Starting from the initial parameterization, $\K$,  
we first consider an intermediate parameterization, 
\begin{equation}\label{def:Ki}
\Ki= \K  + N\xiN
\end{equation}
and then compute the new parameterization as
\begin{equation*}\label{def:Kn}
\Kn= \Phi_\xiLXp\comp \Ki \comp (\id + \xiLDK).
\end{equation*}
In the following, we will invoke Lemmas~\ref{lem:errors}, \ref{lem:lie control} and 
\ref{lem:approximate reducibility}, whose condition \eqref{eq:fake cond} is included 
into the hypothesis \eqref{eq:iterative condition} 
(this corresponds to the first term in \eqref{def:iterative constant}).

\medskip

\paragraph {\emph{Step 1: Control of the intermediate parameterization}.}
We recall that, in \eqref{def:Ki}, 
\begin{equation*}\label{eq:xiN-step}
\xiN = \averxiN+ \Roper \etaN.
\end{equation*}
where 
\begin{equation*}\label{eq:averxiN-step}
\averxiN= \aver{T}^{\text{-}1} \aver{\etaL-T\Roper\etaN}.
\end{equation*}
Hence, from R\"ussmann estimates in Lemma~\ref{lem:Russmann} with bite $\rho$
to $\Roper\etaN$ we obtain
\begin{equation}\label{CaverxiN}
\begin{split}
\abs{\averxiN} & 
\leq \CLaverxiN \norm{\etaL}_{\rho} + \frac{\CNaverxiN}{\gamma \delta^\tau} \norm{\etaN}_\rho
\leq \CaverxiN \wnormeta.
\end{split}
\end{equation}
and with bite $\delta$ we obtain
\begin{equation}\label{CxiN}
\begin{split}
\norm{\xiN}_{\rho-\delta} & 
\leq \CLxiN \norm{\etaL}_{\rho} + \frac{\CNxiN}{\gamma \delta^\tau} \norm{\etaN}_\rho
\leq \CxiN \wnormeta.
\end{split}
\end{equation}


Then 
\[
\Ki - \K = \N \xiN 
\]
and 
\begin{equation}\label{CDeltaKi}
\norm{\Ki-\K}_{\rho - \delta} \leq \CLDeltaKi \norm{\etaL}_\rho + \frac{\CNDeltaKi}{\gamma\delta^\tau}    \norm{\etaN}_\rho \leq \CDeltaKi \ \wnormeta
\end{equation}
The intermediate torus should be included in the domain $\cmani_0$, more
specifically $\Ki(\TT^d_{\rho-\delta}) \subset \cmani_0$. To verify that,
notice that 
\begin{align*}
\dist(\Ki(\TT^d_{\rho-\delta}),\partial \cmani_0) \geq {} &  \dist(\K(\TT^d_\rho),\partial \cmani_0) - \norm{\Ki - \K}_{\rho-\delta} \nonumber \\
\geq {} &  \dist(\K(\TT^d_\rho),\partial \cmani_0) -  \CDeltaKi\ \wnormeta >0, \label{eq:ingredient:iter:1}
\end{align*}
where the last inequality follows if 
\begin{equation}\label{cond:DeltaKi}
\frac{\CDeltaKi}{\dist(\K(\TT^d_\rho),\partial \cmani_0)} \ \wnormeta < 1.
\end{equation}
As we will see, this condition is implied by the third term in
\eqref{def:iterative constant}.

Using R\"ussmann and Cauchy estimates, see Lemma~\ref{lem:Russmann-Cauchy}, 
on \[
\bar K-K= N\averxiN + N \Roper\etaN,
\]
we get
\begin{equation}\label{CDeltaDKi}
\norm{\DKi - \DK}_{\rho-\delta} \leq  
\frac{\CLDeltaDKi}{\delta} \norm{\etaL}_{\rho} + 
\frac{\CNDeltaDKi}{\gamma \delta^{\tau+1}} \norm{\etaN}_{\rho}
\leq  
\frac{\CDeltaDKi}{\delta} \wnormeta.
\end{equation}
and, analogously,
\begin{equation}\label{CDeltaDKiT}
\norm{\DKiT - \DKT}_{\rho-\delta} \leq  
\frac{\CLDeltaDKiT}{\delta} \norm{\etaL}_{\rho} + 
\frac{\CNDeltaDKiT}{\gamma \delta^{\tau+1}} \norm{\etaN}_{\rho}
\leq  
\frac{\CDeltaDKiT}{\delta} \wnormeta.
\end{equation}
Notice that, in particular, 
\begin{equation}\label{condDKi}
	\norm{\DKi}_{\rho-\delta} <\sigmaDK, \quad \norm{\DKiT}_{\rho-\delta} < \sigmaDKT
\end{equation}
provided that 
\begin{equation}\label{condDKi}
	\frac{\CDeltaDKi}{\delta (\sigmaDK - \norm{\DK}_\rho)} \wnormeta <1
\end{equation}
and
\begin{equation}\label{condDKiT}
\frac{\CDeltaDKiT}{\delta (\sigmaDKT - \norm{\DKT}_\rho)} \wnormeta <1,
\end{equation}
conditions that, as we will see, are implied by the fourth and fith terms in
\eqref{def:iterative constant}.

Let $\Li$ be the $\L$ object associated to $\Ki$. Notice that, from
\eqref{condDKi}, 
\begin{equation*}\label{CLi}
\norm{\Li}_{\rho-\delta} < \sigmaL, \quad \norm{\LiT}_{\rho-\delta} < \sigmaLT. 
\end{equation*}
Then, since  
\[
\Li-\L = \begin{pmatrix} \DKi-\DK & \int_0^1 (\Dif\Xp)\comp(\K+\lambda\N\xiN) \ d\lambda \: [\Ki-\K] \end{pmatrix},
\]
we obtain
\begin{equation}\label{CDeltaLi}
\norm{\Li-\L}_{\rho-\delta} \leq 
\frac{\CLDeltaLi}{\delta} \norm{\etaL}_{\rho} + 
\frac{\CNDeltaLi}{\gamma \delta^{\tau+1}} \norm{\etaN}_{\rho}
\leq  
\frac{\CDeltaLi}{\delta} \wnormeta.
\end{equation}
and 
\begin{equation}\label{CDeltaLiT}
\norm{\LiT-\LT}_{\rho-\delta} \leq 
\frac{\CLDeltaLiT}{\delta} \norm{\etaL}_{\rho} + 
\frac{\CNDeltaLiT}{\gamma \delta^{\tau+1}} \norm{\etaN}_{\rho}
\leq  
\frac{\CDeltaLiT}{\delta} \wnormeta.
\end{equation}

\medskip

\paragraph{\em Step 2: Computation of the intermediate invariance errors}

The error of invariance in the intermediate step is 
\begin{equation*}\label{eq:defEi}
\Ei  = \Xh\comp\Ki + \Loper \Ki 
\end{equation*}
which can be written as 
\begin{equation*} \label{eq:Ei0}
\begin{split}
\Ei 
&= \E + \Delta^1\! X_h +  \Loper\N\: \xiN + \N\: \Loper\xiN \\
&= \E + \Delta^1\! X_h +  \Loper\N\: \xiN + \N\: \etaN 
\end{split}
\end{equation*}
where
\begin{equation*}
\Delta^1\! X_h = \int_{0}^1 (\Dif X_h)\comp(\K + \lambda \N\xiN)\: d\lambda \ [\N\xiN],
\end{equation*}
from were we obtain the bound 
\begin{equation}\label{CEi}
\norm{\Ei}_{\rho-\delta} \leq \CLEi \norm{\etaL}_\rho + \frac{\CNEi}{\gamma\delta^\tau} \norm{\etaN}_\rho \leq \CEi\ \wnormeta.
\end{equation}

Also
\begin{equation} 
\label{eq:Ei1}
\begin{split}
\Ei & = \E + (\Dif X_h)\comp K \N\xiN + \Delta^2\! X_h + \Loper\N\: \xiN + \N\: \Loper\xiN \\
       & = \E + \Xoper{\N}\: \xiN + \Delta^2\! X_h + \N\: \Loper\xiN,
       \end{split}
\end{equation}
where
\begin{equation*}
	\Delta^2\! X_h = \int_{0}^1 (1-\lambda) (\Dif^2 X_h)\comp(\K + \lambda (\Ki-\K))\: d\lambda \ [\Ki-\K,\Ki-\K],
\end{equation*}
from where we obtain that the {\em intermediate normal error} is
\begin{equation}\label{eq:etaiN}
\begin{split}
\etaiN 
& = \LiT\: \Omega\comp \Ki\: \Ei \\
& =  (\LiT-\LT)\: \Omega\comp\Ki\: \Ei + \LT\: (\Omega\comp\Ki - \Omega\comp\K)\: \Ei 
+ \LT \Omega\comp\K (\E + \Xoper{\N}\: \xiN + \Delta^2 X_h + \N\: \Loper\xiN) \\
& =  (\LiT-\LT)\: \Omega\comp\Ki\: \Ei + \LT\: (\Omega\comp\Ki - \Omega\comp\K)\: \Ei 
+ \etaN - \EredNN \xiN + \LT \Omega\comp\K\Delta^2 X_h - \Loper\xiN \\
& =  (\LiT-\LT)\: \Omega\comp\Ki\: \Ei + \LT\: (\Omega\comp\Ki - \Omega\comp\K)\: \Ei - \EredNN \xiN + 
\LT \Omega\comp\K \:\Delta^2\! X_h,
\end{split}
\end{equation}
where we emphasize  that $\Loper\xiN= \etaN$ in $\TT^d_\rho$. 
Notice that the intermediate normal error is quadratically small. Quantitatively, 
\begin{equation}\label{QCetaiN}
\norm{\etaiN}_{\rho-\delta} \leq \frac{\QCetaiN}{\delta} \qwnormeta.
\end{equation}

\medskip

\paragraph{\em Step 3:  Control of the new parameterization}


The new approximation is
\begin{equation}\label{eq:Kn-step}
\Kn= \Phi_\xiLXp\comp \Ki \comp (\id + \xiLDK),
\end{equation}
where $\xiL= (\xiLDK, \xiLXp)$ is given by
\begin{equation*}\label{eq:xiL-step}
\xiL  = \Roper(\etaL - T \xiN).
\end{equation*}
Using R\"ussmann estimates we obtain
\begin{equation}\label{CxiL}
\begin{split}
\norm{\xiL}_{\rho-2\delta} \
& \leq \frac{\CLxiL}{\gamma \delta^\tau} \norm{\etaL}_{\rho} +  \frac{\CNxiL}{\gamma^2 \delta^{2\tau}} \norm{\etaN}_{\rho}  
 \leq \frac{\CxiL}{\gamma \delta^{\tau}} \wnormeta.
\end{split}
\end{equation}

Notice that $\Ki(\bar \TT^d_{\rho-\delta}) \subset \cmani_0$, and the
components of $\xiL$ are in $\Anal(\TT^d_{\rho-2\delta})$.  Hence, the
computation of $\Kn$ in  \eqref{eq:Kn-step} could be done if
$\norm{\xiLDK}_{\rho-2\delta} < \delta$ and $\norm{\xiLXp}_{\rho-2\delta} <
\smax$, where we recall $\smax$ is the width of complex `time-domain"  of
$\Phi_s$.  Since $\rho<\smax$, these are conditions that are implied by the
hypothesis 
\begin{equation}\label{conditionxiL}
 \frac{\CxiL}{\gamma \delta^{\tau+1}} \wnormeta < 1,
\end{equation}
which corresponds to the second term in \eqref{def:iterative constant}.

Then, 
\begin{equation}\label{DeltaKn}
\begin{split} 
\Kn - K 
 & =  \Phi_\xiLXp\comp \Ki \comp (\id + \xiLDK) - \Ki  + 
\Ki - \K \\
 & = \int_0^1 \frac{\partial}{\partial\lambda} \left(\Phi_{\lambda\xiLXp}\comp \Ki \comp (\id+\lambda \xiLDK)\right)\ \dif\lambda \: + 
\Ki - \K
 \\
 &=   \int_0^1 \Dif_z \Phi_{\lambda \xiLXp}\comp \Ki  \comp (\id + \lambda \xiLDK) \: \Li\comp(\id + \lambda \xiLDK) \ \dif\lambda \: \xiL
 \: +  \Ki - \K
 \end{split}
\end{equation}
from where 
\begin{equation}\label{CDeltaKn}
\begin{split}
\norm{\Kn - \K}_{\rho-2\delta} 
& \leq \cteDPhi  \sigmaL \norm{\xiL}_{\rho-2\delta} + \norm{\Ki-\K}_{\rho-\delta}  +   \\
& \leq \frac{\CLDeltaKn}{\gamma\delta^\tau} \norm{\etaL}_\rho + \frac{\CNDeltaKn}{\gamma^2\delta^{2\tau}} \norm{\etaN}_\rho \leq  \frac{\CDeltaKn}{\gamma\delta^\tau} \wnormeta.
\end{split}
\end{equation}

The new approximation should remain in $\cmani_0$. In particular, we observe
that 
\begin{align*}\label{eq:ingredient:iter:2}
\dist(\Kn(\TT^d_{\rho-2\delta}),\partial \cmani_0) \geq {} &  \dist(K(\TT^d_\rho),\partial \cmani_0) - \norm{\Kn - K}_{\rho-2\delta} \nonumber \\
\geq {} &  \dist(K(\TT^d_\rho),\partial \cmani_0) -  \frac{\CDeltaKn}{\gamma \delta^{\tau}}  \wnormeta >0, 
\end{align*}
where the last inequality follows from hypothesis \eqref{eq:iterative
condition} (this corresponds to the third term in \eqref{def:iterative
constant}).  We emphasize that this control includes the fact that
$\dist(\Ki(\TT^d_{\rho-\delta}),\partial \cmani_0)>0$ (see
\eqref{cond:DeltaKi}). Moreover, we get \eqref{eq:distB:iter1} and
\eqref{eq:est:DeltaK} in Lemma~\ref{lem:iterative lemma}.

By directly applying  Cauchy estimates to the first part of \eqref{DeltaKn} and bounds \eqref{CDeltaDKi} and \eqref{CDeltaDKiT} 
we get
\begin{equation} \label{CLDeltaKn}
\begin{split}
\norm{\DKn-\DK}_{\rho-3\delta} 
& \leq 
\frac{\CLDeltaDKn}{\gamma \delta^{\tau +1}} \norm{\etaL}_\rho + 
\frac{\CNDeltaDKn}{\gamma^2 \delta^{2\tau +1}} \norm{\etaN}_\rho 
\leq \frac{\CDeltaDKn}{\gamma\delta^{\tau+1}} \wnormeta,
\end{split}
\end{equation}
and
\begin{equation} \label{CLDeltaKnT}
\begin{split}
\norm{\DKnT-\DKT}_{\rho-3\delta} 
& \leq 
\frac{\CLDeltaDKnT}{\gamma \delta^{\tau +1}} \norm{\etaL}_\rho + 
\frac{\CNDeltaDKnT}{\gamma^2 \delta^{2\tau +1}} \norm{\etaN}_\rho 
\leq \frac{\CDeltaDKnT}{\gamma\delta^{\tau+1}} \wnormeta,
\end{split}
\end{equation}
that correspond to \eqref{eq:est:DeltaDK} and  \eqref{eq:est:DeltaDKT}.
Then \eqref{eq:DK:iter1} and \eqref{eq:DKT:iter1} follow from hypotheses
\[
	\frac{\CDeltaDKn}{\gamma \delta^{\tau+1} (\sigmaDK- \norm{\DK}_{\rho})} \wnormeta < 1
\]
and
\[
\frac{\CDeltaDKnT}{\gamma \delta^{\tau+1} (\sigmaDKT - \norm{\DKT}_{\rho})} \wnormeta < 1,
\]
which corresponds to the fourth and fifth terms in \eqref{def:iterative constant}, and imply conditions \eqref{condDKi} and \eqref{condDKiT}.

In the following, we write $\Ln, \Nn, \dots$ for the corresponding $\L, \N, \dots$ objects associated to $\Kn$. In particular, the new tangent frame is 
\[
\Ln =  \begin{pmatrix} \DKn & \Xp\comp\Kn \end{pmatrix},
\]
and
\[
	\norm{\Ln}_{\rho-3\delta} < \sigmaL, \quad \norm{\LnT}_{\rho-3\delta}< \sigmaLT.
\]
Moreover, since
\[
\Ln-\L = \begin{pmatrix} \DKn-\DK & \int_0^1 (\Dif\Xp)\comp(\K+ \lambda (\Kn-\K)) \ d\lambda \: (\Kn-\K) \end{pmatrix},
\]
we get 
\begin{equation}\label{CDeltaLn}
\norm{\Ln-\L}_{\rho-3\delta} \leq 
\frac{\CLDeltaLn}{\gamma\delta^{\tau+1}} \norm{\etaL}_{\rho} + 
\frac{\CNDeltaLn}{\gamma^2 \delta^{2\tau+1}} \norm{\etaN}_{\rho}
\leq  
\frac{\CDeltaLn}{\gamma\delta^{\tau+1}} \wnormeta
\end{equation}
and 
\begin{equation}\label{CDeltaLnT}
\norm{\LnT-\LT}_{\rho-3\delta} \leq 
\frac{\CLDeltaLnT}{\gamma\delta^{\tau+1}} \norm{\etaL}_{\rho} + 
\frac{\CNDeltaLnT}{\gamma^2 \delta^{2\tau+1}} \norm{\etaN}_{\rho}
\leq  
\frac{\CDeltaLnT}{\gamma\delta^{\tau+1}} \wnormeta.
\end{equation}

The new restricted metric is 
\[
	\GLn = \LnT  G\comp \Kn\: \Ln,
\]
from where 
\[
	\GLn- \GL = (\LnT-\LT) \:  G\comp \Kn\: \Ln+ \LT\:  (G\comp\Kn - G\comp K)\: \Ln + 
	\LT\: G\comp K\: (\Ln-\L)
\]
and, then, 
\begin{equation}\label{CDeltaGLn}
	\norm{\GLn-\GL}_{\rho-3\delta} \leq 
	\frac{\CLDeltaGLn}{\gamma\delta^{\tau+1}} \norm{\etaL}_{\rho} + 
\frac{\CNDeltaGLn}{\gamma^2 \delta^{2\tau+1}} \norm{\etaN}_{\rho}
\leq  
	\frac{\CDeltaGLn}{\gamma\delta^{\tau+1}} \wnormeta.
\end{equation}

We know that $\GL$ is invertible (in $\bar\TT^d_{\rho}$), 
and $B= \invGL$ with $\norm{B}_\rho< \sigmaB$.  
We introduce now the constants
\[
\CLDeltaBn:= (\sigmaB)^2 \CLDeltaGLn, \quad
\CNDeltaBn := (\sigmaB)^2 \CNDeltaGLn, \quad
\CDeltaBn := (\sigmaB)^2 \CDeltaGLn.
\]
Then, since 
\begin{equation*}\label{eq:ingredient:iter:6}
\frac{\CDeltaBn}{\gamma\delta^{\tau+1} (\sigmaB-\norm{\B}_\rho)}\ \ \wnormeta < 1,
\end{equation*}
that corresponds to the sixth term  in \eqref{def:iterative constant}, 
from Lemma~\ref{lem:invertibility} we obtain that $\GLn$ is invertible (in $\bar\TT^d_{\rho-3\delta}$) and
\[
\norm{\Bn-\B}_{\rho-3\delta} < (\sigmaB)^2 \norm{\GLn-\GL}_{\rho-3\delta},\quad \norm{\Bn}_{\rho-3\delta} < \sigmaB,
\]
from where we obtain that 
\begin{equation}\label{CDeltaBn}
\norm{\Bn-\B}_{\rho-3\delta} \leq \frac{\CLDeltaBn}{\gamma\delta^{\tau+1}} \norm{\etaL}_{\rho} + 
\frac{\CNDeltaBn}{\gamma^2 \delta^{2\tau+1}} \norm{\etaN}_{\rho}
\leq  
	\frac{\CDeltaBn}{\gamma\delta^{\tau+1}} \wnormeta.
\end{equation}
We obtain estimates \eqref{eq:B:iter1} and \eqref{eq:est:DeltaB} in Lemma~\ref{lem:iterative lemma}.

The new normal frame is 
\[
\Nn = J\comp\Kn \Ln \Bn.
\]
Since
\[
\Nn - \N= (J\comp\Kn -J\comp \K) \Ln \Bn + J\comp \K  (\Ln-\L) \Bn + J\comp\K \L (\Bn-\B),
\]
then
\begin{equation}\label{CDeltaNn}
\norm{\Nn-\N}_{\rho-3\delta} \leq   \frac{\CLDeltaNn}{\gamma\delta^{\tau+1}} \norm{\etaL}_{\rho} + 
\frac{\CNDeltaNn}{\gamma^2 \delta^{2\tau+1}} \norm{\etaN}_{\rho}
\leq  
	\frac{\CDeltaNn}{\gamma\delta^{\tau+1}} \wnormeta
\end{equation}
and
\begin{equation}\label{CDeltaNnT}
\norm{\NnT-\NT}_{\rho-3\delta} \leq   \frac{\CLDeltaNnT}{\gamma\delta^{\tau+1}} \norm{\etaL}_{\rho} + 
\frac{\CNDeltaNnT}{\gamma^2 \delta^{2\tau+1}} \norm{\etaN}_{\rho}
\leq  
	\frac{\CDeltaNnT}{\gamma\delta^{\tau+1}} \wnormeta.
\end{equation}
Then, \eqref{eq:N:iter1} and \eqref{eq:NT:iter1}
follow from hypotheses
\[
	\frac{\CDeltaNn}{\gamma \delta^{\tau+1} (\sigmaN - \norm{\N}_{\rho})} \wnormeta < 1
\]
and
\[
\frac{\CDeltaNnT}{\gamma \delta^{\tau+1} (\sigmaNT - \norm{\NT}_{\rho})} \wnormeta < 1,
\]
which corresponds to the seventh and eighth terms in \eqref{def:iterative constant}.
Hence, we obtain estimates 
\eqref{eq:N:iter1},  \eqref{eq:NT:iter1} and \eqref{eq:est:DeltaN}, \eqref{eq:est:DeltaNT}
in Lemma~\ref{lem:iterative lemma}.

Finally, the new torsion is 
\[
	\Tn= \NnT\: \T_h\comp\Kn\: \Nn,
\]
and satisfies
\[
	\norm{\Tn}_{\rho-3\delta} \leq \CT.
\]
Moreover, since
\[
	\Tn - \T= (\NnT-\NT)\T_h\comp\Kn\: \Nn + \NT (\T_h\comp\Kn-\T_h\comp\K) \: \Nn +  
	\NT \T_h\comp\K\: (\Nn-\N),
\]
we obtain
\begin{equation}\label{CDeltaTn}
  \norm{\Tn-\T}_{\rho-3\delta} \leq   \frac{\CLDeltaTn}{\gamma\delta^{\tau+1}} \norm{\etaL}_{\rho} + 
\frac{\CNDeltaTn}{\gamma^2 \delta^{2\tau+1}} \norm{\etaN}_{\rho}
\leq  
	\frac{\CDeltaTn}{\gamma\delta^{\tau+1}} \wnormeta.
\end{equation}
We know that $\aver{\T}$ is invertible and $\abs{\aver{\T}^{\text{-}1}}< \sigmaTinv$. Mimicking the arguments made above for 
$\GL$ and $B= \invGL$,   we introduce the constants
\[
\CLDeltaiTn:= (\sigmaTinv)^2 \CLDeltaTn, \quad
\CNDeltaiTn := (\sigmaTinv)^2 \CNDeltaTn, \quad
\CDeltaiTn := (\sigmaTinv)^2 \CDeltaTn.
\]
Then, since 
\begin{equation*}
\frac{\CDeltaiTn}{\gamma\delta^{\tau+1} (\sigmaTinv-\abs{\aver{\T}^{\text{-}1}})}\ \ \wnormeta < 1,
\end{equation*}
which corresponds to the ninth term in \eqref{def:iterative constant},
we obtain that $\aver{\Tn}$ is invertible and
\[
\abs{\aver{\Tn}^{\text{-}1}-\aver{\T}^{\text{-}1}} < (\sigmaTinv)^2 \norm{\Tn-\T}_{\rho-3\delta},\quad \abs{\aver{\Tn}^{\text{-}1}}_{\rho-3\delta} < \sigmaTinv,
\]
from where we obtain that 
\begin{equation}\label{CDeltaiTn}
\abs{\aver{\Tn}^{\text{-}1}- \aver{\T}^{\text{-}1}} \leq \frac{\CLDeltaiTn}{\gamma\delta^{\tau+1}} \norm{\etaL}_{\rho} + 
\frac{\CNDeltaiTn}{\gamma^2 \delta^{2\tau+1}} \norm{\etaN}_{\rho}
\leq  
	\frac{\CDeltaiTn}{\gamma\delta^{\tau+1}} \wnormeta.
\end{equation}
We obtain then the estimates \eqref{eq:iT:iter1} and \eqref{eq:est:DeltaiT} on the new object.

\medskip

\paragraph{\emph{Step 4: Computation of the  new invariance errors}.}

In order to compute the new invariance error $\En$, we first compute
\begin{equation}\label{eq:DKn}
\begin{split}
\DKn 
&= 
\Dif_s \Phi_{\xiLXp}\comp \Ki \comp (\id+\xiLDK)\:  \Dif\xiLXp \\ 
&\phantom{=} + 
\Dif_z \Phi_{\xiLXp}\comp \Ki \comp (\id+\xiLDK)\: \Dif\Ki\comp (\id+\xiLDK)\: (I_d+ \Dif\xiLDK)
\\
&= 
\Xp\comp  \Phi_{\xiLXp}\comp \Ki \comp (\id+\xiLDK)\:  \Dif\xiLXp \\
&\phantom{=} + 
\Dif_z \Phi_{\xiLXp}\comp \Ki \comp (\id+\xiLDK)\: \Dif\Ki\comp (\id+\xiLDK)\: (I_d+ \Dif\xiLDK) 
\\
&= \Dif_z \Phi_{\xiLXp}\comp \Ki \comp (\id+\xiLDK)\:  \left(\Xp\comp  \Ki \comp (\id+\xiLDK)\:  \Dif\xiLXp + 
\Dif\Ki\comp (\id+\xiLDK)\: (I_d+ \Dif\xiLDK) \right),
\end{split}
\end{equation}
then 
\begin{equation*}\label{eq:LoperKn}
\begin{split}
\Loper\Kn 
& = \Dif_z \Phi_{\xiLXp}\comp \Ki \comp (\id+\xiLDK) \: (\Xp \comp \Ki \comp (\id+\xiLDK) \Loper\xiLXp + \Loper\Ki\comp(\id+\xiLDK) { \phantom )} \\ & \phantom{=} + {\phantom(}
\DKi\comp(\id+\xiLDK)\: \Loper\xiLDK)\\
& = \Dif_z \Phi_{\xiLXp}\comp \Ki \comp (\id+\xiLDK) \: (\Loper\Ki\comp(\id+\xiLDK) +\Li\comp(\id+\xiLDK)\: \Loper\xiL)
\end{split}
\end{equation*}
and 
\begin{equation*}\label{eq:XhKn}
\begin{split}
\Xh\comp\Kn 
& = \Xh\comp \Phi_{\xiLXp}\comp \Ki \comp (\id+\xiLDK) \\
& = \Dif_z \Phi_{\xiLXp}\comp \Ki \comp (\id+\xiLDK)\: \Xh\comp \Ki \comp (\id+\xiLDK).
\end{split}
\end{equation*}
As a result,
\begin{equation}\label{eq:En}
\begin{split}
\En 
&= \Dif_z \Phi_{\xiLXp}\comp \Ki \comp (\id+\xiLDK) \left( \Ei \comp (\id+\xiLDK) + \Li\comp(\id+\xiLDK)\: \Loper\xiL\right) \\
&= \Ei \comp (\id+\xiLDK) + \Li\comp(\id+\xiLDK)\: \Loper\xiL  + \\
&\phantom{=}  \left(  \Dif_z \Phi_{\xiLXp}\comp \Ki \comp (\id+\xiLDK) - I_{2n} \right) \left( \Ei \comp (\id+\xiLDK) + \Li\comp(\id+\xiLDK)\: \Loper\xiL\right).
\end{split}
\end{equation}
In oder to get a  crude bound of $\En$ from the first line in \eqref{eq:En}, we first bound
\begin{equation}\label{CLiexiL}
\begin{split}
\norm{\Loper\xiL}_{\rho-\delta} 
& = \norm{\etaL - T \xiN}_{\rho-\delta} \\
& \leq \CLLiexiL \norm{\etaL}_\rho + \frac{\CNLiexiL}{\gamma\delta^\tau} \norm{\etaN}_\rho
\leq \CLiexiL \wnormeta
\end{split}
\end{equation}
(notice that $\Loper\xiL= \etaL-T\xiN$ is in fact real-analytic in $\TT_\rho^d$), 
and, then,
\begin{equation}\label{CEn}
\begin{split}
\norm{\En}_{\rho-2\delta} 
& 
\leq \cteDPhi (\norm{\Ei}_{\rho-\delta} + \sigmaL \norm{\Loper\xiL}_{\rho-\delta}) 
\\
& \leq \CLEn \norm{\etaL}_\rho + \frac{\CNEn}{\gamma\delta^\tau} \norm{\etaN}_\rho 
\leq \CEn \wnormeta.
\end{split}
\end{equation}
A posteriori, we will see that $\En$ is quadratically small.

To compute the new normal error $\etanN$ we first compute the  new tangent frame $\Ln$. To do so, we first obtain
\begin{equation*}\label{eq:XpKn}
\begin{split} 
\Xp\comp\Kn 
&= \Xp\comp \Phi_{\xiLXp}\comp \Ki \comp (\id+\xiLDK)\\
&=  \Dif_z \Phi_{\xiLXp}\comp \Ki \comp (\id+\xiLDK) \: \Xp\comp \Ki \comp (\id+\xiLDK)
\end{split}
\end{equation*}
and, using \eqref{eq:DKn}, we get
\begin{equation}\label{eq:Ln}
\begin{split}
\Ln &=  \Dif_z \Phi_{\xiLXp}\comp \Ki \comp (\id+\xiLDK)  \: 
 \Li\comp (\id+\xiLDK) \: ( I_n + \hat\Dif\xiL ),
\end{split}
\end{equation}
where 
\begin{equation*}\label{eq:hDxiL}
\hat\Dif\xiL := \begin{pmatrix} \Dif\xiL & O_{n\times{n-d}}\end{pmatrix}.  
\end{equation*}
From \eqref{eq:En} and \eqref{eq:Ln}
we get
\begin{equation*}\label{eq:bbetaN}
\begin{split}
\etanN & = \Ln^\ttop \Omega\comp\Kn\: \En \\
& = 
(I_n +  (\hat\Dif\xiL)^\ttop)\:(\Li\comp (\id+\xiLDK))^\ttop 
\:\left( \Dif_z \Phi_{\xiLXp}\comp \Ki \comp (\id+\xiLDK)  \right)^\ttop \:\Omega\comp \Phi_{\xiLXp}\comp \Ki \comp (\id+\xiLDK) \\
&\phantom{=}
\:\Dif_z \Phi_{\xiLXp}\comp \Ki \comp (\id+\xiLDK) 
\left( \Ei \comp (\id+\xiLDK) + \Li\comp(\id+\xiLDK)\: \Loper\xiL\right) 
\\
&= (I_n +  (\hat\Dif\xiL)^\ttop)\:(\Li\comp (\id+\xiLDK))^\ttop  
\:\Omega\comp \Ki \comp (\id+\xiLDK)\: \left( \Ei \comp (\id+\xiLDK) + \Li\comp(\id+\xiLDK)\: \Loper\xiL\right)
\\
&=  (I_n +  (\hat\Dif\xiL)^\ttop)\:\left( \etaiN \comp (\id+\xiLDK) + \Omega_\Li \comp (\id+\xiLDK) \: \Loper\xiL\right).
\end{split}
\end{equation*}
See \eqref{eq:etaiN}. We observe that $\etanN$ is quadratically small. Quantitatively, 
since 
\[
\norm{(\Dif \xiL)^\ttop}_{\rho-3\delta} \leq \frac{n}{\delta} \norm{\xiL}_{\rho-2\delta} 
\leq \frac{n \CxiL}{\gamma \delta^{\tau +1}} \wnormeta < n,
\]
where we use condition \eqref{conditionxiL}, 
\[
\norm{\Omega_\Li \comp (\id+\xiLDK)}_{\rho-3\delta} 
 \leq 
\norm{\Omega_\Li}_{\rho-2\delta} \leq \frac{\CNOmegaL}{\gamma\delta^{\tau+1}} \norm{\etaiN}_{\rho-\delta} 
\leq  \frac{\CNOmegaL}{\gamma\delta^{\tau+1}}  \frac{\QCetaiN}{\delta} \qwnormeta,
\]
were we apply Lemma~\ref{lem:approximate lagrangianity} to $\Ki$ and bound \eqref{QCetaiN}, and, finally, 
applying \eqref{CLiexiL} and hypothesis~\ref{eq:iterative condition}, coming from the tenth term of 
\eqref{def:iterative constant}, 
we get
\begin{equation}\label{QCetanN}
\norm{\etanN}_{\rho-3\delta} \leq \frac{\QCetanN}{\delta} \qwnormeta,
\end{equation}
which corresponds to the bound \eqref{lemma:CxinN} in Lemma~\ref{lem:iterative lemma}.

The  new tangent error is
\begin{equation}\label{etanL}
\begin{split}
\etanL 
&= -\NnT\: \Omega\comp\Kn\: \En 
\\
&= - \NT \Omega\comp\K \: \En  
-  \NT \: (\Omega\comp\Kn - \Omega\comp\K)\: \En  -  (\NnT - \NT) \: \Omega\comp\Kn\: \En 
\\
& = 
-\NT\: \Omega\comp\K \left( \Ei \comp (\id+\xiLDK) + \Li\comp(\id+\xiLDK)\: \Loper\xiL\right)
\\
&\phantom{=} 
-\NT\: \Omega\comp\K \left(  \Dif_z \Phi_{\xiLXp}\comp \Ki \comp (\id+\xiLDK) - I_{2n} \right) \left( \Ei \comp (\id+\xiLDK) + \Li\comp(\id+\xiLDK)\: \Loper\xiL\right)
\\
&\phantom{=} 
-  \NT \: (\Omega\comp\Kn - \Omega\comp\K)\: \En - (\NnT - \NT) \: \Omega\comp\Kn\: \En  
\\ 
&=
-\NT\: \Omega\comp\K \E -\NT\: \Omega\comp\K (\Ei-\E) -\NT\: \Omega\comp\K \: \L\: \Loper\xiL
\\
&\phantom{=}
-\NT\: \Omega\comp\K (\Ei \comp (\id+\xiLDK) - \Ei)
\\
&\phantom{=}
- \NT\: \Omega\comp\K  (\Li\comp(\id+\xiLDK) - \L) \: \Loper\xiL
\\
&\phantom{=} 
-\NT\: \Omega\comp\K\: (  \Dif_z \Phi_{\xiLXp}\comp \Ki \comp (\id+\xiLDK) - I_{2n}) \: ( \Ei \comp (\id+\xiLDK) + \Li\comp(\id+\xiLDK)\: \Loper\xiL) 
\\
&\phantom{=}
- (\NnT - \NT) \: \Omega\comp\Kn\: \En  -  \NT \: (\Omega\comp\Kn - \Omega\comp\K)\: \En
\\
& =: \etanLu+ \etanLd + \etanLt + \etanLq +\etanLc,
\end{split}
\end{equation}
where the addends are numbered in order.
In the previous expression, all addends, but the first, are trivially quadratically small. But, in fact, 
$\Loper\xiL$ is selected in such a way that $\etanLu$ {\em is} quadratically small, since 
\[
\begin{split}
\etanLu 
& = \etaL - \NT\: \Omega\comp\K \:  (\Xoper{\N}\: \xiN + \Delta^2 \Xh + \N\: \Loper\xiN) - \Loper\xiL
\\
& = \etaL - T \xiN - \EredLN \xiN - \NT\: \Omega\comp\K \: \Delta^2 \Xh
       - \OmegaN  \: \Loper\xiN - \Loper\xiL
\\
& = - \EredLN \xiN - \NT\: \Omega\comp\K \: \Delta^2 \Xh - \OmegaN  \: \etaN,
\end{split}
\]
where we apply that $\N^\ttop \Omega\comp \K\: \L= I_n$, \eqref{eq:Ei1}, and 
$\Loper\xiL = \etaL - T\xiN$.
From this, we get a bound for the first addend of \eqref{etanL}:
\begin{equation}\label{QCetanLu}
\norm{\etanLu}_{\rho-3\delta} 
\leq \frac{\QCetanLu}{\delta} \qwnormeta.
\end{equation}
For the second addend, since 
\[
	 \Ei \comp (\id+\xiLDK) - \Ei = 
	\int_0^1 \Dif \Ei \comp (\id + \lambda \xiLDK) d\lambda \: \xiLDK,
\]
we get
\begin{equation}\label{QCetanLd}
	\norm{\etanLd}_{\rho-3\delta} \leq \CNT \cteOmega 
	\norm{\Dif\Ei}_{\rho-2\delta} \norm{\xiLDK}_{\rho-2\delta} 
	\leq \frac{\QCetanLd}{\gamma\delta^{\tau+1}} \qwnormeta,
\end{equation}
applying Cauchy estimates.
The third addend is 
\begin{equation*}
\etanLt = -\NT\: \Omega\comp\K  (\Li\comp(\id+\xiLDK) - \Li + \Li - \L) \: \Loper\xiL,
\end{equation*}
from which, proceeding analogously, we obtain the bound
\begin{equation}\label{QCetanLt}
\begin{split}
\norm{\etanLt}_{\rho-3\delta} 
       & \leq \CNT \cteOmega\: (\norm{\Dif\Li}_{\rho-2\delta} \norm{\xiLDK}_{\rho-2\delta} + \norm{\Li - L}_{\rho-2\delta}) \norm{\Loper\xiL}_{\rho-\delta} \\
&\leq \frac{\QCetanLt}{\gamma\delta^{\tau+1}} \qwnormeta.
\end{split}
\end{equation}
In order to bound the fourth addend, we first observe that
\[
\Dif_z \Phi_s  - I_{2n} = 
\int_0^1 \Dif_z \Xp\comp \Phi_{\lambda s} \: \Dif_z \Phi_s \: d\lambda \: s.
\]
Then 
\begin{equation}\label{QCetanLq}
\begin{split}
\norm{\etanLq}_{\rho-3\delta} 
& \leq \CNT \cteOmega \cteDXp \cteDPhi \norm{\xiLXp}_{\rho-2\delta} 
(\norm{\Ei}_{\rho-\delta} + \sigmaL  \norm{\Loper\xiL}_{\rho-\delta})\\
& \leq \frac{\QCetanLq}{\gamma\delta^{\tau}} \qwnormeta .
\end{split}
\end{equation}
 For the fifth addend, we first observe that 
\[
	\Omega\comp \Phi_s - \Omega = 
	\int_0^1 \Dif\Omega \comp \Phi_{\lambda s} \Xp\comp \Phi_{\lambda s} \: d\lambda \: s.
\]
Then, 
\begin{equation}\label{QCetanLc}
\begin{split}
\norm{\etanLc}_{\rho-3\delta} 
& \leq \frac{\QCetanLc}{\gamma\delta^{\tau+1}} \qwnormeta.
\end{split}
\end{equation}
Hence, collecting \eqref{QCetanLu}, \eqref{QCetanLd}, \eqref{QCetanLt}, \eqref{QCetanLq}, and \eqref{QCetanLc} we get
\begin{equation}\label{QCetanL}
\begin{split}
\norm{\etanL}_{\rho-3\delta} 
& \leq \frac{\QCetanL}{\gamma\delta^{\tau+1}} \qwnormeta,
\end{split}
\end{equation}
which is  bound \eqref{lemma:CxinL}.
With this we finish the proof of  Lemma~\ref{lem:iterative lemma}.
\end{proof}

\subsection{Convergence of the iterative process}\label{ssec:convergence}

Once the quadratic procedure has been
established in Section \ref{ssec:iterative lemma},
proving the convergence of the scheme follows standard arguments, that we will detail for
providing explicit conditions for the KAM theorem.

\begin{proof}
[Proof of Theorem \ref{thm:KAM}]
Let us consider the parameterization $\K_0:=K$ with initial invariance error $\E_0:=\E$, 
whose tangent and normal projections are
$\etaLs{0}$ and $\etaNs{0}$, respectively. We also introduce 
$\L_0:= \L$, $\LT_0:= \LT$, 
$\B_0:=\B$, 
$\N_0:= \N$, $\NT_0:= \NT$, 
and $\T_0:=\T$ 
associated to the initial parameterization.
By applying Lemma \ref{lem:iterative lemma} recursively, at the step $j$ we obtain new objects
$\K_j:= \Kn_{j-1}$,
$\E_j:= \En_{j-1}$,
$\etaLs{j}:= \etanLs{j-1}$, 
$\etaNs{j}:= \etanNs{j-1}$, 
$\L_j:= \Ln_{j-1}$, $\LT_j:= \LnT_{j-1}$, 
$\B_j:=\Bn_{j-1}$, 
$\N_j:= \Nn_{j-1}$, $\NT_j:= \NnT_{j-1}$,  
and $\T_j:=\Tn_{j-1}$. 

The domain of analyticity of these objects is reduced at every step, from the
initial value $\rho_0= \rho$ to a limiting value $\rho_\infty$.  At the step
$j$, the parameterization $K_j$ and associated objects are defined in a strip
of width $\rho_j$, and have been produced from the parameterization $K_{j-1}$,
which is defined in a strip of width $\rho_{j-1}$, throughout computations
(involving small divisors equations and derivatives) that produce three bites
of size 
$\delta_{j-1}$ to the width $\rho_{j-1}$, so then
\[
	\rho_j=\rho_{j-1} - 3 \delta_{j-1}.
\]
If we select a geometric sequence of bites 
\[
\delta_j= \frac{\delta_0}{a^j}
\]
with $a>1$, then, from the identity 
\[
\rho_\infty= \rho_0 - 3 \sum_{j= 0}^\infty \delta_j = \rho_0 - 3 \delta_0 \frac{a}{a-1}, 
\]
we get 
\begin{equation}\label{def:a}
	a= \frac{\rho_0 - \rho_\infty}{\rho_0 - 3\delta_0 -\rho_\infty}= \frac{\rho_0 - \rho_\infty}{\rho_1 -\rho_\infty}.
\end{equation}

Let us assume that we have successfully applied $j$ times
Lemma~\ref{lem:iterative lemma} We observe that condition \eqref{eq:iterative
condition} is required at every step, but the construction has been performed
in such a way that we control $\dist(K_j(\TT^d_{\rho_j}),\partial \cmani_0)$,
$\L_j:= \Ln_{j-1}$, $\LT_j:= \LnT_{j-1}$, $\B_j:=\Bn_{j-1}$, $\N_j:=
\Nn_{j-1}$, $\NT_j:= \NnT_{j-1}$,  and $\aver{T_j}^{\text{-}1}$ uniformly with
respect to $j$, so the constants that appear in Subsection~\ref{ssec:some
lemmas}, displayed in Table~\eqref{tab:some lemmas}, and in Lemma
\ref{lem:iterative lemma}, displayed in Tables~\ref{tab:constants:all:2} and
\ref{tab:constants:all:3}, are taken to be the same for all steps by
considering the worst value of $\delta_j$, that is, $\delta_0$. We also take
the upper bounds $\CRb$ and $\CRdb$ for $\CR(\delta_j)$ and $\CRdb(\delta_j)$.

The first computation is tracking the sequence $\ep_j$ of errors, defined to be
\[
\ep_j = \wnorm{\eta_j}{\rho_j}{\delta_j}.
\]
By defining 
\begin{equation}\label{QCetan}
	\QCetan:= \Abs{\QCetanL, a^\tau \QCetanN}
\end{equation}
we get, recursively, 
\begin{equation}\label{eq:recursion}
\begin{split}
\ep_j  &\leq \frac{\QCetan}{\gamma \delta_{j-1}^{\tau+1}} \ep_{j-1}^2 
\leq \left(\frac{\QCetan}{\gamma \delta_0^{\tau+1}} \right)^{2^j-1} \left(a^{\tau+1}\right)^{2^j-j-1} \ep_0^{2^j} = 
\left(\frac{\QCetan a^{\tau+1}}{\gamma \delta_0^{\tau+1}} \: \ep_0 \right)^{2^j-1} a^{-j(\tau+1)} \ep_0,
\end{split}
\end{equation}
where we used the sums $1+2+\ldots+2^{j-1}=2^j-1$, and
$1(j-1)+2(j-2)+2^2(j-3)\ldots+2^{j-2} 1 = 2^j-j-1$.

By imposing 
\begin{equation}\label{eq:theo:conv:1}
\frac{\QCetan a^{\tau+1}}{\gamma \delta_0^{\tau+1}} \: \ep_0 < \fakeE <1,
\end{equation}
which is included in \eqref{eq:KAM:HYP}, we have
\begin{equation}\label{eq:conv:geom}
\ep_j \leq \kappa^{2^j-1} a^{-j(\tau+1)} \ep_0 \leq (\kappa/a^{\tau+1})^j \ep_0 .
\end{equation}

Now, using expression \eqref{eq:conv:geom}, we check Hypothesis
\eqref{eq:iterative condition} of the iterative lemma, Lemma~\ref{lem:iterative
lemma}, so that we can perform the step $j+1$.  The required sufficient
condition is included in the hypothesis \eqref{eq:KAM:HYP} of the KAM theorem,
whose inequality has several terms that correspond to the different components
in \eqref{def:iterative constant}.

The first condition, using \eqref{eq:conv:geom}, is given by
\begin{equation}\label{eq:theo:conv:2}
\max\{1,\Csym\}\:  \frac{1}{\delta_j} \ep_j 
\leq  \max\{1,\Csym\}\: \frac{1}{\delta_0} (\kappa/a^{\tau})^j \ep_0 
\leq  \max\{1,\Csym\}\:\frac{1}{\delta_0}  \ep_0 < \fake,
\end{equation}
where the last inequality  is included in
\eqref{eq:KAM:HYP}. Checking the second conditions is analogous,
and it is 
\begin{equation}\label{eq:theo:conv:4}
\frac{\CxiL}{\gamma\delta_j^{\tau+1}} \ep_j \leq
\frac{\CxiL}{\gamma\delta_0^{\tau+1}} \kappa^j \ep_0 \leq 
\frac{\CxiL}{\gamma\delta_0^{\tau+1}} \ep_0 <1,
\end{equation}
where the last inequality  is again included in \eqref{eq:KAM:HYP}. Also the last, since:
\begin{equation}\label{eq:theo:conv:last}
	\frac{1}{\gamma\delta_j^{\tau+1}} \ep_j \leq  
	\frac{1}{\gamma\delta_0^{\tau+1}} \kappa^j \ep_0 \leq 
	\frac{1}{\gamma\delta_0^{\tau+1}} \ep_0 < \fakeetaN,  
\end{equation}
which is also included in \eqref{eq:KAM:HYP}.

In order to check the rest of conditions for 
$\K_j$, $\L_j$, $\LT_j$, $\B_j$, $\N_j$, $\NT_j$ and $\aver{\T_j}^{\text{-}1}$, 
we have to relate them to the conditions  corresponding to the initial objects 
$\K_0$, $\L_0$, $\LT_0$, $\B_0$, $\N_0$, $\NT_0$ and $\aver{\T_0}^{\text{-}1}$.

The third condition in \eqref{eq:iterative condition} is checked as follows. We recursively obtain that 
\[
\begin{split}
\dist(\K_j(\TT^d_{\rho_j}),\partial \cmani_0)- 
\frac{\CDeltaKn}{\gamma \delta_j^{\tau}} \ep_j 
& \geq
\dist(\K_0(\TT^d_{\rho_0}),\partial \cmani_0)\, -\, 
\sum_{i=0}^j \frac{\CDeltaKn}{\gamma \delta_i^{\tau}} \ep_i 
\\
& \geq
\dist(\K_0(\TT^d_{\rho_0}),\partial \cmani_0)\, -\, 
\sum_{i=0}^j \frac{\CDeltaKn}{\gamma \delta_0^{\tau}} (\kappa/a)^i \ep_0 \\
& \geq
\dist(\K_0(\TT^d_{\rho_0}),\partial \cmani_0)\, -\,
 \frac{\CDeltaKn}{\gamma \delta_0^{\tau}} \frac{a}{a-\kappa} \ep_0 >0 
\end{split}
\]
and the last inequality reads
\begin{equation}\label{eq:theo:conv:3}
\frac{a}{a-\kappa} \, \frac{\CDeltaKn}{\gamma \delta_0^{\tau}\: \dist(\K_0(\TT^d_{\rho_0}),\partial \cmani_0)} \ep_0 <1
\end{equation}
which is included as a condition into \eqref{eq:KAM:HYP}.

In order to check the fourth condition, we again proceed recursively to obtain that
and then we include the last inequality as 
\begin{equation}\label{eq:theo:conv:objects}
\frac{1}{1-\kappa} \ 
\frac{\CDeltaLn}{\gamma \delta_0^{\tau+1} 
(\sigmaL - \norm{\L_0}_{\rho_0})}  \ep_0 < 1
\end{equation}
into \eqref{eq:KAM:HYP}. The rest of conditions in \eqref{eq:iterative condition}, 
associated to $\LT_j$, $\B_j$, $\N_j$, $\NT_j$ and $\aver{\T_j}^{\text{-}1}$, 
follow by reproducing the same computations.
\[
\begin{split}
\norm{\DK_j}_{\rho_j} + 
\frac{\CDeltaDKn}{\gamma \delta_j^{\tau+1}} \ep_j 
& \leq
\norm{\DK_0}_{\rho_0} \, +\, 
\sum_{i=0}^j \frac{\CDeltaDKn}{\gamma \delta_i^{\tau+1}} \ep_i 
\\
& \leq
\norm{\DK_0}_{\rho_0} \, +\, 
\sum_{i=0}^j \frac{\CDeltaDKn}{\gamma \delta_0^{\tau+1}} \kappa^i \ep_0 \\
& \leq 
\norm{\DK_0}_{\rho_0} \, +\,
 \frac{\CDeltaDKn}{\gamma \delta_0^{\tau+1}} \frac{1}{1-\kappa} \ep_0 < \sigmaDK,
\end{split}
\]
and then we include the last inequality as 
\begin{equation}\label{eq:theo:conv:objects}
\frac{1}{1-\kappa} \ 
\frac{\CDeltaDKn}{\gamma \delta_0^{\tau+1} 
(\sigmaDK - \norm{\DK_0}_{\rho_0})}  \ep_0 < 1
\end{equation}
into \eqref{eq:KAM:HYP}. The rest of conditions in \eqref{eq:iterative condition}, 
associated to $\DKT_j$, $\B_j$, $\N_j$, $\NT_j$ and $\aver{\T_j}^{\text{-}1}$, 
{\em mutatis mutandis}.

Having guaranteed all hypothesis in Lemma~\ref{lem:iterative lemma},
we collect the inequalities
\eqref{eq:theo:conv:1},
\eqref{eq:theo:conv:2},
\eqref{eq:theo:conv:3},
\eqref{eq:theo:conv:4},
\eqref{eq:theo:conv:last}, 
\eqref{eq:theo:conv:objects} for $\L$ and the corresponding inequalities for the other objects, that are included into
hypothesis \eqref{eq:KAM:HYP}. This follows
by introducing the constant $\CtheoE$ as
\begin{equation}\label{def:theorem constant}
\begin{split}
\CtheoE := \max \bigg\{ &
 \gamma\delta^\tau {\frac{\max\{1,\Csym\}}{\fake}},\
 \CxiL,\
 \frac{\delta\: \CtheoDeltaK}{\dist (\K(\TT^d_\rho),\partial \cmani_0)},\ 
 \frac{\CtheoDeltaDK}{\sigmaDK - \norm{\DK}_\rho},\ 
 \frac{\CtheoDeltaDKT }{\sigmaDKT - \norm{\DKT}_\rho}, 
\\
& 
\frac{\CtheoDeltaB}{\sigmaB - \norm{\B}_\rho},\ 
\frac{\CtheoDeltaN}{\sigmaN - \norm{\N}_\rho},\ 
\frac{\CtheoDeltaNT }{\sigmaNT - \norm{\NT}_\rho},
\frac{\CtheoDeltaiT}{\sigmaTinv - \abs{\aver{\T}^{\text{-}1}}},\
\frac{1}{\fakeetaN},\ 
a^{\tau+1} \frac{\QCetan}{\kappa}   
\bigg\},
\end{split}
\end{equation}
where
\begin{equation}\label{eq:final constants 1}
	\CtheoDeltaK= \frac{a}{a-\fakeE} \CDeltaKn,\ 
	\CtheoDeltaDK= \frac{1}{1-\fakeE} \CDeltaDKn, \
	\CtheoDeltaDKT= \frac{1}{1-\fakeE} \CDeltaDKnT,
\end{equation}
and
\begin{equation}\label{eq:final constants 2}
        \CtheoDeltaB= \frac{1}{1-\fakeE} \CDeltaBn,\	
        \CtheoDeltaN= \frac{1}{1-\fakeE} \CDeltaNn, \
	\CtheoDeltaNT= \frac{1}{1-\fakeE} \CDeltaNnT,\
	\CtheoDeltaiT= \frac{1}{1-\fakeE} \CDeltaiTn.
\end{equation}

Finally, since
\[
\frac{\mathfrak{C}}{\gamma \delta^{\tau+1}}\: \ep <1,
\]
which is \eqref{eq:KAM:HYP}, we can apply the iterative process infinitely many times. 
Indeed, since $\ep_j\to 0$ when $j\to 0$, the iterative process 
converges to a true quasi-periodic torus $K_\infty$. Moreover, from the computations above
this object satisfies
$K_\infty \in \Anal(\TT^{d}_{\rho_\infty})$ and the controls
\eqref{eq:distKinf}, 
\eqref{eq:DKinf}, 
 \eqref{eq:DKTinf}, 
\eqref{eq:Binf}, 
\eqref{eq:Ninf},
\eqref{eq:NTinf}, and 
\eqref{eq:iTinf}.
Furthermore, we also get the distances
 between the initial and the limiting objects, 
 \eqref{eq:DeltaKinf}, 
  \eqref{eq:DeltaDKinf}, 
  \eqref{eq:DeltaDKTinf}, 
 \eqref{eq:DeltaBinf},
 \eqref{eq:DeltaNinf},
  \eqref{eq:DeltaNTinf},
 \eqref{eq:DeltaiTinf}.
This completes the proof of  the KAM theorem.
\end{proof}

\begin{remark}\label{rem:optimal strips}
We can argue from \eqref{eq:theo:conv:1} that, once one fixes the initial and final strip sizes $\rho_0$ and $\rho_\infty$, an (almost optimal) choice for the ratio $a$ (or the initial bite $\delta= \delta_0$) is the one that minimizes 
$\frac{a}{\delta_0}$. Since 
\[
\frac{a}{\delta} = \frac{\rho - \rho_\infty}{\delta (\rho-3\delta-\rho_\infty)}
\]
the best choice is for $\delta= \frac{\rho-\rho_\infty}{6}$, for which $a= 2$ and, hence 
\[
\delta_j= \frac{\Delta}{2^{j+1}},
\]
where $\Delta= (\rho_\infty-\rho_0)/3$.
This choice is empirically supported by the computations in \cite{FiguerasHL17}.
\end{remark}

\begin{remark}
The choice of the $\delta_j$ above as the geometric series with ratio $1/2$ is justified
by the following rationale. Let us assume that the constants involved in the theorem 
do not depend on $\delta$ (their dependence on it is very mild) then, 
\ref{eq:recursion} can be written as
\[
\ep_j  \leq 
\left(\frac{\QCetan}{\gamma}\ep_0\right)^{2^j-1} 
\left(\frac{1}{\delta_{j-1}^{1/2^{j-1}} \delta_{j-2}^{1/2^{j-2}} \dots \delta_0} \right)^{2^{j-1}(\tau+1)} 
\ep_0.
\]
Then,  one needs to 
minimize
\begin{equation*}\label{eq:tominimize}
\prod_{j = 0}^\infty
\delta_j^{-\frac{1}{2^j}}
\end{equation*}
or, equivalently 
\begin{equation*}\label{eq:tominimize}
- \sum_{j = 0}^\infty \frac{1}{2^j} \log \delta_j,
\end{equation*}
under the constraint
\begin{equation}\label{eq:constraint}
\sum_{j = 0}^\infty \delta_j = \Delta,
\end{equation}
where  $\delta_j > 0$ for all $j\geq 0$. 

Then, if $(\delta_j^0= 2^{-j-1}\Delta)_{j\geq0}$ is the target geometric sequence obtained in Remark~\ref{rem:optimal strips}, 
for any other positive sequence $(\delta_j)_{j\geq 0}$ satisfying \eqref{eq:constraint}, the function $f:[0,1]\to \RR$ defined 
as 
\[
	f(t)= - \sum_{j= 0}^\infty \frac{1}{2^j} \log( (1-t) \delta^0_j + t \delta_j)
\]
is strictly convex and $f'(0)= 0$. 
\end{remark}

\section{Acknowledgements}

The authors are grateful to Alejandro Luque, Benjamin Meco, Christian Bjerkl\"ov and Gerard Farr\'e for fruitful discussions. 

J.-Ll.F. has been partially supported by the Swedish VR Grant 2019-04591, and 
A.H. has  been supported by the Spanish grant PID2021-125535NB-I00 (MCIU/AEI/FEDER, UE).
This work has also being supported by the Spanish State Research Agency, through the
Severo Ochoa and Mar\'ia de Maeztu Program for Centers and Units of
Excellence in R\&D (CEX2020-001084-M).

\bibliographystyle{plain}
\bibliography{references}

\def\cprime{$'$} \def\cprime{$'$} \def\cprime{$'$} \def\cprime{$'$}
\begin{thebibliography}{10}

\bibitem{Arnold63a}
V.I. Arnold.
\newblock Proof of a theorem of {A}. {N}. {K}olmogorov on the preservation of
  conditionally periodic motions under a small perturbation of the
  {H}amiltonian.
\newblock {\em Uspehi Mat. Nauk}, 18(5 (113)):13--40, 1963.

\bibitem{BroerHS96}
H.W. Broer, G.B. Huitema, and M.B. Sevryuk.
\newblock {\em Quasi-periodic motions in families of dynamical systems. {O}rder
  amidst chaos}.
\newblock Lecture Notes in Math., {V}ol 1645. Springer-Verlag, Berlin, 1996.

\bibitem{Cannas01}
A.~Cannas~da Silva.
\newblock {\em Lectures on symplectic geometry}, volume 1764 of {\em Lecture
  Notes in Mathematics}.
\newblock Springer-Verlag, Berlin, 2001.

\bibitem{CellettiC88}
A.~Celletti and L.~Chierchia.
\newblock Construction of {A}nalytic {KAM} {S}urfaces and {E}ffective
  {S}tability {B}ounds.
\newblock {\em Comm. Math. Phys.}, 118(1):199--161, 1988.

\bibitem{CellettiC07}
A.~Celletti and L.~Chierchia.
\newblock K{AM} stability and celestial mechanics.
\newblock {\em Mem. Amer. Math. Soc.}, 187(878):viii+134, 2007.

\bibitem{Llave01}
R.~de~la Llave.
\newblock A tutorial on {KAM} theory.
\newblock In {\em Smooth ergodic theory and its applications (Seattle, WA,
  1999)}, volume~69 of {\em Proc. Sympos. Pure Math.}, pages 175--292. Amer.
  Math. Soc., Providence, RI, 2001.

\bibitem{GonzalezJLV05}
R.~de~la Llave, A.~Gonz{\'a}lez, {\`A}.~Jorba, and J.~Villanueva.
\newblock K{AM} theory without action-angle variables.
\newblock {\em Nonlinearity}, 18(2):855--895, 2005.

\bibitem{LlaveR90}
R.~de~la Llave and D.~Rana.
\newblock Accurate strategies for small divisor problems.
\newblock {\em Bull. Amer. Math. Soc. (N.S.)}, 22(1):85--90, 1990.

\bibitem{Dumas14}
H.S. Dumas.
\newblock {\em The {KAM} story}.
\newblock World Scientific Publishing Co. Pte. Ltd., Hackensack, NJ, 2014.
\newblock A friendly introduction to the content, history, and significance of
  classical Kolmogorov-Arnold-Moser theory.

\bibitem{FiguerasHL17}
J.-Ll. Figueras, A.~Haro, and A.~Luque.
\newblock Rigorous {C}omputer-{A}ssisted {A}pplication of {KAM} {T}heory: {A}
  {M}odern {A}pproach.
\newblock {\em Found. Comput. Math.}, 17(5):1123--1193, 2017.

\bibitem{FiguerasHL18}
J.-Ll. Figueras, A.~Haro, and A.~Luque.
\newblock On the sharpness of the {R}\"ussmann estimates.
\newblock {\em Commun. Nonlinear Sci. Numer. Simul.}, 55:42--55, 2018.

\bibitem{FiguerasHL20}
Jordi-Llu\'{\i}s Figueras, Alex Haro, and Alejandro Luque.
\newblock Effective bounds for the measure of rotations.
\newblock {\em Nonlinearity}, 33(2):700--741, 2020.

\bibitem{FontichLS09}
E.~Fontich, R.~de~la Llave, and Y.~Sire.
\newblock Construction of invariant whiskered tori by a parameterization
  method. {I}. {M}aps and flows in finite dimensions.
\newblock {\em J. Differential Equations}, 246(8):3136--3213, 2009.

\bibitem{GonzalezHL13}
A.~Gonz\'alez, A.~Haro, and R.~de~la Llave.
\newblock Singularity theory for non-twist {KAM} tori.
\newblock {\em Mem. Amer. Math. Soc.}, 227(1067):vi+115, 2014.

\bibitem{Greene79}
J.M. Greene.
\newblock A method for determining a stochastic transition.
\newblock {\em J. Math. Phys}, 20(6):1183--1201, 1975.

\bibitem{HaroCFLM16}
A.~Haro, M.~Canadell, J.-Ll. Figueras, A.~Luque, and J.-M. Mondelo.
\newblock {\em The parameterization method for invariant manifolds}, volume 195
  of {\em Applied Mathematical Sciences}.
\newblock Springer, [Cham], 2016.
\newblock From rigorous results to effective computations.

\bibitem{HaroL19}
A.~Haro and A.~Luque.
\newblock A-posteriori {KAM} theory with optimal estimates for partially
  integrable systems.
\newblock {\em J. Differential Equations}, 266(2-3):1605--1674, 2019.

\bibitem{Herman86}
M.-R. Herman.
\newblock Sur les courbes invariantes par les diff\'eomorphismes de l'anneau.
  {V}ol.\ 2.
\newblock {\em Ast\'erisque}, (144):248, 1986.
\newblock With a correction to: {\it On the curves invariant under
  diffeomorphisms of the annulus, Vol.\ 1} (French) [Ast\'erisque No. 103-104,
  Soc.\ Math.\ France, Paris, 1983; MR 85m:58062].

\bibitem{Kolmogorov54}
A.N. Kolmogorov.
\newblock On conservation of conditionally periodic motions for a small change
  in {H}amilton's function.
\newblock {\em Dokl. Akad. Nauk SSSR (N.S.)}, 98:527--530, 1954.
\newblock Translated in p. 51--56 of \emph{Stochastic Behavior in Classical and
  Quantum Hamiltonian Systems, Como 1977} (eds. G. Casati and J. Ford) Lect.
  Notes Phys. 93, Springer, Berlin, 1979.

\bibitem{MacKay93}
R.S. MacKay.
\newblock {\em Renormalisation in area-preserving maps}, volume~6 of {\em
  Advanced Series in Nonlinear Dynamics}.
\newblock World Scientific Publishing Co. Inc., River Edge, NJ, 1993.

\bibitem{MarsdenW74}
J.~Marsden and A.~Weinstein.
\newblock Reduction of symplectic manifolds with symmetry.
\newblock {\em Rep. Mathematical Phys.}, 5(1):121--130, 1974.

\bibitem{Moser62}
J.~Moser.
\newblock On invariant curves of area-preserving mappings of an annulus.
\newblock {\em Nachr. Akad. Wiss. G\"ottingen Math.-Phys. Kl. II}, 1962:1--20,
  1962.

\bibitem{Moser66c}
J.~Moser.
\newblock On the theory of quasiperiodic motions.
\newblock {\em SIAM Rev.}, 8(2):145--172, 1966.

\bibitem{Neishtadt81}
A.~I. Neishtadt.
\newblock Estimates in the {K}olmogorov theorem on conservation of
  conditionally periodic motions.
\newblock {\em J. Appl. Math. Mech}, 45(6):1016--1025, 1981.

\bibitem{Poschel82}
J.~P{\"o}schel.
\newblock Integrability of {H}amiltonian systems on {C}antor sets.
\newblock {\em Comm. Pure Appl. Math.}, 35(5):653--696, 1982.

\bibitem{Russmann75}
H.~R{\"u}ssmann.
\newblock On optimal estimates for the solutions of linear partial differential
  equations of first order with constant coefficients on the torus.
\newblock In {\em Dynamical systems, theory and applications (Rencontres,
  Battelle Res. Inst., Seattle, Wash., 1974)}, pages 598--624. Lecture Notes in
  Phys., Vol. 38. Springer, Berlin, 1975.

\bibitem{Russmann76a}
H.~R{\"u}ssmann.
\newblock On optimal estimates for the solutions of linear difference equations
  on the circle.
\newblock In {\em Proceedings of the Fifth Conference on Mathematical Methods
  in Celestial Mechanics (Oberwolfach, 1975), Part I. Celestial Mech.},
  volume~14, 1976.

\bibitem{Sevryuk14}
M.B. Sevryuk.
\newblock The classical {KAM} theory in the last decade: a slow progress.
\newblock
  \url{http://www.mathnet.ru/php/presentation.phtml?option_lang=eng&presentid=10271},
  2014.
\newblock The Seventh International Conference on Differential and Functional
  Differential Equations.

\bibitem{Villanueva08}
J.~Villanueva.
\newblock Kolmogorov theorem revisited.
\newblock {\em J. Differential Equations}, 244(9):2251--2276, 2008.

\bibitem{Villanueva17}
J.~Villanueva.
\newblock A new {A}pproach to the {P}arameterization {M}ethod for {L}agrangian
  {T}ori of {H}amiltonian {S}ystems.
\newblock {\em J. Nonlinear Science}, 27:495--530, 2017.

\bibitem{Villanueva18}
Jordi Villanueva.
\newblock A parameterization method for {L}agrangian tori of exact symplectic
  maps of {$\Bbb R^{2r}$}.
\newblock {\em SIAM J. Appl. Dyn. Syst.}, 17(3):2289--2331, 2018.

\end{thebibliography}

\appendix

\section{An auxiliary lemma to control the inverse of a matrix}\label{app:invertibility}

In several instances in the proofs performed in this paper we control 
control the correctionof inverses of matrices using Neumann series argument.
For instance, this affects to the estimates in \eqref{eq:est:DeltaB}, \eqref{eq:est:DeltaiT}.
For convenience, we present the following auxiliary result separately.
Notice that the result is presented for matrices but
it is directly extended for matrix-valued maps with the
corresponding norm (see Subsection \ref{ssec:analytic setting}).

\begin{lemma}\label{lem:invertibility}
Let $M \in \CC^{n \times n}$ be an invertible matrix
satisfying $\abs{M^{\text{-}1}} < \sigma$. Assume that $\bar M \in \CC^{n \times n}$
satisfies
\begin{equation}\label{eq:lem:aux}
\frac{\sigma \abs{M^{\text{-}1}} \abs{\bar M-M}}{\sigma-\abs{M^{\text{-}1}}} < 1\,.
\end{equation}
Then, we have that $\bar M$ is invertible and
\[
\abs{\bar M^{\text{-}1}} < \sigma\, \qquad \abs{\bar M^{\text{-}1}-M^{\text{-}1}} <  \sigma  \abs{M^{\text{-}1}} \abs{\bar M-M}.
\]
\end{lemma}

\begin{proof}
For notational convenience, let us denote 
\[
x= \frac{\abs{M^{\text{-}1}}}{\sigma}, \quad 
\lambda= \frac{\sigma \abs{M^{\text{-}1}} \abs{\bar M-M}}{\sigma-\abs{M^{\text{-}1}}},
\]
so that $x\in ]0,1[$ and $\lambda\in [0,1[$. Since 
\[
\abs{M^{\text{-}1}} \abs{\bar M-M} = \lambda (1-x)  < 1, 
\]
the matrix 
\[
	\bar M= M (I + M^{\text{-}1} (\bar M - M)) 
\]
is invertible and 
\[
	\abs{{\bar M}^{\text{-}1}} \leq \frac{\sigma x}{1- \lambda (1-x)}  < \sigma,
\]
where in the last inequality we use that, since $\lambda<1$, the function $f(x)= x/(1-\lambda(1-x))$ is continuous and strictly increasing in the interval $[0,1]$, and $f(1)= 1$.  Moreover, 
\[
	\abs{\bar M^{\text{-}1}-M^{\text{-}1}} = \abs{ \bar M^{\text{-}1} (M-\bar M) M^{\text{-}1}} < \sigma \abs{M^{\text{-}1}} \abs{\bar M- M}.
\]
\end{proof}

\begin{remark}
In particular, notice that, if $\abs{M^{\text{-}1}} < \sigma$, hypothesis \eqref{eq:lem:aux} holds if 
\begin{equation*}\label{eq:lem:aux2}
\frac{\sigma^2 \abs{\bar M-M}}{\sigma-\abs{M^{\text{-}1}}} < 1.
\end{equation*}
\end{remark}

\section{Compendium of constants involved in the KAM theorem}\label{app:constants}

In this appendix we collect the recipes to compute all constants
involved in the different estimates presented in the proof of the main results of this paper. 
Keeping track of these constants is crucial to apply Theorem~\ref{thm:KAM} 
in particular problems and for concrete values of
parameters. In the following, Table~\ref{tab:some lemmas} corresponds to the geometric
constructions outlined in Subsection~\ref{ssec:some lemmas}, and  
Tables~\ref{tab:constants:all:2} and \ref{tab:constants:all:3} correspond to the iterative lemma
in Subsection~\ref{ssec:iterative lemma}. Finally, Table~\ref{tab:final constants} presents the constants in Theorem~\ref{thm:KAM}, that are defined at the end of its proof, in Subsection~\ref{ssec:convergence}.

The input values are (see Theorem~\ref{thm:KAM}) :
\begin{itemize} 
\item the global bounds 
$\cteOmega$, $\cteG$, $\cteJ$, $\cteJT$, $\cteDOmega$, $\cteDG$, $\cteDJ$, $\cteDJT$, $\cteXH$, $\cteDXH$, $\cteDXHT$, $\cteDDXH$, $\cteTh$, $\cteDTh$, $\cteXp$, $\cteXpT$, $\cteDXp$, $\cteDXpT$, and $\cteDPhi$;
\item the condition numbers 
$\sigmaDK$, $\sigmaDKT$,  $\sigmaL$, $\sigmaLT$, $\sigmaB$, $\sigmaN$, $\sigmaNT$, and $\sigmaTinv$;

\item the control constants  $\fake<1$, $\fakeE<1$, $\fakeetaN$;

\item the strip sizes $\rho_\infty<\rho < r$ and bite $\delta<\frac{\rho-\rho_\infty}{3}$.
\end{itemize}

Constants $\CR(\delta)$ and $\CRd(\delta)$ from Lemmas~\ref{lem:Russmann} and \ref{lem:Russmann-Cauchy} depend on the bite on the strip, $\delta$, and can substituted by the correspondind upper bounds $\CRb$ and $\CRdb$.

\newpage

\color{black}

\bgroup
{
\def\arraystretch{1.5}
\begin{longtable}{| l l l|}
\caption{Constants in Section~\ref{ssec:some lemmas}, Lemmas~\ref{lem:approximate lagrangianity}, \ref{lem:approximate symplecticity}, \ref{lem:errors}, \ref{lem:lie control}, \ref{lem:approximate reducibility}.} 
\label{tab:some lemmas}\\
\hline
Object & Constant & Label \\ 
\hline
\hline
$\OmegaL$  
&
$\CNOmegaL = (d+n-2) \CRd(\delta)$ 
&
\eqref{CNOmegaL} 
\\ 
\hline
$\OmegaN$  
& 
$\CNOmegaN = (\sigmaB)^2 \CNOmegaL$ 
&
\eqref{CNOmegaN}
\\
\hline
$\Esym$  
& 
$\Csym =  \max\{\CNOmegaL,\CNOmegaN\}$ 
&
\eqref{Csym}
\\
\hline
$E$ 
&
$\CLE= \frac{ \CL + \CN \fake}{1-\fake^2},\quad \CNE= \frac{\CN + \CL\fake}{1-\fake^2},
\quad \CE = \CLE + \gamma\delta^\tau \CNE$ 
& 
\eqref{bound:E-eta}
\\
\hline
$E^\ttop$ 
&
$\CLET= n \frac{\CLT + \fake \CNT}{1-\fake^2},\quad \CNET= n \frac{\CNT+ \fake \CLT}{1-\fake^2}, 
\quad \CE = \CLET + \gamma\delta^\tau \CNET$
&
\eqref{bound:ET-eta}
\\
\hline
$\Dif E$ & 
$\CLDE= d (1+ \cteOmega (\CNT \CLE + \CLT\CNE)) \CLE$ %
& 
\\
& 
$\CNDE= d (1+ \cteOmega (\CNT \CLE + \CLT\CNE)) \CNE, \quad \CDE= \CLDE + \gamma\delta^\tau \CNDE$
& 
\eqref{bound:DE-eta}
\\
\hline
$(\Dif E)^\ttop$ & 
$\CLDET= n (1+\cteOmega (\CN \CLET + \CL\CNET)) \CLET$ 
& 
\\
& $\CNDET= n (1+\cteOmega (\CN \CLET + \CL\CNET)) \CNET, \quad \CDET= \CLDET + \gamma\delta^\tau \CNDET$
& 
\eqref{bound:DET-eta}
\\
\hline
$\Loper  K$ 
& 
$\CLieK = \CE \delta\fake+ \cteXH$ 
& \eqref{CLieK} 
\\
\hline
$\Loper  L$ 
& 
$\CLieL = \CDE \fake + \cteDXp \CE \delta\fake +\cteDXH \CL$
& \eqref{CLieL} 
\\ 
\hline
$\Loper  L^\ttop$ 
& 
$\CLieLT = \max \{ \CDET \fake , \cteDXpT \CE \delta\fake \} + \CLT \cteDXHT$ 
& 
\eqref{CLieLT} 
\\
\hline
$\Loper   G_L$ 
& 
$\CLieGL = \CLieLT \cteG \CL + \CLT \cteDG \CLieK \CL + \CLT \cteG \CLieL$ 
& 
\eqref{CLieGL} 
\\ 
\hline
$\Loper   B$ 
& 
$\CLieB = (\sigmaB)^2 \CLieGL$ 
& \eqref{CLieB} 
\\ 
\hline
$\Loper  N$ 
& 
$\CLieN =  \cteDJ  \CLieK \CL \sigma_B + \cteJ \CLieL \sigmaB + \cteJ  \CL \CLieB$ 
& 
\eqref{CLieN} \\ 
\hline
$T$ 
&
$\CT= \CNT \cteTH  \CN$
& \eqref{CT} \\ 
\hline
$\EredLL$ & $\CLEredLL = d + (d+\delta \cteDXp) \CNT \cteOmega \CLE, \quad \CNEredLL = (d+\delta \cteDXp) \CNT \cteOmega \CNE$ & \eqref{CEredLL} \\ 
\hline
$\EredNL$ & $\CLEredNL = (d+\delta \cteDXp) \CLT \cteOmega \CLE, \quad \CNEredNL = d + (d+\delta \cteDXp) \CLT \cteOmega \CNE$ & \eqref{CEredNL} \\ 
\hline
$\EredNN$ & 
$\CLEredNN= \delta\: \CLT \cteDOmega \CLE \CN + \max\{ n + d \CLE \cteOmega \CNT, \delta \cteDXpT \CLE \cteOmega \CN\}$
& 
\\ 
&
$\CNEredNN= \delta\: \CLT \cteDOmega \CNE \CN + \max\{d \CNE \cteOmega \CNT, \delta \cteDXpT \CNE \cteOmega \CN\}$
& 
\\ 
& $\CEredNN= \CLEredNN + \gamma\delta^\tau \CNEredNN$
& \eqref{CEredNN}
\\
\hline
$\EredLN$ 
& $\CLEredLN=  \delta (\sigmaB)^2  \CLT\cteG \cteDJ \CL \CLE + (\sigmaB)^2 \CLEredNL$ 
& 
\\
& $\CNEredLN= \gamma \delta^{\tau+1} (\sigmaB)^2  \CLT\cteG \cteDJ \CL \CNE + \gamma \delta^{\tau} (\sigmaB)^2 \CNEredNL + \sigmaB\CNOmegaL \CLieB$ 
& 
\\ 
& $\CEredLN=\CLEredLN + \CNEredLN$
& \eqref{CEredLN} 
\\
\hline
\end{longtable}}
\egroup

\newpage

\bgroup
{
\def\arraystretch{1.5}
\begin{longtable}{|l l l|}
\caption{Constants in the proof of Lemma \ref{lem:iterative lemma}, steps 1 and 2.
} \label{tab:constants:all:2} \\
\hline
Object & Constant & Label \\
\hline
\hline
$\averxiN$ 
&$\CLaverxiN = \sigmaTinv, \quad \CNaverxiN=  \left(\frac{\delta}{\rho}\right)^{\tau}\!\sigmaTinv \CT \:\CR(\rho) ,\quad  \CaverxiN= \CLaverxiN + \CNaverxiN$
&\eqref{CaverxiN} 
\\
\hline
$\xiN$ 
&$\CLxiN = \CLaverxiN, \quad \CNxiN=  \CNaverxiN + \CR(\delta) ,\quad  \CxiN= \CLxiN + \CNxiN$
&\eqref{CxiN} 
\\
\hline
$\Ki-\K$ 
&$\CLDeltaKi = \CN \CLxiN,\quad \CNDeltaKi= \CN\CNxiN, \quad \CDeltaKi= \CLDeltaKi + \CNDeltaKi$ 
& \eqref{CDeltaKi}
\\
\hline
$\DKi-\DK$ 
&$\CLDeltaDKi = d \CN \CLaverxiN, \quad 
\CNDeltaDKi= d \CN (\CNaverxiN + \CRd(\delta)), \quad \CDeltaDKi = \CLDeltaDKi + \CNDeltaDKi$
& \eqref{CDeltaDKi}
\\
\hline
$\DKiT-\DKT$
&
$\CLDeltaDKiT = n \CNT \CLaverxiN,$
&
\\
&
$\CNDeltaDKiT= n \CNT (\CNaverxiN+\CRd(\delta)), \quad 
\CDeltaDKiT = \CLDeltaDKiT + \CNDeltaDKiT$
& \eqref{CDeltaDKiT}
\\
\hline
$\Li-\L$ 
&$\CLDeltaLi= \CLDeltaDKi  + \delta \cteDXp \CLDeltaKi,\quad \CNDeltaLi= \CNDeltaDKi  + \delta \cteDXp \CNDeltaKi ,\quad
\CDeltaLi = \CLDeltaLi + \CNDeltaLi$ 
&
\eqref{CDeltaLi}
\\
\hline
$\LiT-\LT$ 
& $\CLDeltaLiT= \max\{\CLDeltaDKiT,\delta \cteDXpT \CLDeltaKi\},$ 
&
\\
& $\CNDeltaLiT= \max\{\CNDeltaDKiT, \delta \cteDXpT \CNDeltaKi\},\quad
\CDeltaLiT = \CLDeltaLiT + \CNDeltaLiT$
& 
\eqref{CDeltaLiT}
\\
\hline
\hline
$\Ei$ 
&$\CLEi= \CLE + \cteDXH \CLDeltaKi + \CLieN \CLxiN,$ 
& 
\\
&  $\CNEi= \gamma \delta^\tau(\CNE + \sigmaN) + 
\cteDXH \CNDeltaKi+ \CLieN \CNxiN, \quad
\CEi= \CLEi + \CNEi$ 
& \eqref{CEi}
\\
\hline
$\etaiN$ 
& $\QCetaiN= (\CDeltaLiT \cteOmega + \delta \CLT \cteDOmega \CDeltaKi) \CEi +
\CEredNN \CxiN + \delta \CLT \cteOmega \frac{1}{2} \cteDDXH (\CDeltaKi)^2$ 
& \eqref{QCetaiN}
\\
\hline 
\end{longtable}
}

\egroup

\newpage

\bgroup
{
\def\arraystretch{1.5}
\begin{longtable}{|l l l|}
\caption{Constants in the proof of Lemma~\ref{lem:iterative lemma}, steps 3 and 4.
} \label{tab:constants:all:3} \\
\hline
Object & Constant & Label \\
\hline
\hline
$\xiL$ 
&
$\CLxiL = \CR(\delta) (1+\CT\CLxiN), \quad \CNxiL= \CR(\delta) \CT \CNxiN, \quad \CxiL= \CLxiL + \CNxiL$
& 
\eqref{CxiL} 
\\
\hline
$\Kn-\K$ 
&
$\CLDeltaKn =  \cteDPhi \CL \CLxiL+\gamma\delta^\tau \CLDeltaKi,$ 
&
\\
& 
$\CNDeltaKn= \cteDPhi \CL \CNxiL + \gamma\delta^\tau \CNDeltaKi, \quad \CDeltaKn= \CLDeltaKn + \CNDeltaKn$ 
& 
\eqref{CDeltaKn}
\\
\hline
$\DKn-\DK$ 
&
$\CLDeltaDKn= d \cteDPhi \CL \CLxiL + \gamma\delta^\tau \CLDeltaDKi,\quad \CNDeltaDKn= d \cteDPhi \CL \CNxiL + \gamma\delta^\tau \CNDeltaDKi,$
&
\\
&
$\CDeltaDKn = \CLDeltaDKn + \CNDeltaDKn$
&
\eqref{CLDeltaKn}
\\
\hline
$\DKnT-\DKT$ &
$\CLDeltaDKnT= 2n \cteDPhi \CL \CLxiL + \gamma\delta^\tau \CLDeltaDKiT,\quad \CNDeltaDKnT= 2n \cteDPhi \CL \CNxiL + \gamma\delta^\tau \CNDeltaDKiT,$
&
\\
&
$\CDeltaDKnT = \CLDeltaDKnT + \CNDeltaDKnT$
&
\eqref{CLDeltaKnT}
\\
\hline
$\Ln-\L$ &
$\CLDeltaLn= \CLDeltaDKn  + \delta \cteDXp \CLDeltaKn,$ 
& 
\\
&
$\CNDeltaLn= \CNDeltaDKn  + \delta \cteDXp \CNDeltaKn ,\quad
\CDeltaLn= \CLDeltaLn + \CNDeltaLn$
& \eqref{CDeltaLn}
\\
\hline$\LnT-\LT$ &
$\CLDeltaLnT= \max\{\CLDeltaDKnT,\delta \cteDXpT \CLDeltaKn \},$
&
\\
&
$\CNDeltaLnT= \max\{\CNDeltaDKnT, \delta \cteDXpT \CNDeltaKn \},\quad
\CDeltaLnT = \CLDeltaLnT + \CNDeltaLnT$
& \eqref{CDeltaLnT}
\\
\hline
$\GLn-\GL$ &
$\CLDeltaGLn= \CLDeltaLnT \cteG \CL + \delta \CLT \cteDG \CLDeltaKn \CL+
\CLT \cteG  \CLDeltaLn,$ 
& 
\\
& 
$\CNDeltaGLn= \CNDeltaLnT \cteG \CL + \delta \CLT \cteDG \CNDeltaKn \CL+
\CLT \cteG  \CNDeltaLn,\quad
\CDeltaGLn = \CLDeltaGLn + \CNDeltaGLn$
& \eqref{CDeltaGLn}
\\
\hline
$\Bn - \B$ & 
$\CLDeltaBn = (\sigmaB)^2\: \CLDeltaGLn, \quad
\CNDeltaBn = (\sigmaB)^2\: \CNDeltaGLn, \quad
\CDeltaBn = (\sigmaB)^2\: \CDeltaGLn$
& \eqref{CDeltaBn}
\\
\hline
$\Nn-\N$ &
$\CLDeltaNn=  \delta \cteDJ \CLDeltaKn \CL \sigmaB + \cteJ \CLDeltaLn \sigmaB + 
\cteJ \CL  \CLDeltaBn,$ 
& 
\\
& 
$\CNDeltaNn=   \delta \cteDJ \CNDeltaKn \CL \sigmaB + \cteJ \CNDeltaLn \sigmaB + 
\cteJ \CL \CNDeltaBn,\quad
\CDeltaNn = \CLDeltaNn + \CNDeltaNn$
& \eqref{CDeltaNn}
\\
\hline
$\NnT-\NT$ &
$\CLDeltaNnT=  \delta \sigmaB \CL \cteDJT \CLDeltaKn  +  \sigmaB \CLDeltaLn \cteJT + 
\CLDeltaBn \CL \cteJT,$ 
&
\\
& 
$\CNDeltaNnT=  \delta \sigmaB \CL \cteDJT \CNDeltaKn  +  \sigmaB \CNDeltaLn \cteJT + 
\CNDeltaBn \CL \cteJT,\quad
\CDeltaNnT = \CLDeltaNnT + \CNDeltaNnT$
& \eqref{CDeltaNnT}
\\
\hline
$\Tn-\T$
&
$\CLDeltaTn=  \CLDeltaNnT \cteTh \CN +  \delta \CNT \cteDTh \CLDeltaKn \CN +
 \CNT \cteTh  \CLDeltaNn,$ 
&
\\
& 
$\CNDeltaTn=  \CNDeltaNnT \cteTh \CN +  \delta \CNT \cteDTh \CNDeltaKn \CN +
 \CNT \cteTh  \CNDeltaNn,\quad
\CDeltaTn = \CLDeltaTn + \CNDeltaTn$
& \eqref{CDeltaTn}
\\
\hline
$\aver{\Tn}^{\text{-}1} - \aver{T}^{\text{-}1}$ & 
$\CLDeltaiTn =(\sigmaTinv\!)^2\: \CLDeltaTn, \quad
\CNDeltaiTn = (\sigmaTinv\!)^2\: \CNDeltaTn, \quad
\CDeltaiTn = (\sigmaTinv\!)^2\: \CDeltaTn$
& \eqref{CDeltaiTn}
\\
\hline
\hline
$\LieO\xiL$
& $\CLLiexiL= 1+ \CT\CLxiN,\quad \CNLiexiL= \CT\CNxiN,\quad 
\CLiexiL= \CLLiexiL + \CNLiexiL$
&\eqref{CLiexiL}
\\
\hline
$\En$ &
$\CLEn= \cteDPhi  (\CLEi + \sigmaL \CLLiexiL), \quad \CNEn= \cteDPhi  (\CNEi + \sigmaL \CNLiexiL), \quad \CEn= \CLEn + \CNEn$
&\eqref{CEn}
\\
\hline
$\etanN$
&$\QCetanN= (1+n) (1 + \CNOmegaL \CLiexiL \fakeetaN) \QCetaiN$
&\eqref{QCetanN}
\\
\hline
$\etanLu$
& $\QCetanLu= \CEredLN \CxiN + \delta \CNT \cteOmega \frac{1}{2} \cteDDXH (\CDeltaKi)^2 
+ \gamma \delta^{\tau} \CNOmegaN$ 
&
\eqref{QCetanLu}
\\
\hline
$\etanLd$ 
&
$\QCetanLd= \CNT \cteOmega\: d \: \CEi \CxiL$ 
&
\eqref{QCetanLd}
\\
\hline
$\etanLt$
&
$\QCetanLt=  \CNT \cteOmega\: ( d\: \sigmaL \CxiL + \gamma \delta^\tau \CDeltaLi)  \CLiexiL$ 
&
\eqref{QCetanLt}
\\
\hline
$\etanLq$
&
$\QCetanLq=  \CNT \cteOmega \cteDXp \cteDPhi \CxiL (\CEi + \sigmaL \CLiexiL ) $ 
&
\eqref{QCetanLq}
\\
\hline
$\etanLc$
&
$\QCetanLc=   (\CDeltaNnT \cteOmega + \delta \CNT \cteDOmega \CDeltaKn) \CEn$ 
&
\eqref{QCetanLc}
\\
\hline
$\etanL$
&
$\QCetanL= \gamma\delta^\tau \QCetanLu + \QCetanLd +  \QCetanLt +  \delta \QCetanLq + \QCetanLc$
&
\eqref{QCetanL}
\\
\hline
\end{longtable}
}
\egroup

\newpage

\bgroup
{
\def\arraystretch{1.5}
\begin{longtable}{|l l|}
\caption{Constants in Theorem~\ref{thm:KAM}, defined in Subsection~\ref{ssec:convergence}.
} \label{tab:final constants} \\
\hline
Constant & Label \\
\hline
\hline
$a= \frac{\rho - \rho_\infty}{\rho - 3\delta -\rho_\infty}$ & \eqref{def:a}
\\
\hline 
$\QCetan = \Abs{\QCetanL, a^\tau \QCetanN}$ & \eqref{QCetan}
\\
\hline
$\CtheoDeltaK= \frac{a}{a-\fakeE} \CDeltaKn$ &\eqref{eq:final constants 1}
\\
\hline
$\CtheoDeltaDK= \frac{1}{1-\fakeE} \CDeltaDKn$ &\eqref{eq:final constants 1}
\\
\hline
$\CtheoDeltaDKT= \frac{1}{1-\fakeE} \CDeltaDKnT$ &\eqref{eq:final constants 1}
\\
\hline
 $\CtheoDeltaB= \frac{1}{1-\fakeE} \CDeltaBn$ & \eqref{eq:final constants 2}
 \\
 \hline
 $\CtheoDeltaN= \frac{1}{1-\fakeE} \CDeltaNn$ & \eqref{eq:final constants 2} 
 \\
 \hline
 $\CtheoDeltaNT= \frac{1}{1-\fakeE} \CDeltaNnT$ & \eqref{eq:final constants 2}
 \\
 \hline
 $\CtheoDeltaiT= \frac{1}{1-\fakeE} \CDeltaiTn$ & \eqref{eq:final constants 2}
 \\
 \hline
 $\CtheoE = \max \bigg\{ \gamma\delta^\tau {\frac{\max\{1,\Csym\}}{\fake}},\
 \CxiL,\
 \frac{\delta\: \CtheoDeltaK}{\dist (\K(\TT^d_\rho),\partial \cmani_0)},\ 
 \frac{\CtheoDeltaDK}{\sigmaDK - \norm{\DK}_\rho},\ 
 \frac{\CtheoDeltaDKT }{\sigmaDKT - \norm{\DKT}_\rho}, $ &
\\
$\phantom{\CtheoE = \max \bigg\{} 
 \frac{\CtheoDeltaB}{\sigmaB - \norm{\B}_\rho},\ 
\frac{\CtheoDeltaN}{\sigmaN - \norm{\N}_\rho},\ 
\frac{\CtheoDeltaNT }{\sigmaNT - \norm{\NT}_\rho},
\frac{\CtheoDeltaiT}{\sigmaTinv - \abs{\aver{\T}^{\text{-}1}}},\
\frac{1}{\fakeetaN},\ 
a^{\tau+1} \frac{\QCetan}{\kappa}   \bigg\}$
 & \eqref{def:theorem constant}
 \\
 \hline
\end{longtable}
}
\egroup

\end{document}